\documentclass[12pt]{amsart}
\usepackage{amssymb}
\usepackage{mathrsfs}
\usepackage{verbatim}
\address{Department of Mathematics, Yale University, New Haven, CT, USA}
\email{ivan.loseu@gmail.com}
\thanks{MSC 2010: 16E35, 16G99}
\newcommand{\OCat}{\mathcal{O}}
\newcommand{\Cat}{\mathcal{C}}
\newcommand{\gr}{\operatorname{gr}}
\newcommand{\K}{\mathbb{K}}

\newcommand{\Ext}{\operatorname{Ext}}
\newcommand{\End}{\operatorname{End}}
\newcommand{\Hom}{\operatorname{Hom}}
\newcommand{\Loc}{\operatorname{Loc}}
\newcommand{\A}{\mathcal{A}}
\newcommand{\Ca}{\mathsf{C}}
\newcommand{\Z}{\mathbb{Z}}
\newcommand{\C}{\mathbb{C}}
\newcommand{\R}{\mathbb{R}}

\newcommand{\WC}{\mathfrak{WC}}

\newcommand{\param}{\mathfrak{p}}
\newcommand{\paramq}{\mathfrak{P}}
\newcommand{\Str}{O}
\newcommand{\Coh}{\operatorname{Coh}}
\newcommand{\B}{\mathcal{B}}

\newcommand{\Tcal}{\mathcal{T}}
\newcommand{\g}{\mathfrak{g}}

\newcommand{\quo}{/\!/}
\newcommand{\VA}{\operatorname{V}}
\newcommand{\F}{\mathbb{F}}
\newcommand{\Q}{\mathbb{Q}}
\newcommand{\h}{\mathfrak{h}}
\newcommand{\M}{\mathcal{M}}

\newcommand{\U}{\mathcal{U}}

\newcommand{\I}{\mathcal{I}}
\newcommand{\HC}{\operatorname{HC}}

\newcommand{\Ring}{\mathsf{R}}
\newcommand{\Ifrak}{\mathfrak{I}}
\newcommand{\tf}{\mathfrak{t}}
\newcommand{\Weyl}{\mathbb{A}}

\newcommand{\Pic}{\operatorname{Pic}}
\newtheorem{Thm}{Theorem}[section]
\newtheorem{Prop}[Thm]{Proposition}
\newtheorem{Cor}[Thm]{Corollary}
\newtheorem{Lem}[Thm]{Lemma}
\theoremstyle{definition}
\newtheorem{Ex}[Thm]{Example}
\newtheorem{defi}[Thm]{Definition}
\newtheorem{Rem}[Thm]{Remark}

\numberwithin{equation}{section}
\title{On modular categories $\mathcal{O}$
for quantized symplectic resolutions}
\author{Ivan Losev}
\begin{document}
\begin{abstract}
In this paper we study highest weight and standardly stratified structures on modular analogs of categories
$\mathcal{O}$ over quantizations of symplectic resolutions. We also show how to recover the usual
categories $\mathcal{O}$ (reduced mod $p\gg 0$) from our modular categories.
More precisely, we consider a conical symplectic resolution that is
defined over a finite localization of $\Z$ and is equipped with a Hamiltonian action of
a torus $T$ that has finitely many
fixed points. We consider algebras $\A_\lambda$ of global sections of a quantization in characterstic
$p\gg 0$, where $\lambda$ is a parameter.
Then we consider  a category  $\tilde{\mathcal{O}}_\lambda$
consisting of all finite dimensional $T$-equivariant $\A_\lambda$-modules.
Under reasonable assumptions that hold in most examples of interest, we show that for $\lambda$ lying in a {\it p-alcove} $\,^p\!A$, the category
$\tilde{\mathcal{O}}_\lambda$ is highest weight (in some generalized sense).
Moreover, we show that every face of $\,^p\!A$ that survives in $\,^p\!A/p$ when $p\rightarrow \infty$
defines a standardly stratified structure on $\tilde{\OCat}_\lambda$.
We identify the associated graded categories for these standardly stratified
structures with reductions mod $p$ of the usual categories $\mathcal{O}$
in characteristic $0$. Applications of our construction include computations
of wall-crossing bijections in characteristic $p$ and the existence of gradings
on categories $\mathcal{O}$ in characteristic $0$.
\end{abstract}
\maketitle
\tableofcontents
\section{Introduction}
In this paper we study some aspects of the representation theory of quantizations of symplectic resolutions
in large positive characteristic.

\subsection{Representations in characteristic $0$}
We start by recalling some features of the representation theory in characteristic $0$, which is
more classical and more extensively studied. Let $X$
be a conical symplectic resolution  over $\C$ that is defined over a finite localization of $\Z$.
We assume that it comes with an action of a torus $T$ that is Hamiltonian, has finitely many fixed points,
and is  defined over a finite localization of $\Z$. Examples are provided by $X=T^*(G/B)$ (where $G$ is a semisimple
algebraic group, $B$ is a Borel subgroup, and $T$ is a maximal torus in $B$) or $X=\operatorname{Hilb}_n(\mathbb{A}^2)$
(and $T$ is a one-dimensional torus acting on $\mathbb{A}^2$ in a Hamiltonian way,  the action lifts
to $\operatorname{Hilb}_n(\mathbb{A}^2)$).

Consider the space $\paramq:=H^2(X,\C)$. To $\lambda\in \paramq$, we can assign a quantization $\A^\theta_\lambda$
of $X$ that is a microlocal sheaf of filtered algebras on $X$. For example, in the case of $X=T^*(G/B)$ we get
the microlocalization of the sheaf of $\lambda-\rho$-twisted differential operators on $G/B$. Taking the global
sections of the sheaf $\A^\theta_\lambda$, we get a filtered quantization of $\C[X]$ to be denoted by $\A_\lambda$.
In the case when $X=T^*(G/B)$ we get the central reductions of the universal enveloping algebra $U(\g)$,
while for the case of $X=\operatorname{Hilb}_n(\mathbb{A}^2)$ we recover the spherical subalgebras of the rational
Cherednik algebras for $S_n$ (introduced in \cite{EG}). We note that $T$ acts on $\A_\lambda$ and the action
is Hamiltonian.

Since $T$ has finitely many fixed points in $X$, it makes sense to speak about categories $\mathcal{O}$ for $\A_\lambda$,
see \cite[Sections 3.2,3.3]{BLPW}. These categories depend on a choice of a generic one-parameter subgroup $\nu:\C^\times\rightarrow T$
(where ``generic'' means that $X^{\nu(\C^\times)}=X^T$).  The category corresponding to $\nu$ is denoted by $\mathcal{O}_\nu(\A_\lambda)$. It can be defined as the category of finitely generated
$\A_\lambda$-modules that admit a $\nu(\C^\times)$-equivariant structure so that the weights of $\nu$
are bounded from above.   For $X=T^*(G/B)$ and regular $\lambda$, we recover an infinitesimal block of
the  classical BGG category  $\mathcal{O}$ (and the simples are labelled by the elements of the Weyl group),
while for $X=\operatorname{Hilb}_n(\mathbb{A}^2)$ we get the category $\mathcal{O}$ for the rational Cherednik
algebra for $S_n$ (and the simples are labelled by the partitions of $n$).
The categories $\mathcal{O}$ over quantized symplectic resolutions have been studied very extensively recently from various perspectives, see, e.g., \cite{BLPW,CWR,Perv}--
not to mention numerous papers dealing with special
cases such as BGG categories $\mathcal{O}$, categories $\mathcal{O}$ over rational Cherednik
algebras, Nakajima quiver varieties and quantized Coulomb branches.

Let us list some known results about the structure of $\mathcal{O}_\nu(\A_\lambda)$.
We can talk about {\it regular} (i.e., non-degenerate)
parameters $\lambda$. Under mild assumptions,  we will show below that there is a finite union of
affine hyperplanes (to be called {\it singular}) in $\paramq$ such that a parameter outside this finite union is
regular. For a regular parameter, the simples in $\mathcal{O}_\nu(\A_\lambda)$
are labelled by $X^T$, \cite[Section 5.1]{BLPW}. Also it is known that $\mathcal{O}_\nu(\A_\lambda)$
is a highest weight category provided $\lambda$ is regular, \cite[Section 5.2]{BLPW}.

In the examples we consider, these finite collections of hyperplanes are as follows.
For $X=T^*(G/B)$ we just take the root hyperplanes, $\ker\alpha^\vee$, where $\alpha^\vee$
runs over the set of coroots. For $X=\operatorname{Hilb}_n(\mathbb{A}^2)$ we usually consider a shifted parameter,
$c=\lambda-1/2$. Here for the hyperplanes we can take the following
collection of points: $\{-\frac{a}{b}+i| 1\leqslant a<b\leqslant n, i\in \Z, |i|\leqslant s\}$
for some non-negative integer $s$ (one can actually prove that $s=0$ but this is not going to
be relevant for our purposes).

Note that we can also consider the category of $T$-equivariant objects $\mathcal{O}_\nu^T(\A_\lambda)$
in $\mathcal{O}_\nu(\A_\lambda)$. This category is the direct sum of several copies of $\mathcal{O}_\nu(\A_\lambda)$
labelled by the characters of $T$, so we do not get anything new. The simples are labelled by
$X^T\times \mathfrak{X}(T)$, where $\mathfrak{X}(T)$ stands for the character lattice of $T$.
In this case, we can define a highest weight order on $X^T\times \mathfrak{X}(T)$
as follows: $(\mathsf{x},\kappa)<_\nu(\mathsf{x}',\kappa')$ if $\langle \nu,\kappa\rangle<
\langle\nu,\kappa'\rangle$.

\subsection{Representations in characteristic $p\gg 0$}
From now on and until the end of the introduction
suppose $\lambda$ is rational, i.e., $\lambda\in \paramq_{\Q}:=H^2(X,\Q)$.
We assume that all the objects from the previous section
($X$, the $T$-action, $\A_\lambda^\theta$) are defined over a finite localization $\Ring$
of $\Z$ depending only on the denominator of $\lambda$.
In particular, we have an $\Ring$-algebra $\A_{\lambda,\Ring}$ and the base
change $\A_{\lambda,\F}:=\F\otimes_{\Ring}\A_{\lambda,\Ring}$ for $\F:=\overline{\F}_p$
with $p\gg 0$ (we will also impose some congruence conditions on $p$, for example, when $X=\operatorname{Hilb}_n(\mathbb{A}^2)$ we will assume that $p+1$ is divisible by $n!$).
So we get a quantization $\A_{\lambda,\F}$ of $\F[X]$ for all
$\lambda\in \paramq_{\F}$. On $\A_{\lambda,\F}$ we still have a Hamiltonian action of $T_{\F}$.
In what follows we assume that the center of $\A_{\lambda,\F}$ is identified
with $\F[X_\F^{(1)}]$, where ``(1)'' denotes the Frobenius twist. This holds in most of the
examples of interest. In particular, the algebra $\A_{\lambda,\F}$ is a finite module
over its center, hence all irreducible $\A_{\lambda,\F}$-modules are finite
dimensional.

For $\lambda\in \paramq_{\F_p}(=H^2(X,\F_p))$ (this is the most interesting case),
we consider the category $\tilde{\OCat}(\A_{\lambda,\F})$
consisting of all finite dimensional $T_{\F}$-equivariant $\A_{\lambda,\F}$-modules.
Consider the {\it $p$-singular} hyperplanes in $\paramq_{\F_p}$ that are obtained as the reductions
mod $p$ of the singular hyperplanes mentioned in the previous section.
We  can still talk about the {\it regular} parameters in $\paramq_{\F_p}$:
the parameters lying outside these hyperplanes. The $p$-singular hyperplanes
split the integral lattice $\paramq_{\Z}$ into the union of regions to be called {\it p-alcoves}.
For example, for $X=T^*(G/B)$ we get the usual $p$-alcoves (the {\it $p$-singular} hyperplanes
are of the form $\langle\alpha^\vee,\bullet\rangle=pm$, where $\alpha^\vee$ is a coroot and
$m\in \Z$). In particular, we  have the so called fundamental $p$-alcove, it is consists
from the points $\lambda$ in the weight lattice such that $\langle \alpha_i^\vee,\lambda\rangle>0$
(where $\alpha_i^\vee$ are simple coroots) and $\langle\alpha_0^\vee,\lambda\rangle>-p$,
where $\alpha_0^\vee$ is the minimal coroot. In the case of $X=\operatorname{Hilb}_n(\mathbb{A}^2)$,
each $p$-alcove consists of integers $z$ satisfying $\displaystyle
\frac{(p+1)a}{b}+s<z<\frac{(p+1)a'}{b'}-s$, where $\displaystyle
\frac{a}{b}<\frac{a'}{b'}$ are rational numbers with denominators between
$2$ and $n$ such that the interval $\displaystyle (\frac{a}{b},\frac{a'}{b'})$
contains no such rational numbers (and the nonnegative integer $s$ is such as in
the previous section).

It turns out that inside of each $p$-alcove categories $\tilde{\OCat}(\A_{\lambda,\F})$
are naturally equivalent, Proposition \ref{Prop:cat_O_equiv}. The categories for two $p$-alcove that are opposite with
respect to a common face are perverse equivalent via  so called {\it wall-crossing
functors}.

\subsection{Highest weight  structures}
It turns out that the categories $\tilde{\OCat}(\A_{\lambda,\F})$ carry highest weight structures
(in some generalized sense to be explained below). We have one highest weight structure for each generic one-parameter subgroup of $T$.
Recall that a one-parameter subgroup $\nu:\mathbb{G}_m\rightarrow T$ is called {\it generic}
if $X^{\operatorname{im}\nu}=X^T$. We fix a one-parameter subgroup $\nu$ with this property.

Recall that a highest weight structure on an abelian category is a partial order on
the set of irreducible objects subject to certain upper triangularity conditions.
Usually, the definition is given in the case
when the category of interest is equivalent to that of modules over a finite
dimensional algebra. This is not the case with the category $\tilde{\OCat}(\A_{\lambda,\F})$:
the number of simples is infinite (and, moreover, the category does not have projectives,
they only exist in a suitable completion). However, we still have a partial order
on the set of simples that is defined as follows. The simples in $\tilde{\OCat}(\A_{\lambda,\F})$
are labelled by pairs $(\mathsf{x},\kappa)$, where $\mathsf{x}\in X^T$ and $\kappa\in \mathfrak{X}(T)$.
To $\mathsf{x}$ we can assign $\mathsf{c}_{\nu,\lambda}(\mathsf{x})\in \mathfrak{t}^*_{\Q}$,
the highest weight of the corresponding Verma module (over $\C$). We can now define
a partial order $\leqslant_{\lambda,\nu}$ on $X^T\times \mathfrak{X}(T)$ as follows: $(\mathsf{x},\kappa)\leqslant_{\lambda,\nu} (\mathsf{x}',\kappa')$ if
$(\mathsf{x},\kappa)=(\mathsf{x}',\kappa')$ or
$$\kappa-\mathsf{c}_{\nu,\lambda}(\mathsf{x})= \kappa'-\mathsf{c}_{\nu,\lambda}(\mathsf{x}')\operatorname{mod}p,
\langle\nu,\kappa\rangle<\langle\nu,\kappa'\rangle.$$
Now pick two integers $z_1\leqslant z_2$. We can consider the subquotient
$\tilde{\OCat}(\tilde{\A}_{\lambda,\F})_{[z_1,z_2]}$ corresponding to all
labels $(\mathsf{x},\kappa)$ with $z_1\leqslant \langle\nu,\kappa \rangle\leqslant z_2$.

Our first principal result, Theorem \ref{Thm:phw_structure}, is that this subquotient
category is highest weight with respect
to the partial order $\leqslant_{\nu,\lambda}$. Note that when $\dim T=1$, then
the number of simples in $\tilde{\OCat}(\A_{\lambda,\F})_{[z_1,z_2]}$ is finite,
while when $\dim T>1$, we need to slightly extend the definition of a highest
weight category.

It turns out that the equivalences of the categories $\tilde{\OCat}(\A_{\lambda,\F})$
for two parameters in the same alcove respect the highest weight structures in a suitable sense.
So we have a collection of highest weight categories labelled by alcoves.

\subsection{Standardly stratified structures}
We define the real alcoves (in $\paramq_\R$) as connected components of the complement in
$\paramq_\R$ to the hyperplanes that are obtained from the singular ones by translating them
by the lattice $\paramq_{\Z}$. Note that there is a natural bijection between
the real alcoves and the $p$-alcoves for $p\gg 0$  (we rescale a $p$-alcove
by $1/p$, then there is a unique real alcove whose intersection with a real $p$-alcove
is large in a suitable sense).

Now let $A$ be a real alcove, $\,^p\!A$ be the corresponding $p$-alcove. We impose a
congruence condition on $p$ requiring that $p+1$ is divisible by a certain number,
see Section \ref{SS_alcoves}.
Then pick a face $\Theta$ of $A$ (of any codimension). To $\Theta$ and a generic
one-parameter subgroup $\nu:\C^\times\rightarrow T$ (and, in fact, to some additional
data) we will assign a {\it standardly stratified structure} on $\tilde{\OCat}(\A_{\lambda,\F})$
(that comes from a pre-order on
the set of simples and generalizes the notion of a highest weight structure).
We will give a rigorous definition of this standardly stratified structure in Section \ref{S_SSS}.

Such a structure, in particular, gives rise to a filtration on $\tilde{\OCat}(\A_{\lambda,\F})$
by Serre subcategories. The main result of this paper, Theorem \ref{Thm:mod_O_st_stratif},
\begin{itemize}
\item[(i)]
checks the axioms of a
standardly stratified structure,
\item[(ii)] relates the associated graded to the category
$\mathcal{O}_\nu^T(\A_{\bar{\lambda}})$ for a suitable parameter $\bar{\lambda}$
(informally, the relation is that the associated graded is the reduction to characteristic
$p$ of $\mathcal{O}_\nu^T(\A_{\bar{\lambda}})$),
\item[(iii)] and shows that the reductions to characteristic $p$ of projective and simple
objects of $\mathcal{O}_\nu^T(\A_{\bar{\lambda}})$ become standard and proper
standard objects for the standardly stratified structure.
\end{itemize}

(ii) basically means that  we can  recover characteristic $0$ categories $\mathcal{O}$ from their
characteristic $p$ analogs, $\tilde{\OCat}$.

We also relate the wall-crossing functor corresponding to $\Theta$ to the standardly stratified
structures on the categories corresponding to $A$ (and the alcove opposite to $A$ with respect to
$\Theta$). Namely, we show that the wall-crossing functor is a partial Ringel duality functor
for the standardly stratified structures, Theorem \ref{Thm:WC_part_Ringel}.

\subsection{Applications}
There are two applications of the results outlined in the previous section that we explore
in this paper.

First, we use the characterization of the wall-crossing functors as partial Ringel dualities
to prove that the wall-crossing bijections (i.e., the bijections between the sets of simples
induced by the perverse wall-crossing functors) in characteristic $p\gg 0$ are the same as
in characteristic $0$. Computing the wall-crossing bijections is an important ingredient
in the Bezrukavnikov-Okounkov program of studying the representations of quantizations of
symplectic resolutions in characteristic $p$.

Second, under some additional assumptions on $X$ that hold in all examples we know,
the category $\tilde{\OCat}(\A_{\lambda,\F})$ comes with an additional grading
(induced by the contracting torus action on $X$). We use (ii) of the previous
section to show that the grading carries over to the categories $\OCat_\nu(\A_\lambda)$.
Conjecturally, the grading on $\tilde{\OCat}(\A_{\lambda,\F})$ is Koszul and we deduce
from here that the resulting grading on $\OCat_\nu(\A_\lambda)$ is Koszul.

{\bf Acknowledgements}. I would like to thanks Roman Bezrukavnikov and Andrei Okounkov
for stimulating discussions.
This work has been funded by the  Russian Academic Excellence Project '5-100'
This work was also partially supported by the NSF under grant DMS-1501558.

\section{Preliminaries on quantizations}
\subsection{Symplectic resolutions}
Let $Y$ be a normal Poisson affine variety over $\C$  with an action of $\C^\times$.
The action gives rise to a natural grading on  the algebra $\C[Y]$ of regular functions on $Y$:
$\C[Y]=\bigoplus_{i\in \Z}\C[Y]_i$. We assume that the grading is positive:  $\C[Y]_i=0$
 when $i<0$, and $\C[Y]_0=\C$. Further, we assume  there is a positive integer $d$
such that $\{\C[Y]_i,\C[Y]_j\}\subset \C[Y]_{i+j-d}$ for all $i,j$. By a {\it symplectic resolution
of singularities} of $Y$ one means a pair $(X,\rho)$ of
\begin{itemize}
\item a smooth symplectic algebraic variety $X$ (with form $\omega$)
\item a morphism $\rho:X\rightarrow Y$ of Poisson varieties that is a projective resolution of singularities.
\end{itemize}
Below we assume that $(X,\rho)$ is a symplectic resolution of singularities.
The $\C^\times$-action lifts from $Y$ to $X$ making $\rho$ equivariant, see
Step 1 of the proof of \cite[Proposition A.7]{Namikawa_flop}. Clearly, such a lift is unique.

\subsubsection{Structural results}
Note that  $\rho^*:\C[Y]\rightarrow \C[X]$  is an isomorphism because $Y$ is normal
and $\rho$ is proper and birational. By the Grauert-Riemenschneider theorem, we have
$H^i(X,O_X)=0$ for $i>0$.

The following is \cite[Proposition 2.5]{BPW}.

\begin{Lem}\label{Lem:odd_vanishing}
We have $H^i(X,\C)=0$ for $i$ odd.
\end{Lem}

\begin{Cor}\label{Lem:Pic_Hom}
The Chern character map $c_1:\operatorname{Pic}(X)\rightarrow H^2(X,\Z)$ induces an isomorphism
$\mathbb{Q}\otimes_{\Z}\operatorname{Pic}(X)\xrightarrow{\sim} H^2(X,\mathbb{Q})$.
\end{Cor}

\subsubsection{Deformations}\label{SSS_deformations}
We will be interested in deformations $\hat{X}/\param'$ of $X$, where $\param'$ is a finite dimensional vector space,
and $\hat{X}$ is a symplectic scheme over $\param'$ with a symplectic form $\hat{\omega}\in
\Omega^2(\hat{X}/\param')$ and also with a $\C^\times$-action subject to the following conditions.

\begin{itemize}
\item We have $X=\{0\}\times_{\param'}\hat{X}$ and $\hat{\omega}$ restricts to $\omega$;
\item the morphism $\hat{X}\rightarrow \param'$ is $\C^\times$-equivariant for the action
on $\param'$ given by $t.p=t^{-d}p$;
\item the restriction of the action to $X$ coincides with the contracting action;
\item $t.\hat{\omega}:=t^{d}\hat{\omega}$.
\end{itemize}
It turns out that there is a universal such deformation $X_{\param}$ over $\param:=H^2(X,\C)$
(any other deformation is obtained via the pull-back with respect to
a unique linear map $\param'\rightarrow \param$).  This result for formal deformations is
due to Kaledin-Verbitsky, \cite{KV}, but then it carries over to the algebraic setting
thanks to the contracting $\C^\times$-action on $X$, see  \cite{Namikawa_flop}.

Let $Y_{\param}$ stand for $\operatorname{Spec}(\C[X_{\param}])$.   The natural morphism
$\tilde{\rho}:X_{\param}\rightarrow Y_{\param}$ is projective and birational.
Since the variety $X$ has no higher cohomology, $Y_{\param}$
is a deformation of $Y$ over $\param$ meaning, in particular, that $\C[Y_{\param}]/(\param)=\C[Y]$.
For $\lambda\in \param$, let $X_\lambda,Y_\lambda$
denote the fibers of $X_{\param},Y_{\param}$ over $\lambda$. Let $\param^{sing}$ denote the
locus of $\lambda\in \param$ such that $\rho_\lambda:X_\lambda\rightarrow Y_\lambda$
is not an isomorphism and set $\param^{reg}:=\param \setminus \param^{sing}$. Then, according to Namikawa, \cite[Main theorem]{Namikawa_Poisson_bir}, $\param^{sing}$ is the union of
codimension $1$ subspaces in $\param$ to be called {\it walls} (or {\it classical walls}).
The elements of $\param^{reg}$ will be called {\it generic}.

\subsubsection{Classification of symplectic resolutions and the Weyl group}\label{SSS_sympl_resol_classif}
Let us describe the possible symplectic resolutions of $Y$ following Namikawa.

If $X^1,X^2$ are two  symplectic resolutions of $Y$, then there are
open subvarieties $\breve{X}^i\subset X^i, i=1,2,$ with $\operatorname{codim}_{X^i} X^i\setminus \breve{X}^i
\geqslant 2$ and $\breve{X}^1\xrightarrow{\sim} \breve{X}^2$,
see, e.g., \cite[Proposition 2.19]{BPW}. This allows to identify the Picard groups, $\operatorname{Pic}(X^1)=\operatorname{Pic}(X^2)$. Let $\param_{\Z}$ be the image of
$\operatorname{Pic}(X)$ in $H^2(X,\C)$.

Set $\param_{\R}:=\R\otimes_{\Z}\param_{\Z}=H^2(X,\R)$. According to \cite{Namikawa_Poisson_bir},
there is a finite group $W$ acting on $\param_{\R}$ as a reflection group, such that the
movable cone $C_{mov}$ of $X$ (that does not depend on the choice of a resolution by the previous
paragraph) is a fundamental chamber for $W$. Further, the set of classical walls introduced in
Section \ref{SSS_deformations} is $W$-stable  and the walls for $C$ are among the classical walls. So the
classical walls
further split $C_{mov}$ into chambers (to be called {\it classical chambers}) and the set of (isomorphism classes of) conical symplectic resolutions of $Y$ is in one-to-one
correspondence with the set of chambers inside $C_{mov}$.

For $\theta\in C_{mov}\setminus \param^{sing}$ we define $X^\theta$ to be the symplectic resolutions corresponding to
the classical chamber containing $\theta$ (to be called below simply the chamber of $\theta$).
For $w\in W$, we set $X^{w\theta}:=X^\theta$. But we twist the identification of $\param$ with
$H^2(X,\C)$ by $w$ so that the ample cone for $X$ in $\param$ now contains $w\theta$.

We note that $W$ acts on $Y_{\param}$ by $\C^\times$-equivariant Poisson automorphisms making $Y_{\param}\rightarrow \param$
$W$-equivariant. The quotient $Y_\param/W$ is a universal $\C^\times$-equivariant
deformation of $Y$, \cite{Namikawa2}. In particular, it is independent of the choice of
a resolution $X$. Note that
the locus in $X_{\param}$, where $X_{\param}\rightarrow Y_{\param}$ is not an isomorphism
has codimension $\geqslant 2$. Let $Y_\param^0$ denote the complement to this locus. This allows to identify  $\Pic(X_{\param})$
with $\Pic(Y_\param^0)$. On the other hand, every line bundle on $X$ uniquely deforms
to $X_{\param}$ giving an identification $\Pic(X)\cong \Pic(X_{\param})$. So for two different
resolutions $X,X'$ we get identifications
$$\Pic(X)\xrightarrow{\sim} \Pic(X_\param)\xrightarrow{\sim} \Pic(Y^0_\param)
\xrightarrow{\sim} \Pic(X'_\param)\xrightarrow{\sim} \Pic(X').$$
It is easy to see that the resulting identification $\Pic(X)\xrightarrow{\sim}
\Pic(X')$ is the same as what we had in the beginning of the section.

Note that $W$ acts on $Y_\param^0$, which gives a $W$-action on $\Pic(X^\theta)$. We identify
$\Pic(X^{\theta})$ with $\Pic(X^{w\theta})$ by  $w$. In particular, for $\chi\in \Pic(X)$
it makes sense to speak about the line bundle $\Str(\chi)$ on  $X^\theta$ for every
generic $\theta$.

We note that $W$ acts on $Y_{\param}$ by $\C^\times$-equivariant Poisson automorphisms making $Y_{\param}\rightarrow \param$
$W$-equivariant. The quotient $Y_\param/W$ is a universal $\C^\times$-equivariant
deformation of $Y$, \cite{Namikawa2}. In particular, it is independent of the choice of
a resolution $X$. Note that
the locus in $X_{\param}$, where $X_{\param}\rightarrow Y_{\param}$ is not an isomorphism
has codimension $\geqslant 2$. Let $Y_\param^0$ denote the complement to this locus. This allows to identify  $\Pic(X_{\param})$
with $\Pic(Y_\param^0)$. On the other hand, every line bundle on $X$ uniquely deforms
to $X_{\param}$ giving an identification $\Pic(X)\cong \Pic(X_{\param})$. So for two different
resolutions $X,X'$ we get identifications
$$\Pic(X)\xrightarrow{\sim} \Pic(X_\param)\xrightarrow{\sim} \Pic(Y^0_\param)
\xrightarrow{\sim} \Pic(X'_\param)\xrightarrow{\sim} \Pic(X').$$
It is easy to see that the resulting identification $\Pic(X)\xrightarrow{\sim}
\Pic(X')$ is the same as what we had before. So $W$ naturally acts on $\operatorname{Pic}(X)$.


\subsubsection{Example: cotangent bundles to flag varieties}
Let us proceed to examples.

Take a semisimple Lie algebra $\g$ over $\C$. Let $G$ be the corresponding connected semisimple group of
adjoint type and $B\subset G$ be a Borel subgroup. Consider the flag variety $\mathcal{B}=G/B$.
For $X$ we can take the cotangent bundle $T^*\mathcal{B}$, then $Y$ is the nilpotent cone  $\mathcal{N}
\subset \g\cong \g^*$
and $\rho:X\rightarrow Y$ is the Springer resolution.

We can identify $\param$ with $\mathfrak{h}^*$, where $\mathfrak{h}$ is a Cartan subalgebra
of $\mathfrak{b}$. Then $W$ is the Weyl group of $\g$ and $C_{mov}$ is the positive Weyl chamber. The locus $\param^{sing}$ coincides with the locus of non-regular elements in $\mathfrak{h}^*$ so the classical walls are
$\ker\alpha$, where $\alpha$ runs over the set of (positive) roots.  The universal deformation $X_{\param}$ is the homogeneous vector bundle $G\times_B \mathfrak{b}$ and $Y_{\param}=\g^*\times_{\h^*/W}\h^*$, the morphism $\tilde{\rho}:X_{\param}\rightarrow Y_{\param}$ is Grothendieck's simultaneous resolution.

This example can be generalized in several ways. For $X$,  we can take cotangent bundles to partial flag varieties.
More generally, we can also take (parabolic) Slodowy varieties: preimages of transverse slices to nilpotent orbits in
the cotangent bundles of (partial) flag varieties.

\subsubsection{Example: $\operatorname{Hilb}_n(\C^2)$}\label{SSS_Hilb_n}
Set $Y=(\C^{2})^n/S_n$. This variety admits a symplectic resolution,
$X=\operatorname{Hilb}_n(\C^2)$. The space $\param=H^2(X,\C)$ is one-dimensional
and $W=\Z/2\Z$, see, e.g., \cite[Section 3.6]{orbit}.

There is a classical description of $X,Y$ as Nakajima quiver varieties.
Consider the vector space $V=\operatorname{End}(\C^n)\oplus \C^{n*}$.
It comes equipped with a natural $G:=\operatorname{GL}_n(\C)$-action. The action
extends to a Hamiltonian action on $T^*V=V\oplus V^*$ with moment map $\varphi$, the comoment map
$\g\rightarrow T^*V$ is given by $\varphi^*(x):=x_V$, where $x_V$ is the vector
field on $V$ defined by $x\in \g$.

Now pick a nontrivial character $\theta$ of $G$ so that $\theta=\det^k$, where $k\neq 0$.
Consider the semistable locus $(T^*V)^{\theta-ss}\subset T^*V$. Then we can
define the Hamiltonian reduction $\varphi^{-1}(0)^{\theta-ss}/G$.
It is naturally identified with $\operatorname{Hilb}_n(\C^2)$. Furthermore,
$Y=\varphi^{-1}(0)\quo G$.

We note that the Weyl group has order 2.

This construction of a symplectic resolution can be generalized for $V$ being a framed representation  space of any quiver. In this way we get general smooth Nakajima quiver varieties, all of them are symplectic resolutions.

\subsubsection{Conical slices}\label{SSS_conical_slices}
Recall that, by a result of Kaledin, \cite{Kaledin}, $Y_{\lambda}$ has finitely many symplectic leaves
for all $\lambda\in \param$.

Take $y\in Y_{\param}$. Let $\zeta$ denote the image of $y$ in $\param$. Consider the
completion $\C[Y_\param]^{\wedge_\zeta}$. This is a Poisson algebra over the completion $\C[\param]^{\wedge_\zeta}$. Further, set $Y_{\param}^{\wedge_y}:=\operatorname{Spec}(\C[Y_{\param}^{\wedge_y}])$
and $X_\param^{\wedge_y}:=Y_{\param}^{\wedge_y}\times_{Y_\param}X_\param$.
Note that $\C[Y_\param^{\wedge_y}]$ is naturally identified with $\C[X_\param^{\wedge_y}]$.

The space $\param^*$ naturally embeds into $\C[\param]^{\wedge_\zeta}$.  We have a unique $\C^\times$-action on $\C[\param]^{\wedge_\zeta}$ by topological algebra automorphisms characterized by the property
that $\param^*$ has degree $d$.

Below we always impose the following assumption that holds in all examples we know.

\begin{itemize}
\item[($\diamondsuit$)] There is a $\C^\times$-action on $X_\param^{\wedge_y}$ that
\begin{itemize}
\item makes the morphism $X_\param^{\wedge_y}\rightarrow \param^{\wedge_\zeta}$
$\C^\times$-equivariant,
\item rescales the fiberwise symplectic form on $X_\param^{\wedge_y}$ by $t\mapsto t^d$,
\item and the induced action on $Y_\param^{\wedge_y}$ is contracting.
\end{itemize}
\end{itemize}

We would like to deduce some corollaries. We can decompose $\C[Y_{\param}]^{\wedge_y}$
into the completed tensor product of complete local Poisson algebras $\C[[T_y\mathcal{L}]]
\widehat{\otimes} \underline{A}'_{\param}$, where $\mathcal{L}$ is the symplectic leaf
through $y$ in $Y_\lambda$ and $\underline{A}'_\param$ is the Poisson centralizer of
$\C[[T_y\mathcal{L}]]$ under a suitable embedding $\C[[T_y\mathcal{L}]]
\hookrightarrow \C[Y_\param]^{\wedge_y}$. Such embeddings are conjugate under
Hamiltonian automorphisms  of $\C[Y_\param]^{\wedge_y}$ hence $\underline{A}'_\param$
is well-defined.

Under assumption ($\diamondsuit$), we can choose the decomposition
$$\C[Y_\param]^{\wedge_y}\cong \C[[T_y\mathcal{L}]]
\widehat{\otimes} \underline{A}'_{\param}$$
to be $\C^\times$-stable. Define $\underline{A}_\param$ to be
the $\C^\times$-finite part of $\underline{A}'_\param$. Then
$\underline{A}_\param$ is a finitely generated graded Poisson algebra.
We set $\underline{Y}_\param:=\operatorname{Spec}(\underline{A}_\param)$
so that $Y_{\param}^{\wedge_y}$ is identified with
$(\underline{Y}_\param\times T_y \mathcal{L})^{\wedge_0}$.

The variety $\underline{Y}$ admits a conical symplectic resolution $\underline{X}$
with a deformation $\underline{X}_\param$ over $\param$. For example, the deformation
$\underline{X}$ is constructed in the following way: we take a homogeoneous
coordinate ring for the projective (over $\underline{Y}^{\wedge_0}$) scheme $\underline{Y}^{\wedge_0}\times_{Y_\zeta^{\wedge_y}}X_\zeta$.
It carries a natural $\C^\times$-action. Then $\underline{X}$ is the Proj of
the $\C^\times$-finite part of this homogeneous coordinate ring. By the construction, we have
$X_\param^{\wedge_y}\cong (\underline{X}_\param\times T_y\mathcal{L})^{\wedge_0}$.

Now we return to the general situation. Set $\underline{\param}:=H^2(\underline{X},\C)$.
Note that $\underline{\param}=H^2(\underline{\pi}^{-1}(0),\C)$ and
$\underline{\pi}^{-1}(0)\hookrightarrow X_\lambda$. The spaces $H^2(X_\param,\C)$ and
$H^2(X,\C)=\param$ are naturally identified. This gives rise to the pullback map
$\param=H^2(X_{\param},\C)\rightarrow H^2(\underline{\pi}^{-1}(0),\C)=\underline{\param}$
to be denoted by $\eta$. We have a similar map between the Picard groups:
for this one needs to notice that $\operatorname{Pic}(\underline{X})=
\operatorname{Pic}(\underline{X}^{\wedge_0})$ thanks to the contracting $\C^\times$-action.
The maps $\eta:\operatorname{Pic}(X)\rightarrow \operatorname{Pic}(\underline{X})$
and $\eta: H^2(X,\C)\rightarrow H^2(\underline{X},\C)$ are intertwined by the 1st
Chern character map.

By the construction, we have $\underline{X}_{\param}=\param\times_{\underline{\param}} \underline{X}_{\underline{\param}}$.

\subsubsection{Rank 1 case}
Here we consider the case when $\zeta$ is generic (in the sense explained below)
in a classical wall.

\begin{Prop}\label{Prop:rank1}
Suppose that $\zeta\in \param$ is such that
the only rational hyperplane in $\param$ that contains $\zeta$ is a classical
wall, say $\Gamma$. Pick $y\in Y_\zeta$ and form the symplectic resolution  $\underline{X}$
as in Section \ref{SSS_conical_slices}. Then the kernel of $\eta: \param\rightarrow
\underline{\param}$ coincides with $\Gamma$.
\end{Prop}
\begin{proof}
We can identify $H^2_{DR}(X_\zeta)$ with $H^2_{DR}(X)$ by means of the Gauss-Manin
connection. Under this identification the class of the symplectic form on $X_\zeta$
is $\zeta$ by the construction of the period map in \cite{KV}. The pullback
of the form from $X_\zeta$ to $\underline{X}^{\wedge_0}$ is rescaled by a torus
action and hence it is zero. We conclude that $\eta(\zeta)=0$. On the other hand,
we have seen in the end of Section \ref{SSS_conical_slices} that $\eta$ is defined over $\Q$.
It follows that $\Gamma$ lies in the kernel. Since  $\eta$ takes ample line bundles to
ample line bundles it cannot be zero. Therefore the kernel coincides with $\Gamma$.
\end{proof}

\begin{Ex}\label{Ex:rank1_cotangent}
Let $X=T^*\mathcal{B}$. For $\lambda$ as above, we have $\underline{X}=T^*\mathbb{P}^1$
hence $\underline{\param}$ is 1-dimensional. If $\alpha$ is the unique simple root
vanishing at $\zeta$, then the map $\eta:\param\rightarrow \underline{\param}$
sends $\zeta$ to $\langle \zeta,\alpha^\vee\rangle$.
\end{Ex}



\subsection{Quantizations}
We will study quantizations of $Y,Y_{\param}, X,X_{\param}$. By a quantization of $Y$, we mean
\begin{itemize}
\item
a filtered
algebra $\A=\bigcup_i \A_{\leqslant i}$ such that $[\A_{\leqslant i},\A_{\leqslant j}]\subset
\A_{\leqslant i+j-d}$ for all $i,j$
\item together with an isomorphism
$\gr\A\cong \C[Y]$ of graded Poisson algebras.
\end{itemize}
Similarly, a quantization $\tilde{\A}$ of $Y_{\param}$
is a filtered $\C[\param]$-algebra (with $\param^*$ in degree $d$) together with an isomorphism
$\gr\tilde{\A}\cong \C[Y_{\param}]$ of graded Poisson $\C[\param]$-algebras. For $\lambda\in \param$,
we set $\A_\lambda:=\C_\lambda\otimes_{\C[\param]}\tilde{\A}$, this is a filtered quantization of
$\C[Y]$.  So a quantization of $\C[Y_{\param}]$ can be viewed as a family of quantizations of
$\C[Y]$ parameterized by $\param$.

By a quantization of $X=X^\theta$, we mean
\begin{itemize}
\item  a sheaf $\A^\theta$ of filtered algebras in the conical
topology on $X$ (in this topology, ``open'' means Zariski open and $\C^\times$-stable) that is complete
and separated with respect to the filtration
\item together with an isomorphism $\gr\A^\theta\cong \Str_{X^\theta}$ (of sheaves of graded Poisson algebras).
\end{itemize}

Similarly, we can talk about quantizations of $X_{\param}^\theta$.

\subsubsection{Classification of quantizations of $X$}
\cite[Theorem 1.8]{BK} (with ramifications given in \cite[Section 2.3]{quant_iso}) shows that the quantizations
$\A^\theta$ of $X$ are parameterized (up to an isomorphism) by the points in $\param=H^2(X,\C)$.
Below we will use the notation $\paramq$ for $\param$ viewed as a parameter space for quantizations.
We $\paramq$ view as an affine space, the associated vector space is
$\param$.

More precisely, there is a {\it canonical} quantization $\A_{\paramq}^\theta$ of $X_{\param}^\theta$ such that the quantization of $X^\theta$ corresponding to $\lambda\in \paramq$ is the specialization of $\A_{\paramq}^\theta$ to $\lambda$.


It follows from \cite[Section 2.3]{quant_iso} that $\A^\theta_{-\lambda}$ is isomorphic to $(\A_\lambda^\theta)^{opp}$.

\subsubsection{Algebras of global sections}
We set $\A_{\paramq}:=\Gamma(\A_{\paramq}^\theta), \A_\lambda=\Gamma(\A_\lambda^\theta)$. It follows
from \cite[Section 3.3]{BPW} that the algebras $\A_{\paramq}, \A_\lambda$ are independent of the choice of $\theta$.
From $H^i(X^\theta,\Str_{X^\theta})=0$, we deduce that the higher cohomology of both $\A_{\paramq}^\theta$
and $\A^\theta_\lambda$ vanish. In particular, $\A_\lambda$ is the specialization of $\A_{\paramq}$ at $\lambda$.
Also we see  that $\A_{\paramq}$ is a quantization of $\C[X_{\param}]$
and $\A_\lambda$ is a quantization of $\C[X]=\C[Y]$.

From the isomorphism  $\A_\lambda^\theta\cong \A_{-\lambda}^\theta$ we deduce $\A_{-\lambda}\cong \A_\lambda^{opp}$.
Also we have $\A_\lambda\cong \A_{w\lambda}$ for all $\lambda\in \param,w\in W$,
see \cite[Section 3.3]{BPW}.

\subsubsection{Example: central reductions of universal enveloping algebras}
Let us start with the example of $X=T^*\mathcal{B}$. Identify the center of the universal enveloping algebra
$U(\g)$ with $\C[\h^*]^W$ by means of the Harish-Chandra isomorphism. Then the quantization $\A_\lambda$
of $Y=\mathcal{N}$ corresponding to $\lambda\in \h^*$ is the central reduction $U(\g)\otimes_{\C[\h^*]^W}\C_\lambda$.
The  quantization $\A_\lambda^\theta$ is (the microlocalization to $T^*\mathcal{B}$ of) the sheaf $\mathcal{D}^{\lambda-\rho}_{\mathcal{B}}$ of $\lambda-\rho$-twisted differential operators on $\mathcal{B}$. Here, as usual, $\rho$ is half the sum of positive roots.
Similarly, the universal quantization $\A_{\paramq}$ is $U(\g)\otimes_{\C[\h^*]^W}\C[\h^*]$.

Let us elaborate on the construction of $\A_{\paramq}^\theta$. Consider the universal sheaf $\tilde{\mathcal{D}}_{\mathcal{B}}$ of twisted differential operators. It is constructed
as $\varpi_*(\mathcal{D}_{G/U})^T$, where $U$ stands for the unipotent radical of $B$,
$T=B/U$ is the maximal torus and $\varpi:G/U\rightarrow G/B$ is the projection.
The microlocalization of $\tilde{\mathcal{D}}_{\mathcal{B}}$ is the universal quantization of
$T^*\mathcal{B}$. It is a sheaf of $\C[\mathfrak{t}^*]$-algebras with the map
$\C[\mathfrak{t}^*]\rightarrow \Gamma(\tilde{\mathcal{D}}_{\mathcal{B}})$ coming from the
the quantum comoment map for the action of $T$ on $\mathcal{D}_{G/U}$
shifted so that the fiber of $\varpi_*(\mathcal{D}_{G/U})^T$ over $0$ is $\mathcal{D}^{-\rho}_{\mathcal{B}}$.

\subsubsection{Example: quantizations of $\operatorname{Hilb}_n(\C^2),(\C^2)^n/S_n$}
Now let us describe the quantizations of the Hilbert scheme $X:=\operatorname{Hilb}_n(\C^2)$.
Consider the algebra $D(V)$ of linear differential
operators on the space $V$ from Section \ref{SSS_Hilb_n}. The algebra $D(V)$
carries a Hamiltonian action of $G$. We choose the {\it symmetrized}
quantum comoment map $\Phi(x)$ given by the formula $\Phi(x)=\frac{1}{2}(x_V+x_{V^*})$.

We have $\paramq=\C$. The quantization $\A_{\paramq}^{\theta}$ is given by
$\pi_{G*}[D_V/D_V\Phi([\g,\g])|_{(T^*V)^{\theta-ss}}]^G$,
where we write $D_V$ for the microlocal sheaf  of algebras on $T^*V$ that is obtained
from the algebra $D(V)$ by microlocalization and $\pi_G$ is the quotient
morphism $\mu^{-1}(0)^{\theta-ss}\rightarrow \mu^{-1}(0)^{\theta-ss}/G$.
In our case,  the global sections of $\A_\lambda^\theta$ is the quantum Hamiltonian reduction on the level of algebras: $\A_\lambda=[D(V)/D(V)\{\Phi(x)-\langle\lambda,x\rangle| x\in \g\}]^G$, it is a quantization
of $Y=(\C^2)^n/S_n$, see, e.g., \cite[Lemma 4.2.4]{quant_iso}.

There is an alternative way to construct $\A_\lambda$ due to Etingof and
Ginzburg, \cite{EG}, as the so called {\it spherical rational Cherednik algebras}.
The  full  rational Cherednik algebra $H_c$, where $c\in \C$ is a parameter, is
the quotient of the smash-product  algebra
$\C\langle x_1,\ldots,x_n,y_1,\ldots,y_n\rangle\#S_n$ (where the triangular brackets indicate the
free algebra) by the relations
\begin{align*}
& [x_i,x_j]=[y_i,y_j]=0,\\
& [y_i,x_j]=c(ij), i\neq j,\\
& [y_i,x_i]=1-c\sum_{j\neq i}(ij).
\end{align*}

Let $e$ be the averaging idempotent in $\C S_n$. We can view $e$ as an element
of $H_c$. Then we can consider the unital algebra $e H_c e$, this is the so called
spherical rational Cherednik algebra. It is a quantization of $\C[Y]$.

In fact, by \cite[Theorem 1.3.1]{GG}, $\A_\lambda\cong e H_{c}e$, where $c=\lambda-1/2$.

Both constructions can be generalized. The quantum Hamiltonian reduction construction
produces quantizations of arbitrary Nakajima quiver varieties, while rational Cherednik algebras
can be generalized to so called symplectic reflection algebras that give quantizations of
the varieties $\C^{2n}/\Gamma$, where $\Gamma$ is a finite subgroup of $\operatorname{Sp}_{2n}(\C)$.

\subsection{Localization theorems}
We are interested in the categories $\A_\lambda\operatorname{-mod}$ of all finitely generated
$\A_\lambda$-modules and $\operatorname{Coh}(\A^\theta_\lambda)$ of all coherent
sheaves of $\A_\lambda^\theta$-modules, see \cite[Section 2.3]{BL} for a definition of the latter. These categories
are related via the global section functor $\Gamma_\lambda^\theta: \operatorname{Coh}(\A_\lambda^\theta)
\rightarrow \A_\lambda\operatorname{-mod}$ (mapping a coherent sheaf into its global sections)
and its left adjoint, the localization functor $\Loc_\lambda^\theta:=\A_\lambda^\theta\otimes_{\A_\lambda}\bullet$.
These functors have derived versions: $R\Gamma_\lambda^\theta:D^b(\Coh(\A_\lambda^\theta))\rightarrow
D^b(\A_\lambda\operatorname{-mod})$ and $L\Loc_\lambda^\theta: D^-(\A_\lambda\operatorname{-mod})
\rightarrow D^-(\Coh(\A_\lambda^\theta))$, the latter restricts to the bounded derived categories
whenever $\A_\lambda$ has finite homological dimension.

We say that  {\it abelian  localization holds} for $(\lambda,\theta)$ if the functors
$\Gamma_\lambda^\theta,\Loc_\lambda^\theta$ are mutually inverse equivalences
between $\Coh(\A_\lambda^\theta),\A_\lambda\operatorname{-mod}$. Similarly, we
say that {\it derived localization holds} for $(\lambda,\theta)$ if $R\Gamma_\lambda^\theta,
L\Loc_\lambda^\theta$ are mutually inverse equivalences.

\subsubsection{Examples}
In the case when $X=T^*\mathcal{B}$
one can explicitly describe the parameters where abelian and derived localization hold.

\begin{Prop}
Let $X=T^*\mathcal{B}$.  The following claims are true:
\begin{itemize}
\item Derived localization holds for $(\lambda,\theta)$ if and only if $\lambda$ is regular meaning that
$\langle\lambda,\alpha^\vee\rangle\neq 0$ for every coroot $\alpha^\vee$.
\item Let $\theta$ be in the positive Weyl chamber. Then abelian localization holds for $(\lambda,\theta)$
if and only if $\langle\lambda,\alpha^\vee\rangle\not\in \Z_{\leqslant 0}$ for all positive coroots
$\alpha^\vee$.
\end{itemize}
\end{Prop}

These are the classical derived and
abelian Beilinson-Bernstein localization theorems, \cite{BB,BB_derived}.

\subsubsection{Example: quantizations of $\operatorname{Hilb}_n(\C^2)$}
Here we have the following result, \cite{GS1,GS2,KR,BE}.

\begin{Prop}\label{Prop:Gies_loc}
Let $X=\operatorname{Hilb}_n(\C^2)$.
The following claims are true:
\begin{enumerate}
\item The homological dimension of $\A_\lambda$ is infinite, equivalently, derived localization theorem
 fails for $(\lambda,\theta\neq 0)$,  if and only if $c=\lambda-1/2$ lies in
$(-1,0)$ and is a rational number with denominator $\leqslant n$.
\item For $\theta>0$, abelian localization holds for $(\lambda,\theta)$ if and only if
$$(c+\Z_{\geqslant 0})\cap \{-\frac{a}{b}|1\leqslant a<b\leqslant n\}=\varnothing.$$
\end{enumerate}
\end{Prop}

\subsection{Harish-Chandra bimodules}
\subsubsection{Definition}
Let $R$ be a commutative Noetherian ring.
Let $\A,\A'$ be two $\Z_{\geqslant 0}$-filtered $R$-algebras such that $\gr\A,\gr\A'$ are  finitely generated
commutative $R$-algebras. We assume  $\gr\A,\gr\A'$ are isomorphic, we fix an isomorphism
and denote the resulting algebra by $A$.

By a Harish-Chandra (shortly,
HC) bimodule we mean an
$\A$-$\A'$-bimodule $\B$ that can be equipped with a {\it good filtration}, i.e.,  an $\A$-$\A'$-bimodule
filtration  bounded from below subject to the following two properties:
\begin{itemize}
\item the induced left and right $A$-actions on $\gr\B$ coincide,
\item $\gr\B$ is a finitely generated $A$-module.
\end{itemize}

By a homomorphism of Harish-Chandra bimodules we mean a  bimodule homomorphism. The category of HC $\A$-$\A'$-bimodules is denoted by $\HC(\A\text{-}\A')$. We also consider the full subcategory $D^b_{HC}(\A\text{-}\A')$ of the derived category of  $\A$-$\A'$-bimodules with Harish-Chandra homology.

By the associated variety $\VA(\mathcal{B})$ of $\mathcal{B}$ we mean the support of the finitely generated
$A$-module $\gr\B$ in $\operatorname{Spec}(A)$.

We will concentrate on the algebras $\A,\A',$ etc., of the form $\A_\lambda$ (or some other specialization
of $\A_{\paramq}$). For the time being, $R=\C$, while starting Section \ref{S_R_forms} we will also consider the situations when $R$ is a localization of $\Z$ or a positive characteristic field.

For $\chi\in \param$, let $\HC(\A_{\paramq},\chi)$ denote the category of all HC
$\A_{\paramq}$-bimodules $\B$ such that $[\alpha,m]=\langle\alpha,\chi\rangle m$.

\subsubsection{Properties}
Now let $\A,\A',\A''$ be three $R$-algebras whose associated graded are identified with the
same $R$-algebra $A$ as before. Let $\B_1\in \operatorname{HC}(\A\operatorname{-}\A'),
\B_2\in\operatorname{HC}(\A'\operatorname{-}\A'')$. Then $\operatorname{Tor}_i^{\A'}(\B_1,\B_2)\in
\operatorname{HC}(\A\operatorname{-}\A'')$. Indeed, $\operatorname{Tor}_i^{\A'}(\B_1,\B_2)$ comes with a
natural bounded from below filtration such that $\gr \operatorname{Tor}_i^{\A'}(\B_1,\B_2)$ is a subquotient
of $\operatorname{Tor}_i^{A}(\gr \B_1,\gr \B_2)$.

Similarly, if $\B_1\in \operatorname{HC}(\A\operatorname{-}\A'),
\B_2\in\operatorname{HC}(\A\operatorname{-}\A'')$, then $\operatorname{Ext}^i_{\A}(\B_1,\B_2)\in
\operatorname{HC}(\A'\operatorname{-}\A'')$ (and the similar claim holds for the Ext's in the
category of right $\A$-modules).

\subsubsection{Translation bimodules}\label{SSS_trans_bimod}
In order to approach abelian localization in the next section,
we will need translation bimodules introduced
in the present generality in \cite[Section 6.3]{BPW}.

Set $X:=X^\theta$.
Pick $\chi\in \operatorname{Pic}(X)$ (recall that the Picard groups of different symplectic resolutions
of $Y$ are naturally identified, see Section \ref{SSS_sympl_resol_classif}). Let $O(\chi)$ denote the corresponding line bundle on $X$.
Since $H^i(X,O_X)=0$ for $i>0$, the line bundle $O(\chi)$ uniquely deforms
to a right $\A_{\paramq}^\theta$-module. It was shown in \cite[Section 5.1]{BPW} that the deformation carries an $\A_{\paramq}^\theta$-bimodule
structure, where the adjoint action of $\mu\in \param^*$ is by $\langle c_1(\chi),\mu\rangle$.
We will denote the resulting bimodule by $\A_{\paramq,\chi}^\theta$.

Set $\A_{\paramq,\chi}:=\Gamma(\A_{\paramq,\chi}^\theta)$, this is an $\A_{\paramq}$-bimodule that is independent
of the choice of $\theta$ by \cite[Proposition 6.24]{BPW}. Note that $\A_{\paramq,\chi}\in \HC(\A_{\paramq},\chi)$, \cite[Proposition 6.23]{BPW}.
Set $\A_{\lambda,\chi}:=\A_{\paramq,\chi}\otimes_{\C[\paramq]}\C_\lambda$,
this is an $\A_{\lambda+\chi}$-$\A_\lambda$-bimodule (here and below we abuse the notation and write
$\A_{\lambda+\chi}$ instead of $\A_{\lambda+c_1(\chi)}$). We call $\A_{\lambda,\chi}$ a {\it translation bimodule}.

The following result was obtained in \cite[Proposition 6.26]{BPW}.

\begin{Lem}\label{Lem:translation_coincidence}
Suppose that $H^i(X^\theta, O(\chi))=0$ for all $i>0$.
Then
$\A_{\lambda,\chi}\xrightarrow{\sim} R\Gamma(\A_{\lambda,\chi}^\theta)$.
\end{Lem}

Let us recall some properties of the translation bimodules obtained in \cite{BPW}.

\begin{Lem}[Proposition 6.31 in \cite{BPW}]\label{Lem:trans_functor}
Suppose abelian localization  holds for $(\lambda+\chi,\theta)$. Then we have a functor
isomorphism
$$\Gamma_{\lambda+\chi}^\theta(\A^\theta_{\lambda,\chi}\otimes_{\A^\theta_\lambda}L\Loc_\lambda^\theta(\bullet))\cong
\A_{\lambda,\chi}\otimes^L_{\A_\lambda}\bullet.$$
\end{Lem}

\begin{Cor}\label{Cor:trans_equiv}
Suppose that abelian localization holds for $(\lambda,\theta),(\lambda+\chi,\theta)$.
Then the bimodules $\A_{\lambda,\chi}, \A_{\lambda+\chi,-\chi}$ are mutually inverse
Morita equivalences.
\end{Cor}

\subsubsection{Restriction functors}\label{SSS_HC_restr}
In \cite[Section 3.3]{Perv} we have constructed the restriction functors for HC bimodules over $\A_\lambda$ (or $\A_{\paramq}$) under the assumption that $Y$ has conical slices. We can generalize
this construction using the assumption ($\diamondsuit$) from Section \ref{SSS_conical_slices}.

We use the notation and conventions of Section \ref{SSS_conical_slices}.
Pick $y\in Y_{\param}$ and let $\lambda$ denote its image in $\param$
and $\mathcal{L}$ denote the symplectic leaf of $y$.
Consider the algebra $\underline{\A}_{\underline{\paramq}}$, the analog of $\A_{\paramq}$
for $\underline{X}$, and set
$\underline{\A}_\paramq:=\C[\paramq]\otimes_{\C[\underline{\paramq}]}\underline{\A}_{\underline{\paramq}}$.
Further let $\mathbb{A}$ denote the Weyl algebra of the symplectic vector space
$T_y\mathcal{L}$. Form the filtered algebra $\mathbb{A}\otimes \underline{\A}_{\paramq}$,
take its Rees algebra $R_\hbar(\mathbb{A}\otimes \underline{\A}_{\paramq})$
and complete it at zero. Denote this completion by $R_\hbar(\mathbb{A}\otimes \underline{\A}_{\paramq})^{\wedge_0}$.

On the other hand, we can take the Rees algebra $R_\hbar(\A_{\paramq})$ and complete it
at $y$ getting the algebra $R_\hbar(\A_{\paramq})^{\wedge_y}$. It was shown in
\cite[Section 3.3]{Perv} that we have a $\C[[\paramq,\hbar]]$-linear isomorphism
\begin{equation}\label{eq:completed_iso} R_\hbar(\A_{\paramq})^{\wedge_y}\cong R_\hbar(\mathbb{A}\otimes \underline{\A}_{\paramq})^{\wedge_0}
\end{equation}
that lifts the isomorphism $\C[Y_\param]^{\wedge_y}\cong \C[T_y \mathcal{L}\times \underline{Y}_\param]^{\wedge_0}$ from Section \ref{SSS_conical_slices}.

Similarly to \cite[Section 3.3]{Perv}, (\ref{eq:completed_iso}) give rise to the restriction
functor $\bullet_{\dagger,y}:\HC(\A_{\paramq},\chi)\rightarrow \HC(\underline{\A}_{\paramq},\chi)$.
The functor has the following properties:
\begin{itemize}
\item[(i)] It is exact and $\C[\paramq]$-linear.
\item[(ii)] We have $\B_{\dagger,y}=0$ if and only if $y\not\in \VA(\B)$.
\item[(iii)] For $\B\in \HC(\A_{\lambda+\chi}\operatorname{-}\A_\lambda)$ and $y\in Y$, we have $\dim \B_{\dagger,y}<\infty$
if and only of $\mathcal{L}$ is open in $\VA(\B)$.
\item[(iv)] The functor $\bullet_{\dagger,y}$ is monoidal.
\item[(v)] Assume that $H^i(X^\theta, \Str(\chi))=0$ for all $i>0$.
Then $(\A_{\paramq,\chi})_{\dagger,y}=\underline{\A}_{\paramq,\chi}$.
\end{itemize}

We note that the bimodule $\underline{\A}_{\paramq,\chi}$ is obtained
from $\C[\paramq]\otimes_{\C[\underline{\paramq}]}\underline{\A}_{\underline{\paramq},\eta(\chi)}$
by shifting the left $\C[\paramq]$-action.


\subsection{Translation bimodules and abelian localization}
\subsubsection{Tensor products of translation bimodules}
We want to compare the bimodules
$\A_{\lambda,\chi+\chi'}$ and $\A_{\lambda+\chi,\chi'}\otimes^L_{\A_{\lambda+\chi}}\A_{\lambda,\chi}$.
Note that we have a natural homomorphism
\begin{equation}\label{eq:transl_homom_univ}
\A_{\paramq,\chi'}\otimes^L_{\A_{\paramq}}\A_{\paramq,\chi}
\rightarrow \A_{\paramq,\chi+\chi'}.
\end{equation}
It is induced by the isomorphism
\begin{equation}\label{eq:transl_homom_local}
\A^\theta_{\paramq,\chi'}\otimes_{\A^\theta_{\paramq}}\A^\theta_{\paramq,\chi}
\rightarrow \A^\theta_{\paramq,\chi+\chi'}.
\end{equation}
(\ref{eq:transl_homom_univ}) specializes to
\begin{equation}\label{eq:transl_homom_spec}
\A_{\lambda+\chi,\chi'}\otimes^L_{\A_{\lambda+\chi}}\A_{\lambda,\chi}
\rightarrow \A_{\lambda,\chi+\chi'}.
\end{equation}

Below we will need a characterization of (\ref{eq:transl_homom_spec}).

\begin{Lem}\label{Lem:homom_spec}
Suppose that there is a resolution $X=X^\theta$ such that $H^i(X, O(\chi+\chi'))=0$
for all $i>0$. Then
(\ref{eq:transl_homom_spec}) is the unique homomorphism whose microlocalization
to $Y^{reg}$ coincides with the microlocalization of the specialization of
(\ref{eq:transl_homom_local}).
\end{Lem}
\begin{proof}
Clearly, (\ref{eq:transl_homom_spec}) has the required property. We need to show that
it characterizes it uniquely. This will follow if we check that
$\A_{\lambda,\chi+\chi'}\hookrightarrow \Gamma(\A_{\lambda,\chi+\chi'}|_{Y^{reg}})$.

 By Lemma \ref{Lem:translation_coincidence}, we have that
$\A_{\lambda,\chi+\chi'}\xrightarrow{\sim} R\Gamma(\A_{\lambda,\chi+\chi'}^\theta)$.
By the assumption of the lemma, $R\Gamma(\A_{\lambda,\chi+\chi'}^\theta)=
\Gamma(\A_{\lambda,\chi+\chi'}^\theta)$.
And it is clear that $\Gamma(\A_{\lambda,\chi+\chi'}^\theta)
\hookrightarrow \Gamma(\A_{\lambda,\chi+\chi'}^\theta|_{Y^{reg}})$.
\end{proof}


\subsubsection{Main result}
The following is the main result relating translation bimodules to abelian localization.

\begin{Prop}\label{Prop:loc_transl}
Let $\chi$  in the chamber of $\theta$ be such that
$H^i(X^\theta, \mathcal{O}(m\chi))=0$  for all $m>0$ and all $i>0$. Suppose, further,
that the bimodules $\A_{\lambda+m\chi,\chi},\A_{\lambda+(m+1)\chi,-\chi}$
are mutually inverse Morita equivalences for each $m\geqslant 0$. Then abelian localization holds for $(\lambda,\theta)$.
\end{Prop}
\begin{proof}
The proof is similar to that of \cite[Lemma 4.4]{BL}.

Thanks to \cite[Proposition 5.13]{BPW}, what we need to show is that, for all $m>1$, the homomorphism
\begin{equation}\label{eq:bimod_iso_Morita}
\A_{\lambda+(m-1)\chi,\chi}\otimes_{\A_{\lambda+(m-1)\chi}} \A_{\lambda,(m-1)\chi}\rightarrow \A_{\lambda,m\chi}
\end{equation}
is an isomorphism.  We will construct its inverse.
Consider the homomorphism
$$\A_{\lambda+m\chi,-\chi}\otimes_{\A_{\lambda+m\chi}} \A_{\lambda,m\chi}\rightarrow \A_{\lambda,(m-1)\chi}.$$
Tensoring both sides with $\A_{\lambda+(m-1)\chi,\chi}$ (that is inverse to the first factor
in the left hand side), we get
\begin{equation}\label{eq:bimod_iso_Morita1}\A_{\lambda,m\chi}\rightarrow \A_{\lambda+(m-1)\chi,\chi}\otimes_{\A_{\lambda+(m-1)\chi}}\A_{\lambda,(m-1)\chi}.
\end{equation}
We claim that (\ref{eq:bimod_iso_Morita}) and (\ref{eq:bimod_iso_Morita1})
are mutually inverse. Note that, by the construction, these homomorphisms
are mutually inverse after microlocalizing to $Y^{reg}$. It follows from
Lemma \ref{Lem:homom_spec} that they are mutually inverse.
\end{proof}

\subsubsection{Further results}
First, we have the following lemma.

\begin{Lem}\label{Lem:Morita_open}
Let $\chi\in \Pic(X)$. Then the locus of $\lambda\in \paramq$, where the homomorphisms
$$\A_{\lambda,\chi}\otimes_{\A_\lambda}\A_{\lambda+\chi,-\chi}\rightarrow \A_{\lambda+\chi},
\A_{\lambda+\chi,-\chi}\otimes_{\A_{\lambda+\chi}}\A_{\lambda,\chi}\rightarrow \A_{\lambda}$$
are isomorphisms is non-empty Zariski open  in $\paramq$.
\end{Lem}
\begin{proof}
This follows as in the proof of (2) of \cite[Proposition 4.5]{BL} using
\cite[Proposition 2.6]{CWR}.
\end{proof}

From (iv),(v) from Section \ref{SSS_HC_restr} and Proposition \ref{Prop:loc_transl}
we deduce the following.

\begin{Cor}\label{Cor:loc_slice}
If abelian localization holds for $\A_\lambda^\theta$, then it also
holds for $\underline{\A}_\lambda^\theta$.
\end{Cor}

Let us mention one more important property of translation bimodules.

\begin{Lem}\label{Lem:transl_trans}
Let $\lambda\in \paramq$ and $\chi_1,\chi_2,\chi_3$ be such that abelian localization holds
for $(\lambda,\theta),(\lambda+\chi_1,\theta)$ and for $(\lambda'-\chi_3,\theta')$ and $(\lambda',\theta')$,
where $\lambda'=\lambda+\chi$ with $\chi=\chi_1+\chi_2+\chi_3$. Then
$$\A_{\lambda,\chi}=\A_{\lambda+\chi_1+\chi_2,\chi_3}\otimes_{\A_{\lambda+\chi_1+\chi_2}}
\A_{\lambda+\chi_1,\chi_2}\otimes_{\A_{\lambda+\chi_1}}\A_{\lambda,\chi_1}.$$
\end{Lem}
\begin{proof}
We will consider the case when $\chi_3=0$ (the case when $\chi_1=0$ is similar and together they imply
the general case). We have a natural bimodule homomorphism
\begin{equation}\label{eq:bimod_homom}\A_{\lambda'-\chi_2,\chi_2}\otimes_{\A_{\lambda+\chi_1}}\A_{\lambda,\chi_1}\rightarrow \A_{\lambda,\chi}.\end{equation}
But $\A_{\lambda'-\chi_2,\chi_2}$ is a Morita equivalence bimodule  by Corollary \ref{Cor:trans_equiv}
and its inverse is $\A_{\lambda',-\chi_2}$.
From here, using Lemma \ref{Lem:homom_spec}, we see that we get an inverse of (\ref{eq:bimod_homom}) from the natural
homomorphism
$$\A_{\lambda',-\chi_2}\otimes_{\A_{\lambda'}}\A_{\lambda,\chi}\rightarrow \A_{\lambda,\chi_1}$$
by tensoring with $\A_{\lambda'-\chi_2,\chi_2}$.
\end{proof}

\subsection{Wall-crossing functors}
Here we will recall wall-crossing functors. These functors are a classical tool in the representation
theory of semisimple Lie algebras. In the generality we need they were constructed in
\cite[Section 6.4]{BPW} and further studied in \cite{Perv}, where it was shown that some
of these functors are perverse equivalences.

\subsubsection{Definition}
Let $\lambda\in \param, \chi\in \operatorname{Pic}(X)$, where $X=X^\theta$. Suppose that
abelian localization holds for  $(\lambda+\chi,\theta)$, while derived localization holds
for $(\lambda,\theta)$. Consider the functor $\WC_{\lambda+\chi\leftarrow\lambda}:=
\A_{\lambda,\chi}\otimes^L_{\A_\lambda}\bullet$, as we have seen in Section
\ref{SSS_trans_bimod}, this is a derived equivalence.
Note that if abelian equivalence holds for $(\lambda,\theta)$, then $\WC_{\lambda+\chi\leftarrow \lambda}$
is an abelian equivalence.


\subsubsection{Wall-crossing between simple algebras}
Now let us examine the behavior of $\WC_{\lambda+\chi\leftarrow \lambda}$, when the algebras
$\A_\lambda,\A_{\lambda+\chi}$ are simple.

\begin{Lem}\label{Lem:WC_simple}
If the algebras $\A_\lambda,\A_{\lambda+\chi}$ are simple, then $\A_{\lambda,\chi}\otimes^L_{\A_\lambda}\bullet$
is an abelian equivalence.
\end{Lem}
\begin{proof}
Note that we have natural homomorphisms $$\A_{\lambda,\chi}\otimes_{\A_\lambda} \A_{\lambda+\chi,-\chi}
\rightarrow \A_{\lambda+\chi} \text{ and } \A_{\lambda+\chi,-\chi}\otimes_{\A_{\lambda+\chi}}\A_{\lambda,\chi}
\rightarrow \A_\lambda.$$ Pick a generic point $y\in Y$ and consider the corresponding restriction functor
between the categories of Harish-Chandra bimodules, see Section \ref{SSS_HC_restr}.
The target category for this functor is $\mathsf{Vect}$ because of the choice of $y$.
Moreover, the images of $\A_{\lambda+\chi,-\chi},\A_{\lambda,\chi}$ are one-dimensional.
The functor sends the bimodule homomorphisms above to the identity maps. So their kernels and cokernels
vanish, equivalently they have proper associated varieties. This already shows that the homomorphisms
are surjective. But the kernels are also Harish-Chandra bimodules. Since the algebras $\A_\lambda,\A_{\lambda+\chi}$
are simple they have no Harish-Chandra bimodules with proper associated varieties, this follows,
for example, from \cite[Lemma 4.4]{B_ineq}.
\end{proof}

\subsubsection{Perversity}\label{SS_Perv}
Let us recall the general definition of a perverse equivalence due to Chuang and Rouquier.
Let $\mathcal{T}^1,\mathcal{T}^2$ be  triangulated categories equipped with  $t$-structures
that are homologically finite (each object in $\mathcal{T}^i$ has only finitely many nonzero
homology groups). Let $\mathcal{C}^1,\mathcal{C}^2$ denote the hearts of $\mathcal{T}^1,\mathcal{T}^2$, respectively.

Fix  filtrations $\mathcal{C}^i=\mathcal{C}_0^i\supset \mathcal{C}_1^i\supset\ldots \supset\mathcal{C}_k^i=\{0\}$ by Serre subcategories. We  are going to define a perverse equivalence with respect to these filtrations. By definition, this is a triangulated equivalence $\mathcal{F}:\mathcal{T}^1\xrightarrow{\sim} \mathcal{T}^2$ subject to the following conditions:
\begin{itemize}
\item[(P1)] For any $j$, the equivalence $\mathcal{F}$ restricts to an equivalence
$\mathcal{T}^1_{\mathcal{C}_j^1}\rightarrow \mathcal{T}^2_{\mathcal{C}_j^2}$, where
we write $\mathcal{T}^i_{\mathcal{C}_j^i}, i=1,2,$ for the category of all objects
in $\mathcal{T}^i$ with homology (computed with respect to the t-structures of interest)
in $\mathcal{C}_j^i$.
\item[(P2)] For $M\in \Cat_j^1$, we have $H_\ell(\mathcal{F}M)=0$ for $\ell<j$
and $H_\ell(\mathcal{F}M)\in \Cat^2_{j+1}$ for $\ell>j$.
\item[(P3)] The  functor $M\mapsto H_j(\mathcal{F}M)$ induces an equivalence $\Cat^1_j/\Cat^1_{j+1}\xrightarrow{\sim}
\Cat^2_j/\Cat^2_{j+1}$  of abelian categories.
\end{itemize}

We note that thanks for (P3), $\mathcal{F}$ induces a bijection $\varphi: \operatorname{Irr}(\Cat^1)
\xrightarrow{\sim} \operatorname{Irr}(\Cat^2)$.

It turns out that the wall-crossing functor $\WC_{\lambda+\chi\leftarrow \lambda}$ is perverse
under certain additional assumptions, \cite[Section 3.1]{Perv}. Namely, suppose that abelian localization holds
for $(\lambda,\theta), (\lambda+\chi,\theta')$ and derived localization holds
for $(\lambda,\theta')$, where $\theta,\theta'$ lie in chambers that are opposite
with respect to a common face, denote it by $\Gamma$. Then $\WC_{\lambda+\chi\leftarrow \lambda}$ is perverse, \cite[Theorem 3.1]{Perv}.
Under some more assumptions one can describe the filtrations in terms of annihilators by certain ideals.
Namely, let  $\param_0$ be the subspace of $\param$ spanned by $\Gamma$.
Let us assume that abelian localization holds for $(\hat{\lambda},\theta)$ and $(\hat{\lambda}+\chi,\theta')$
and derived localization holds for $(\hat{\lambda},\theta')$ for a Weil generic $\hat{\lambda}\in \paramq_0:=\lambda+\param_0$. Then it was shown in
\cite[Theorem 3.1]{Perv} that there are chains of ideals $\A_{\paramq_0}=\I^{k+1}_{\paramq_0}\supset \I^{k}_{\paramq_0}\supset\ldots
\supset \I^0_{\paramq_0}=\{0\}$ and $\A_{\paramq_0+\chi}=\I^{k+1}_{\paramq_0+\chi}\supset \I^{k}_{\paramq_0+\chi}\supset\ldots
\supset \I^0_{\paramq_0+\chi}=\{0\}$, where $k=\frac{1}{2}\dim X$ with the following  properties:
\begin{itemize}
\item[(a)] For a Weil generic $\lambda'\in \paramq_0$, the specialization $\I^j_{\lambda'}$ is the minimal
ideal $\I\subset \A_{\lambda'}$ with $\operatorname{GK-}\dim (\A_{\lambda'}/\I)\leqslant 2(k-j)$,
where we write $\operatorname{GK-}\dim (\A_{\lambda'}/\I)$ for the GK dimension of
$\A_{\lambda'}/\I$. The similar characterization
is true for $\I^j_{\lambda'+\chi}\subset \A_{\lambda'+\chi}$.
\item[(b)] For a Zariski generic $\lambda'\in\paramq_0$ and $\B:=\A_{\lambda',\chi}$, we have
\begin{itemize}
\item[(b1)] For all $i,j$, we have $\I_{\lambda'+\chi}^j\operatorname{Tor}^{\A_{\lambda'}}_i(\B, \A_{\lambda'}/\I_{\lambda'}^j)=0$.
\item[(b2)] For all $i,j$, we have $\operatorname{Tor}^{\A_{\lambda'+\chi}}_i(\A_{\lambda'+\chi}/\I_{\lambda'+\chi}^j,
\B)\I_{\lambda'}^j=0$.
\item[(b3)] We have   $\operatorname{Tor}^{\A_{\lambda'}}_i(\B, \A_{\lambda'}/\I_{\lambda'}^j)=0$
  for $i<j$.
\item[(b4)] We have $\I_{\lambda'+\chi}^{j-1}\operatorname{Tor}^{\A_{\lambda'}}_i(\B, \A_{\lambda'}/\I^j_{\lambda'})=
\operatorname{Tor}^{\A_{\lambda'+\chi}}_i(\A_{\lambda'+\chi}/\I_{\lambda'+\chi}^j,\B)\I_{\lambda'}^{j-1}=0$
for $i>j$.
\item[(b5)]  Set $\B_{j}:=\operatorname{Tor}^{\A_{\lambda'}}_{j}(\B, \A_{\lambda'}/\I^j_{\lambda'})$.
The kernel and the cokernel of the natural homomorphism $$\B_{j}\otimes_{\A_{\lambda'}}
\operatorname{Hom}_{\A_{\lambda'+\chi}}(\B_j, \A_{\lambda'+\chi}/\I_{\lambda'+\chi}^j)\rightarrow \A_{\lambda'+\chi}/\I_{\lambda'+\chi}^j$$
are annihilated by $\I_{\lambda'+\chi}^{j-1}$ on the left and on the right.
\item[(b6)] The kernel and the cokernel of the natural homomorphism
$$\operatorname{Hom}_{\A_{\lambda'}}(\B_j, \A_{\lambda'}/\I_{\lambda'}^j)\otimes_{\A_{\lambda'+\chi}}
\B_{j}\rightarrow \A_{\lambda'}/\I_{\lambda'}^j.$$ 
are annihilated on the left and on the right by $\I_{\lambda'}^{j-1}$.
\end{itemize}
\end{itemize}
It was shown in \cite[Section 3.4]{Perv} that (a) implies that, for a Zariski
generic $\lambda'$, the subcategories   $\A_{\lambda'}/\I_{\lambda'}^j\operatorname{-mod}
\subset \A_{\lambda'}\operatorname{-mod}$ and
$\A_{\lambda'+\chi}/\I_{\lambda'+\chi}^j\operatorname{-mod}
\subset \A_{\lambda'+\chi}\operatorname{-mod}$ are Serre.
Further, it was shown there that the functor $\A_{\lambda',\chi}\otimes^L_{\A_{\lambda'}}\bullet$
is perverse with respect to the filtrations by these subcategories once the conditions of (b) hold.
Namely, (b1) and (b2) imply (P1), then (b3) implies the first condition in (P2), and
(b4)-(b6) imply the second condition in (P2) and (P3).
Conversely, it is straightforward to see that if
$\A_{\lambda',\chi}\otimes^L_{\A_{\lambda'}}\bullet$ is perverse with respect to
the filtrations above, then (b1)-(b6) hold.

\begin{Rem}
Using techniques of Section \ref{SS_reg_param} below we can show that the locus of $\lambda'$ in (b)
is given by removing finitely many hyperplanes from $\paramq_0$.
\end{Rem}

\begin{Rem}\label{Rem:WC_perv_ideal}
Note that the filtrations making a derived equivalence perverse are uniquely recovered from that equivalence
(provided they  exist). Recall that $\lambda$ and $\lambda+\chi$ are such that abelian localization
holds for $(\lambda,\theta),(\lambda+\chi,\theta')$. In particular, $\mathfrak{WC}_{\lambda+\chi\leftarrow \lambda}$ is perverse.
In particular, thanks to Lemma \ref{Lem:transl_trans} and (b), the filtrations making $\WC_{\lambda+\chi\leftarrow \lambda}$ perverse are always given by annihilation by a suitable chain of two-sided ideals (that are obtained
from $\I^j_{\lambda'},\I^j_{\lambda'+\chi}$ via Morita equivalences
$\A_{\lambda'}\operatorname{-bimod}\xrightarrow{\sim}
\A_{\lambda}\operatorname{-bimod}$ and
$\A_{\lambda'}\operatorname{-bimod}\xrightarrow{\sim}
\A_{\lambda}\operatorname{-bimod}$ for a suitable element $\lambda'$).
In particular, these chains
of ideals in $\A_{\lambda},\A_{\lambda+\chi}$ are determined uniquely.
\end{Rem}


\section{Preliminaries on categories $\mathcal{O}$}
\subsection{Highest weight and standardly stratified structures}\label{SS_HW_SSS}
Let $\K$ be a field. Let $\Cat$ be a $\K$-linear abelian category equivalent to the category of
finite  dimensional modules over a split unital associative finite dimensional $\K$-algebra. We will write
$\Tcal$ for an indexing set for the simple objects of  $\Cat$. Let us write $L(\tau)$
for the simple object indexed by $\tau\in \Tcal$ and $P(\tau)$ for the projective cover of
$L(\tau)$.

\subsubsection{Highest weight categories}\label{SSS_HW}
The additional structure of a highest weight category on $\mathcal{C}$ is a partial  order
$\leqslant$ on $\Tcal$ that should satisfy  axioms (HW1),(HW2) to be explained below.
To state the axioms we need some notation.

The partial order $\leqslant$ defines a filtration $\Cat_{\leqslant \tau}$ on $\Cat$ by Serre subcategories indexed by $\Tcal$:  the subcategory $\Cat_{\leqslant \tau}$ is, by definition, the Serre span
of the simples $L(\tau')$ with $\tau'\leqslant \tau$. Define $\Cat_{<\tau}$ analogously and let
$\Cat_\tau$ denote the quotient $\Cat_{\leqslant \tau}/\Cat_{<\tau}$.  The first axiom of a highest weight
category is as follows:
\begin{itemize}
\item[(HW1)] $\Cat_\tau$ is equivalent to the category of finite dimensional vector spaces for all $\tau$.
\end{itemize}
To formulate the second axiom we need some  more notation.
Let $\pi_\tau$ denote the quotient
functor $\Cat_{\leqslant \tau}\twoheadrightarrow \Cat_{\tau}$. Let us write $\Delta_\tau:\Cat_\tau\rightarrow \Cat_{\leqslant \tau}$ for the left adjoint functor of $\pi_\tau$.
Let $P_\tau$ denote the indecomposable object in $\Cat_\tau$ and set $\Delta(\tau)=\Delta_\tau(P_\tau)$.
The object $\Delta(\tau)$ is called {\it standard}. The second axiom of a highest weight category is as
follows:
\begin{itemize}
\item[(HW2)] $P(\tau)$ surjects onto $\Delta(\tau)$ in such a way that the kernel is filtered by
$\Delta(\tau')$'s with $\tau'>\tau$.
\end{itemize}

We note that $\Cat^{opp}$ is also a highest weight category with respect to the same order.
The standard objects in $\Cat^{opp}$ are called costandard objects and are denoted by $\nabla(\tau)$.
The have the following property: $\dim \Ext^i_{\Cat}(\Delta(\tau),\nabla(\tau'))=\delta_{i,0}\delta_{\tau,\tau'}$.

Below we will need the following lemma.

\begin{Lem}\label{Lem:cost_inclusion}
For $M\in \Cat, \tau\in \mathcal{T}$,  the following two conditions are equivalent:
\begin{enumerate}
\item $\dim\Hom(\Delta(\tau'),M)=\delta_{\tau,\tau'}$,
\item $M\hookrightarrow \nabla(\tau)$.
\end{enumerate}
\end{Lem}
\begin{proof}
It is clear that (2)$\Rightarrow$(1). Let us prove (1)$\Rightarrow$(2).
Note that $M\in \Cat_{\leqslant \tau}$, so we can replace $\Cat$ with
$\Cat_{\leqslant \tau}$ and assume that $\tau$
is maximal. In this case, $\nabla(\tau)$ is injective, $\Delta(\tau)$
is projective. So $L(\tau)$ occurs in $M$ with multiplicity 1 that
gives a nonzero homomorphism $M\rightarrow \nabla(\tau)$. We need
to show that it is injective. Let $K$ be the kernel. Then $L(\tau)$
is not  a composition factor of $K$. It follows that $\Hom(\Delta(\tau'),K)=0$
for all $\tau'$. Hence $K=0$.
\end{proof}

\begin{Rem}\label{Rem:gen_hw}
There are various generalizations of the notion of a highest weight category.
We will need the situation when $\Cat$ has finite length property and also
enough projectives but not necessarily finitely many simples. Instead of requiring
that there are finitely many simples we can impose more general conditions:
\begin{itemize}
\item the length of all increasing chains of
elements in the poset $(\operatorname{Irr}(\Cat),\leqslant)$ is bounded from above
by some constant, say $d$,
\item and  the  quotient functors $\Cat_{\leqslant \tau}\twoheadrightarrow \Cat_\tau$
have left adjoints.
\end{itemize}
Then we say that $\Cat$ is a highest weight category if (HW1),(HW2) hold.
\end{Rem}

\subsubsection{Standardly stratified categories}\label{SSS_SSS}
Let us define standardly stratified categories following \cite[Section 2]{LW}.

Let $\Cat$ be the same as in the first paragraph of Section \ref{SS_HW_SSS}.
The additional structure of a standardly stratified category on $\mathcal{C}$ is a partial  {\it pre-order}
$\preceq$ on $\Tcal$ that should satisfy  axioms (SS1),(SS2) to be explained below. Let $\Xi$ denote
the set of equivalence classes of $\preceq$. Then, for $\xi\in \Xi$, we can define the Serre subcategories
$\Cat_{\prec\xi}\subset \Cat_{\preceq \xi}$ similarly to Section \ref{SSS_HW}, their quotient $\Cat_\xi$ together with
the quotient functor $\pi_\xi:\Cat_{\leqslant \xi}\rightarrow \Cat_\xi$, and its left adjoint $\Delta_\xi$.
Further, for $\tau\in \xi$, we write $L_\xi(\tau)$ for $\pi_\xi(L(\tau))$ and $P_\xi(\tau)$ for the projective
cover of $L_\xi(\tau)$ in $\Cat_\xi$. We define the standard (resp. proper standard) objects by
$\Delta_{\preceq}(\tau):=\Delta_\xi(P_\xi(\tau))$, resp., $\bar{\Delta}_{\preceq}(\tau):=\Delta_\xi(L_\xi(\tau))$.

The axioms of a standardly stratified category as defined in \cite{LW} are as follows.

\begin{itemize}
\item[(SS1)] The functor $\Delta_\xi$ is exact.
\item[(SS2)] The projective object $P(\tau)$ admits an epimorphism onto $\Delta_{\preceq}(\tau)$
with kernel filtered by $\Delta_{\preceq}(\tau')$'s with $\tau'\succ \tau$.
\end{itemize}

We will be mostly interested in standardly stratified structures on highest weight categories subject to suitable compatibility conditions. Namely, let
$\leqslant$ be a partial order on $\mathcal{T}$ defining a highest weight structure on $\mathcal{C}$. We say that a pre-order
$\preceq$ on $\mathcal{T}$ is compatible with $\leqslant$ if
 $\tau\prec \tau'\Rightarrow \tau\leqslant \tau'\Rightarrow \tau\preceq\tau'$.

\begin{Lem}\label{Lem:comp_SS_structure}
Let $\preceq$ be a pre-order compatible with $\leqslant$. Then it defines a standardly stratified structure
on $\Cat$ if and only if both $\pi_\xi^!$ and  $\pi^*_\xi$ (the right adjoint of $\pi_\xi$) are exact.
\end{Lem}
\begin{proof}
$\Cat$ satisfies conditions of \cite[Lemma 3.3]{CWR} because it is highest weight. Our claim
follows from \cite[Lemma 3.3]{CWR}.
\end{proof}

\begin{Rem}\label{Rem:gen_SS}
Remark \ref{Rem:gen_hw} can be generalized to the setting of standardly stratified categories in
a straightforward way.
\end{Rem}

\subsubsection{Partial Ringel dualities}
Let $\Cat$ be a highest weight category with respect to a partial order $\leqslant$ and let
$\preceq$ be a compatible partial pre-order giving a standardly stratified structure.

Now let $\Cat'$ be another category with $\operatorname{Irr}(\Cat')\xrightarrow{\sim}
\mathcal{T}$. Define a new partial order on $\mathcal{T}$ by $\tau\leqslant'\tau'$ if
\begin{itemize}
\item
either $\tau\prec\tau'$
\item or $\tau\sim\tau'$ and $\tau\geqslant \tau'$.
\end{itemize}
The axioms of a partial order
are straightforward to verify. Suppose that
$\Cat'$ is highest weight with respect to $\leqslant'$ and standardly stratified
with respect to $\preceq$.

By a partial Ringel duality functor we mean a derived equivalence $\psi:
D^b(\Cat)\rightarrow D^b(\Cat')$ that maps $\Delta(\tau)$ to
$\Delta_{\preceq}(\underline{\nabla}(\tau))$, where $\underline{\nabla}(\tau)$
is the costandard object in the subquotient category $\Cat'_\xi$. Note that
such a functor is defined uniquely up to pre- or post-composing with an abelian
equivalence that is the identity on $K_0$. Also $\Cat'$ is defined uniquely
up to an abelian equivalence which is the identity on the set of simples.

Note that if $\preceq$ is the trivial pre-order, then we recover the usual notion of
Ringel duality. The usual Ringel dual always exists. In general, the existence is unclear,
see  \cite[Section 4.5]{Perv}.

\begin{Rem}\label{Rem:partial_Ringel_gen}
We can generalize partial Ringel dualities to categories of the kind considered in
Remarks \ref{Rem:gen_hw}, \ref{Rem:gen_SS}.
\end{Rem}

\subsection{Categories $\mathcal{O}$: general setting}\label{SS_O_gen}
Let $\A$ be a Noetherian  associative algebra
equipped with a rational Hamiltonian $\C^\times$-action $\nu$.
Let $h\in \A$ be the image of $1$ under the comoment map and let $\A^i$ denote the $i$th graded component
so that $\A^i:=\{a| [h,a]=ia\}$.  We set $\A^{\geqslant 0}:=\bigoplus_{i\geqslant 0}\A^i,
\A^{> 0}:=\bigoplus_{i>0}\A^i, \Ca_\nu(\A):=\A^{\geqslant 0}/(\A^{\geqslant 0}\cap \A\A^{>0})$. Note that $\A^{\geqslant 0}\cap \A\A^{>0}$ is a two-sided ideal in $\A^{\geqslant 0}$
so $\Ca_\nu(\A)$ is an algebra. We have a natural isomorphism
\begin{equation}\label{eq:Cartan_iso}\Ca_{\nu}(\A)\cong \A^0/\bigoplus_{i>0}\A^{-i}\A^i.\end{equation}
Let us note that the definitions of $\A^{>0},\A^{\geqslant 0}, \Ca_\nu(\A)$ make sense even if
the action $\nu$ is not Hamiltonian.

By the category $\OCat_\nu(\A)$ we mean the full subcategory of the category $\A\operatorname{-mod}$
of the finitely generated $\A$-modules consisting of all modules such that $\A^{>0}$ acts locally
nilpotently. We have the {\it Verma module} functor $\Delta_\nu:\Ca_\nu(\A)\operatorname{-mod}\rightarrow\OCat_\nu(\A)$
given by $$\Delta_\nu(N):=\A\otimes_{\A^{\geqslant 0}}N=(\A/\A\A^{>0})\otimes_{\Ca_{\nu}(\A)}N.$$
Note that if $h$ acts on $N$ with eigenvalue $\alpha$ (so that $N$ is a single generalized $h$-eigenspace), then $h$ acts locally finitely on $\Delta_\nu(N)$ with eigenvalues in $\alpha+\Z_{\leqslant 0}$. The space $N$ is embedded into $\Delta_\nu(N)$ as the $\alpha$-eigenspace for
$h$. The module $\Delta_\nu(N)$ has the maximal submodule that does not intersect $N$. The quotient is denoted by $L_\nu(N)$. The map $N\mapsto L_\nu(N)$ is easily
seen to be a bijection $\operatorname{Irr}(\Ca_{\nu}(\A))\xrightarrow{\sim} \operatorname{Irr}(\mathcal{O}_\nu(\A))$.

Assume now that $\dim\Ca_\nu(\A)<\infty$. Let $N_1,\ldots,N_r$ be the full list of the simple
$\Ca_\nu(\A)$-modules. Let $\alpha_1,\ldots,\alpha_r$ be the eigenvalues of $h$ on these modules
(note that $h$ maps into the center of $\Ca_\nu(\A)$ and so acts by a scalar on every irreducible
module). We define a partial order $\leqslant$ on
the set $N_1,\ldots,N_r$ by setting $N_i\leqslant N_j$ if $\alpha_j-\alpha_i\in \Z_{>0}$ or $i=j$.
Using this partial order it is easy to show that all the generalized eigenspaces for $h$
are finite dimensional and all modules in the category $\mathcal{O}$ have finite length.
Moreover, we say that $N_i,N_j$ (or the corresponding simples $L_\nu(N_i),L_\nu(N_j)$) lie in the
same {\it $h$-block} if $\alpha_i-\alpha_j\in \Z$. Note that the simples of $\mathcal{O}$ from
different $h$-blocks lie in different blocks so we can decompose $\mathcal{O}$ according to $h$-blocks:
$\mathcal{O}_\nu(\A)=\sum_{\beta\in \C/\Z} \mathcal{O}^{\beta+\Z}_\nu(\A)$, where
$\mathcal{O}^{\beta+\Z}_\nu(\A)$ is the Serre span of the simples in the $h$-block corresponding to
$\beta$.

We will also need a graded version $\mathcal{O}^{gr}_\nu(\A)$ of the category $\mathcal{O}$.
By definition, it consists of the modules in $\mathcal{O}_\nu(\A)$ together with a grading
(compatible with the grading on $\A$ coming from $\nu$). The homomorphisms are grading preserving.
Note that the irreducibles in  $\mathcal{O}^{gr}_\nu(\A)$ are labelled by the graded irreducible
$\Ca_\nu(\A)$-modules. In the situation when $\dim \Ca_{\nu}(\A)<\infty$, the category
$\mathcal{O}^{gr}_\nu(\A)$ splits into the direct sum $\bigoplus_{\alpha\in \C}\mathcal{O}^{\alpha}_\nu(\A)$,
where  $\mathcal{O}^{\alpha}_\nu(\A)$ consists of all modules $M$ such that $h$ acts on the graded component
$M(i)$ with single eigenvalue $\alpha+i$. Note that $\mathcal{O}^\alpha_\nu(\A)\cong \mathcal{O}^{\alpha+\Z}_\nu(\A)$.

\subsection{Categories $\mathcal{O}$: setting of quantized symplectic resolutions}
Now we are going to concentrate on the case when $\A=\A_\lambda$ is the algebra of global sections of
a quantization $\A_\lambda^\theta$. Suppose that $X$ is equipped with a Hamiltonian
action of a torus $T$ such that $X^T$ is finite. Choose a one-parameter subgroup
$\nu:\C^\times\rightarrow T$ and assume that $X^{\nu(\C^\times)}$ is a finite set.
In this case, $X^{\nu(\C^\times)}=X^T$.

\begin{Lem}
The algebra $\Ca_\nu(\A_{\paramq})$ is finitely generated over $\C[\paramq]$. In particular,
$\Ca_\nu(\A_\lambda)=\Ca_\nu(\A_{\paramq})\otimes_{\C[\paramq]}\C_\lambda$ is finite dimensional.
\end{Lem}
\begin{proof}
The algebra $\Ca_\nu(\A_{\paramq})$ carries a natural filtration and $\Ca_\nu(\C[Y_{\param}])
\twoheadrightarrow \gr\Ca_\nu(\A_{\paramq})$. So it is enough to show that $\Ca_\nu(\C[Y_{\param}])$
is finitely generated over $\C[\param]$. Note that  $\Ca_\nu(\C[Y_{\param}])/(\param)=\Ca_\nu(\C[Y])$
is finitely generated as an algebra because it is a quotient of $\C[Y]^T$. Since $T$ has finitely many
fixed points in $Y$, the algebra $\C[Y]$ is finite dimensional. It follows
that  $\Ca_\nu(\C[Y_{\param}])$ is finitely generated over $\C[\param]$.
\end{proof}

For a Zariski generic $\lambda\in \paramq$ one can give a more precise description of $\Ca_\nu(\A_\lambda)$, see \cite[Proposition 5.3]{CWR} or \cite[Section 5.1]{BLPW} for a somewhat weaker
result. Namely, we can define the Cartan subquotient $\Ca_{\nu}(\A_\lambda^\theta)$
that will be a sheaf on $X^{\nu(\C^\times)}$, see \cite[Section 5.2]{CWR}. Since $\nu$
is generic and $T$ has finitely many fixed points in $X$,
$\Ca_{\nu}(\A_\lambda^\theta)$ is an algebra naturally identified with $\C[X^{\nu(\C^\times)}]$.
By the construction, there is a natural homomorphism $\Ca_{\nu}(\A_\lambda)\rightarrow
\Ca_\nu(\A_\lambda^\theta)=\C[X^{\nu(\C^\times)}]$, see \cite[Section 5.3]{CWR}.
Similarly, we have the algebra $\Ca_\nu(\A_{\paramq}^\theta)$ that is naturally isomorphic
to $\C[\paramq][X^T]$. We have a $\C[\paramq]$-linear isomorphism $\Ca_\nu(\A_{\paramq})
\rightarrow \Ca_\nu(\A^\theta_{\paramq})$.

\begin{Lem}[Proposition 5.3 in \cite{CWR}]
For a Zariski generic $\lambda\in \paramq$, the homomorphism
$\Ca_{\nu}(\A_\lambda)\rightarrow \C[X^T]$ is an isomorphism.
\end{Lem}
Below we will give a more precise description of this Zariski open subset.

Following \cite[Section 3.3]{BLPW} (see also \cite[Section 4.3]{CWR}), one can define the subcategory $\mathcal{O}_\nu(\A_\lambda^\theta)\subset
\Coh(\A_\lambda^\theta)$. By definition, it consists of the
coherent sheaves of $\A_\lambda^\theta$-modules supported on the contracting locus for $\nu$ that admit
a $\nu(\C^\times)$-equivariant structure (in a weak sense, i.e., without a quantum comoment map).
It was shown in \cite[Corollary 3.19]{BLPW} that
the functors $L_i\Loc_\lambda^\theta, R^i\Gamma_\lambda^\theta$ map between categories $\mathcal{O}$. In particular, if abelian localization holds for
$(\lambda,\theta)$, then the categories $\mathcal{O}_\nu(\A_\lambda), \mathcal{O}_\nu(\A_\lambda^\theta)$
are equivalent (via $\Loc$ and $\Gamma$).

Note that similarly to Section \ref{SS_O_gen} we can consider the graded versions $\mathcal{O}_\nu^{gr}(\A_\lambda),
\mathcal{O}^{gr}_\nu(\A_\lambda^\theta)$.

Also we can consider the $T$-equivariant versions
$\mathcal{O}^T_{\nu}(\A_{\lambda}), \mathcal{O}^T_\nu(\A_\lambda^\theta)$ of all $T$-equivariant objects
in $\mathcal{O}_\nu(\A_\lambda), \mathcal{O}_\nu(\A_\lambda)$ (when $T$ is one-dimensional, we recover
the categories $\mathcal{O}_\nu^{gr}(\A_\lambda)$, $\mathcal{O}_\nu^{gr}(\A_\lambda^\theta)$).

\subsection{Verma modules and their duals}
\subsubsection{Objects $\Delta_{\nu,\lambda}(\mathsf{x})$}
Let $\mathsf{x}\in X^T$. We can view $\C_\mathsf{x}$, the simple $\C[X^T]$-module corresponding to $\mathsf{x}$, as a module over $\Ca_\nu(\A_\lambda)$ via the homomorphism
$\Ca_\nu(\A_\lambda)\rightarrow \C[X^T]$. We denote the module $\Delta_\nu(\C_\mathsf{x})$ by  $\Delta_{\nu,\lambda}(\mathsf{x})$. Note that we have
\begin{equation}\label{eq:Verma_universal}
\Hom_{\A_\lambda}(\Delta_{\nu,\lambda}(\mathsf{x}),N)=
\Hom_{\Ca_\nu(\A_\lambda)}(\C_{\mathsf{x}},N^{\A^{>0}_\lambda}).
\end{equation}
The module $\Delta_{\nu,\lambda}$ has unique simple quotient module
denoted by $L_{\nu,\lambda}(\mathsf{x})$.

The modules $\Delta_{\nu,\lambda}(\mathsf{x})$ come in a family over $\paramq$: we can use the homomorphism
$\Ca_\nu(\A_{\paramq})\rightarrow \C[\paramq][X^T]$ to form the module $\Delta_{\nu,\paramq}(\mathsf{x})$
whose specialization to $\lambda$ is $\Delta_{\nu,\lambda}(\mathsf{x})$. Similarly, we can consider a homogenized
version. Namely, form the Rees algebra $\A_{\param,\hbar}$ of $\A_{\paramq}$. We still have a homomorphism
$\Ca_\nu(\A_{\param,\hbar})\rightarrow \C[\param,\hbar][X^T]$ and so can form the
$\A_{\param,\hbar}$-module $\Delta_{\nu,\param,\hbar}(\mathsf{x})$ that specializes to $\Delta_{\nu,\paramq}(\mathsf{x})$ when $\hbar=1$.

Recall that we write $\param^{sing}$ for the locus in $\param$, where $X_\lambda$
is not affine, this is a union of hyperplanes, see Section \ref{SSS_deformations}. We write
$\param^{reg}$ for $\param\setminus \param^{sing}$.
Note that, for $\lambda\in \param^{reg}$, we get $\Delta_{\nu,\lambda,0}(\mathsf{x})=\C[(X_\lambda)^+_\mathsf{x}]$,
where $(X_\lambda)^+_\mathsf{x}:=\{x\in X_\lambda| \lim_{t\rightarrow 0}\nu(t)x=\mathsf{x}\}$. The variety $(X_\lambda)^+_{\mathsf{x}}$ is a smooth lagrangian subvariety
in $X_\lambda$ that is isomorphic to an affine space.

Now let $\kappa\in \mathfrak{X}(T)$. We can form the graded versions $\Delta_{\nu,\lambda}(\mathsf{x},\kappa),
\Delta_{\nu,\paramq}(\mathsf{x},\kappa)$ etc. by putting $\C_\mathsf{x}$ in degree $\kappa$.

\subsubsection{Flatness}
Here we are going to give a description of the locus in $\param\times \C$ where
$\Delta_{\nu,\param,\hbar}(\mathsf{x})$ is flat.

The proof the following lemma repeats that of \cite{BL}[Lemma 3.5].

\begin{Lem}\label{Lem:gen_freeness}
Let $Y$ be a closed subvariety in $\param\times \C$. Let $M$ be a finitely
generated $\C[Y]\otimes_{\C[\param,\hbar]}\A_{\param,\hbar}$-module. Then
there is an open dense affine subvariety $Y^0\subset Y$ such that
$\C[Y^0]\otimes_{\C[Y]}M$ is a free $\C[Y^0]$-module.
\end{Lem}

\begin{Cor}\label{Cor:Verma_flat_stratif}
There is a partition $\param\times \C=\bigsqcup_{i=1}^k Y^i$
into the union of locally closed affine subvarieties such that
\begin{enumerate}
\item
$\C[Y^i]\otimes_{\C[\param,\hbar]}\Delta_{\nu,\param,\hbar}(\mathsf{x})$
is free over $\C[Y^i]$,
\item and, for all $j$, the union $\bigsqcup_{i=1}^j Y^i$ is open.
\end{enumerate}
\end{Cor}

Now note that $\Delta_{\nu,\param,\hbar}(\mathsf{x})$ is $\C^\times$-equivariant
with respect to the Hamiltonian torus action and all eigenspaces are
finitely generated $\C[\param,\hbar]$-modules. Let  $\Delta_{\nu,\param,\hbar}(\mathsf{x})^j$
denote the eigenspace of weight $j$. Note that, for  $z_1,z_2\in Y^i$
we have $\dim \Delta_{\nu,\param,\hbar}(\mathsf{x})^j_{z_1}\cong
\Delta_{\nu,\param,\hbar}(\mathsf{x})^j_{z_2}$, an isomorphism of vector spaces.
We get the following corollary.

\begin{Cor}\label{Cor:Verma_flat_open}
The locus in $\param\times \C$, where the fibers of all $\C[\param,\hbar]$-modules
$\Delta_{\nu,\param,\hbar}(\mathsf{x})^j$ have the minimal possible dimension
is Zariski open in $\param\times \C$.
\end{Cor}

Denote this locus by $\tilde{Y}_{\mathsf{x}}$.

\begin{Prop}\label{Prop:Verma_flat_description}
We have $\tilde{Y}_{\mathsf{x}}\cap (\param\times \{0\})\supset\param^{reg}$.
\end{Prop}
\begin{proof}
Note that the $T$-fixed points in $X_\zeta$ are in a natural bijection with the
$T$-fixed points in $X$.
Let $\mathsf{x}'$ be the point  in the $T$-fixed locus in $X_\zeta$
corresponding to $\mathsf{x}$ for $\zeta\in \param^{reg}$.
Form the completion $\A_{\param,\hbar}^{\wedge_{\mathsf{x}'}}$. This completion
is $T\times\C^\times$-equivariantly identified with $\C[\param]\widehat{\otimes}_\C \Weyl_\hbar(T_{\mathsf{x}}X)^{\wedge_0}$, where we write
$\Weyl_\hbar(T_{\mathsf{x}}X)$ for the homogenized Weyl algebra of the
symplectic vector space $T_{\mathsf{x}}X$. From here it is easy to see that
$$\A_{\param,\hbar}^{\wedge_{\mathsf{x}'}}/\A_{\param,\hbar}^{\wedge_{\mathsf{x}'}}
\left(\A_{\param,\hbar}^{\wedge_{\mathsf{x}'}}\right)^{>0}$$
is free over $\C[\param,\hbar]$. Moreover, the $T$-eigenspaces
are free over $\C[\param,\hbar]$.

On the other hand, we claim that the $T$-finite
part, $M$, of this module is identified with
\begin{equation}\label{eq:Verma_completed1}
\C[\param,\hbar]^{\wedge_{(\zeta,0)}}\otimes_{\C[\param,\hbar]}\Delta_{\nu,\param,\hbar}(\mathsf{x}).
\end{equation}
Indeed, thanks to the direct analog of (\ref{eq:Verma_universal}), we have a homomorphism
$$\C[\param,\hbar]^{\wedge_{(\zeta,0)}}\otimes_{\C[\param,\hbar]}\Delta_{\nu,\param,\hbar}(\mathsf{x})
\rightarrow M.$$
Its specialization to the closed point is an isomorphism and the target is free over $\C[\param,\hbar]^{\wedge_{(\zeta,0)}}$. It follows that this homomorphism is an isomorphism.

So (\ref{eq:Verma_completed1}) is free over $\C[\param,\hbar]^{\wedge_{(\zeta,0)}}$.
The claim of the proposition follows.
\end{proof}

\begin{defi}\label{defi:fl_locus}
Let $(\param\times\C)^{fl}$ denote the intersection of the open subsets
$\tilde{Y}_{\mathsf{x}}$ for all $\mathsf{x}\in X^T$. Set
$\paramq^{fl}:=(\param\times \C)^{fl}\cap (\param\times \{1\})$.
\end{defi}

We write $(T_{\mathsf{x}}X)^{>0}$ for the $\nu$-contacting locus in $T_{\mathsf{x}}X$.

\begin{Cor}\label{Cor:Verma_character}
The following two conditions are equivalent:
\begin{itemize}
\item  $\lambda\in \paramq^{fl}$,
\item  for all $\mathsf{x}\in X^T$,  we have a $T$-equivariant isomorphism
$\Delta_{\nu,\lambda}(\mathsf{x})\cong \C[(T_{\mathsf{x}}X)^{>0}]$.
\end{itemize}
\end{Cor}

\subsubsection{Contravariant duality}
Let us discuss contravariant duality for categories $\mathcal{O}$. Take $M\in \mathcal{O}_\nu(\A_\lambda)$.
This module decomposes as $\bigoplus_{\alpha\in \C}M_\alpha$, where $M_\alpha$ is the generalized eigenspace
for $h=d_1\nu$ with eigenvalue $\alpha$. Recall that all $M_\alpha$ are finite dimensional,
see, e.g., \cite[Lemma 3.13]{BLPW}. Then we can consider the restricted dual $M^\vee$, this is
a right $\A_\lambda$-module. One can show that it lies in a category $\mathcal{O}$ for $\A_\lambda^{opp}$,
more precisely, in $\OCat_{-\nu}(\A_\lambda^{opp})$. Using the isomorphism $\A_\lambda^{opp}\cong
\A_{-\lambda}$, we get $M^\vee\in \mathcal{O}_{-\nu}(\A_{-\lambda})$. The quasi-inverse functor of $M\mapsto
M^\vee$ is constructed completely analogously and we again denote it by $\bullet^\vee$.

The following result was established in \cite{BLPW_appendix}.

\begin{Lem}\label{Lem:Ext_Tor}
For $M_1\in \mathcal{O}_\nu(\A_\lambda), M_2\in \mathcal{O}_{-\nu}(\A_\lambda^{opp})$, we get
$$\operatorname{Tor}_{\A_\lambda}^i(M_1,M_2)^*\cong \operatorname{Ext}_{\A_\lambda}^i(M_1,M_2^\vee).$$
\end{Lem}

Note that $\Ca_{-\nu}(\A_\lambda^{opp})$ is naturally identified with $\Ca_{\nu}(\A_\lambda)^{opp}$.
So for $N\in \Ca_{\nu}(\A_\lambda)$ let us set $\nabla_{\nu}(N)=\Delta_{-\nu}^r(N^*)^\vee$,
where we write $\Delta^r_{-\nu}$ for the Verma module functor for the category $\mathcal{O}_{-\nu}(\A_\lambda^{opp})$.
Alternatively, one can define $\nabla_\nu(N)$ as
$\Hom_{\Ca_{\nu}(\A_\lambda)}(\A_\lambda/\A_{\lambda}^{<0}\A_\lambda,N)$, see \cite[Section 4.2]{CWR}
(here we take the restricted Hom, i.e., the direct sum of Hom's from the graded components).

We write $\nabla_{\nu,\lambda}(\mathsf{x})$ for  $\Delta_{\nu,-\lambda}(\mathsf{x})^\vee$.



\subsection{Highest weight structures}
Now let us discuss highest weight structures on $\mathcal{O}_\nu(\A_\lambda),$ $\mathcal{O}_\nu(\A_\lambda^\theta)$.
\subsubsection{Ext vanishing}\label{SSS_Ext_vanish}
A proof of the following result can be found in \cite[Lemma 4.8]{Perv}.

\begin{Lem}\label{Lem:Ext_vanishing}
For a Zariski generic parameter $\lambda\in \paramq$ and $\mathsf{x},\mathsf{x}'\in X^{\nu(\C^\times)}$, we have
$\Ca_\nu(\A_\lambda)\xrightarrow{\sim}\C[X^{\nu(\C^\times)}]$ and
$$\dim \operatorname{Ext}^i_{\A_\lambda}(\Delta_\nu(\mathsf{x}),\nabla_\nu(\mathsf{x}'))=\delta_{i,0}\delta_{\mathsf{x},\mathsf{x}'}.$$
\end{Lem}

A number of interesting corollaries of this result was deduced in \cite[Section 5.2]{BLPW}.

\begin{Cor}\label{Cor:Ext_vanishing}
For $\lambda$ as in Lemma \ref{Lem:Ext_vanishing}, the natural functor
$D^b(\mathcal{O}_\nu(\A_\lambda))\rightarrow D^b(\A_\lambda\operatorname{-mod})$
is a fully faithful embedding.
\end{Cor}

Also Lemma \ref{Lem:Ext_vanishing} implies that the category $\mathcal{O}_\nu(\A_\lambda)$ is highest weight, where the order is as in Section \ref{SS_O_gen}.
In more detail, it is as follows. To $\lambda\in \paramq$ and $\mathsf{x}\in X^{\nu(\C^\times)}$
we can assign the scalar $\mathsf{c}_{\nu,\lambda}(\mathsf{x})$ equal to the image of $h\in \Ca_\nu(\A_\lambda)$ under the projection
$\Ca_{\nu}(\A_\lambda)\rightarrow \C$ of evaluation at $\mathsf{x}$. Then the order $\leqslant_\lambda$ on $X^{\nu(\C^\times)}$
is introduced as follows: $\mathsf{x}\leqslant_\lambda \mathsf{x}'$ if $\mathsf{x}=\mathsf{x}'$ or $\mathsf{c}_{\nu,\lambda}(\mathsf{x}')-
\mathsf{c}_{\nu,\lambda}(\mathsf{x})\in \Z_{>0}$.

\begin{Cor}\label{Cor:local_hw}
The category $\OCat_\nu(\A_\lambda^\theta)$ has a highest weight structure for all $\lambda$
via an equivalence $\OCat_\nu(\A^\theta_\lambda)\cong \OCat_\nu(\A_{\lambda+n\chi})$
for $\chi$ in the chamber of $\theta$ and sufficiently large $n$.
\end{Cor}

Below we will see that this highest weight structure is independent of the choice of $n$
(as long as $n$ is large enough), see \cite[Section 5.3]{BLPW}.

\subsubsection{Identification of $K_0$ with $\C X^T\times \mathfrak{X}(T)$}
Now we describe the $K_0$ groups of the categories $\mathcal{O}_{\nu}(\A^\theta)$
and their $T$-equivariant versions (under some assumptions on $\lambda$).

Suppose that  $\lambda\in \paramq^{fl}$ (see Definition \ref{defi:fl_locus}) satisfies the conditions of Lemma \ref{Lem:Ext_vanishing}. Since $\OCat_\nu(\A_\lambda)$ is a highest weight
category with standards $\Delta_{\nu,\lambda}(\mathsf{x})$,
see Section \ref{SSS_Ext_vanish}, the  classes $[\Delta_{\nu,\lambda}(\mathsf{x})]$
form a basis in $K_0(\OCat_\nu(\A_\lambda))$. This gives an identification of
$K_0(\OCat_\nu(\A_\lambda))$ with $\C X^T$. We can also identify $K_0(\OCat^T_\nu(\A_\lambda))$
with $\C X^T\times \mathfrak{X}(T)$ by sending the class of $\hat{\Delta}_{\nu,\lambda}(\mathsf{x},\kappa)$
to $(\mathsf{x},\kappa)$.

\subsection{Examples}
\subsubsection{Cotangent bundles to flag varieties}
Let us consider the special case when $X=T^*(G/B)$.
Let $T$ denote a maximal torus in $B$, it acts on $X$ in a Hamiltonian
fashion. Let $\nu:\C^\times\rightarrow T$ be a one-parameter subgroup. The set $X^{\nu(\C^\times)}$
is finite if and only if $\nu$ is regular, i.e.,  its centralizer in $G$ is $T$. In this case,
the fixed points are labelled by $W$, where $W=N_G(T)/T$ is the Weyl group of $\g$.

We recover the usual BGG category $\mathcal{O}$ (in the version, where it consists
of finitely generated $\U_\lambda$-modules with locally finite action of $\mathfrak{b}$). Conditions of
Lemma \ref{Lem:Ext_vanishing} hold under the assumption that $\lambda$ is regular. Here $\Ca_\nu(\A_\lambda)$
is naturally identified with $\C[\h]/\{f\in \C[\h]| f(w\lambda)=0, \forall w\}$.  Further, we have $\mathsf{c}_{\nu,\lambda}(\mathsf{x})=\langle w\lambda, \nu\rangle$ for $\mathsf{x}\in X^{\nu(\C^\times)}$ corresponding to $w\in W$.

\subsubsection{Example: rational Cherednik algebras of type A}
Consider the case when $X=\operatorname{Hilb}_n(\C^2)$. Let $\h$ denote
the span of $y_1,\ldots,y_n$ and $\h^*$ denote the span of $x_1,\ldots,x_n$. Then
$S(\h^*),S(\h)$ are included into $H_c$ as subalgebras. Moreover,
$\C[Y]=S(\h\oplus \h^*)^{S_n}$.

The category $\mathcal{O}_c$ was introduced in \cite{GGOR} as the category of
all $H_c$-modules that are finitely generated over $S(\h^*)$ and have locally
nilpotent action of $\h$.  We can introduce the Verma modules $\Delta_c(\mu)=H_c
\otimes_{S(\h)\#S_n}\mu$, where $\mu$ runs over the irreducible $S_n$-modules
(that are naturally labelled by partitions of $n$).
For  $c\not\in \{-\frac{a}{b}| 1\leqslant a<b\leqslant n\}$, the functor $M\mapsto eM:
H_c\operatorname{-mod}\rightarrow e H_c e\operatorname{-mod}$
is an equivalence of categories, see, e.g., \cite[Corollary 4.2]{BE}.

On $X$ we have a one-parameter Hamiltonian torus acting. We  choose $\nu$ so that $\nu(t).x_i=t^{-1}x_i,\nu(t).y_i=ty_i$.
The fixed points in $X$ are also naturally parameterized by the partitions of $n$.

The simples in $\mathcal{O}_\nu(\A_\lambda)$ (where $\lambda=c+\frac{1}{2}$)
are the  modules of the form $eL_{c}(\mu)$. Let $\mathsf{x}_\mu$ denote the fixed point in $X^T$ such that
$eL_c(\mu)$ corresponds to $\mathsf{x}_\mu$. In fact, the labeling $\mu\mapsto x_\mu$ is the standard
one, but we will not need this.

Set $\mathsf{x}:=\mathsf{x}_\mu$. Let us compute the number $\mathsf{c}_{\nu,\lambda}(\mathsf{x})$. We introduce two important statistics of a Young diagram $\mu$, namely, $\operatorname{cont}(\mu)$ and $ n(\mu)$. Recall that by the content of a box in a Young diagram
we mean the difference of its horizontal and vertical coordinates. The content of $\mu$ denoted by
$\operatorname{cont}(\mu)$ is the sum of contents of all boxes, for example, $\operatorname{cont}((n))=
0+1+2+\ldots+(n-1)$ (here and below by $(n)$ me mean a Young diagram with a single row of $n$ boxes). If $\mu=(\mu_1,\ldots,\mu_k)$, then we write $n(\mu):=\sum_{i=1}^k (i-1)\mu_i$.
This is the minimal degree in which $\mu$ occurs in $S(\h)$.

\begin{Lem}\label{Lem:RCA_highest_wt}
Suppose that $c$ is Zariski generic. Then, for a suitable
choice of $h$ (it is defined up to adding a constant), we have
$\mathsf{c}_{\nu,\lambda}(\mathsf{x})=c\operatorname{cont}(\mu)-n(\mu)$ (where $\mu,\mathsf{x}$
correspond to the same Young diagram).
\end{Lem}
\begin{proof}
Note that $\mathsf{c}_{\nu,\lambda}(\mathsf{x})$ is an affine function in $\lambda$,
see \cite[Section 6.2]{CWR}. So it is enough to prove the equality assuming $\lambda$ is Weil generic.
Here the simple object corresponding to $\mathsf{x}$ is  $e\Delta_c(\mu)$.
It follows that $\mathsf{c}_{\nu,\lambda}(\mathsf{x})$ is the highest weight space in $e\Delta_c(\mu)$.
The highest weight in $\Delta_c(\mu)$ is $c\operatorname{cont}(\mu)$ and the required equality follows from the fact
that $n(\mu)$ is the minimal degree of $\mu^*$ in $S(\h^*)$.
\end{proof}

\section{Very generic and regular parameters}\label{S_params}
\subsection{Generic simplicity}
\subsubsection{Essential hyperplanes}\label{SSS_essential}
Here we introduce the notion of an {\it essential hyperplane}. Pick a classical wall $\Gamma$. Fix $\zeta$ as in Proposition \ref{Prop:rank1} so that, for all $y\in Y_\lambda$,
the rank of the map $\eta:\param\rightarrow \underline{\param}$ is equal to $1$.
Let $\underline{\param}'$ denote the image of $\eta$, these spaces are identified for
all $y$. We consider the lattice $\underline{\param}'_\Z$ that is the image of
$\Pic(X)$, it is isomorphic to $\Z$.

Consider the algebra $\underline{\A}_{\underline{\paramq}'}$. Pick an ample line bundle
$\Str(\chi)$ on $X$. Let $\Str(\underline{\chi})$ denote the induced line bundle on
$\underline{X}$. Note that the map $\chi\mapsto \underline{\chi}$ gives $\eta$ after
tensoring with $\mathbb{Q}$.

Consider the bimodules $\underline{\A}_{\underline{\paramq}',\underline{\chi}},
\underline{\A}_{\underline{\paramq}',-\underline{\chi}}$. Thanks to  Lemma
\ref{Lem:Morita_open}, the locus in $\underline{\paramq}'$,
where these bimodules fail to be mutually inverse Morita equivalence bimodules
is a finite subset to be denoted by $\Sigma_y$.

We define $\Sigma_\Gamma$ as follows. We pick points $y_1,\ldots,y_k$, one in
each minimal symplectic leaf in $Y_\zeta$. Consider the union $\bigcup_i \Sigma_{y_i}$.
By definition,
$\Sigma_\Gamma$  consists of all $\underline{\lambda}$ in $\underline{\paramq}'$
that have integral (i.e., belonging to
$\underline{\param}'_{\Z}$) difference with a point in $\Sigma_y$.

For a wall $\Gamma$, let $\alpha_\Gamma$ denote an indecomposable element in $\param_\Z^*$ such that $\Gamma=\ker\alpha_\Gamma$ (this element is defined up to multiplication by $\pm 1$). We can identify
$\underline{\param}'$ with $\C$ by taking the image of $\alpha_\Gamma$ for $1$.

\begin{defi}\label{defi:essential}
By an {\it essential hyperplane} in $\paramq$ we mean a hyperplane of the form
$\tilde{\Gamma}$ given by $\langle \alpha_\Gamma,\bullet\rangle=\sigma$ for $\sigma\in \Sigma_\Gamma$
for some classical wall $\Gamma$. We say that this essential hyperplane is associated with
$\Gamma$.
\end{defi}

\begin{defi}\label{defi:relevant}
Let $\lambda\in \paramq$. We say that an essential hyperplane $\tilde{\Gamma}$ is {\it relevant}
for $\lambda$ if $(\lambda+\param_{\Z})\cap \tilde{\Gamma}\neq \varnothing$.
\end{defi}


%
%


\begin{Ex}\label{Ex:cotangent_S}
Let $X=T^*\mathcal{B}$. Then $\Sigma_\Gamma=\Z$
for all $\Gamma$.
\end{Ex}

\begin{Ex}\label{Ex:Gieseker_S}
Let $X$ be the Hilbert scheme $\operatorname{Hilb}_n(\C^2)$. Then there is only one hyperplane $\Gamma$,
which is $\Gamma=\{0\}$. Then for $\Sigma_\Gamma$ we can take $\Z+\{-\frac{a}{b}+\frac{1}{2}| 1\leqslant
a<b\leqslant n\}$. This follows, for example, from \cite[Section 4.1]{BE} combined with
the construction of translation bimodules in \cite{GS1}.
\end{Ex}

The sets $\Sigma_\Gamma$ can also be computed for quantizations of Nakajima quiver varieties, see \cite[Section 9.4]{BL},
at least, under some additional assumptions, e.g., that the underlying quiver is of finite or affine type.

\subsubsection{Very generic parameters}
\begin{defi}\label{defi:very_generic}
We say that $\lambda\in \paramq$ is {\it very generic} if it doesn't lie in any essential hyperplane.
\end{defi}

\begin{Lem}\label{Loc:localization_very_generic}
Let $\lambda$ be very generic. Then the following claims hold.
\begin{enumerate}
\item The algebra $\A_\lambda$ is simple.
\item All functors $\A_{\lambda,\chi}\otimes^L_{\A_\lambda}\bullet$ are
t-exact equivalences.
\item Abelian localization holds for $(\lambda,\theta)$
with any $\theta$.
\end{enumerate}
\end{Lem}
\begin{proof}
We just need to prove (1): (2) will follow from Lemma \ref{Lem:WC_simple}
because once $\lambda$ is very generic, then so is $\lambda+\chi$. And (3)
follows from (2) and Proposition \ref{Prop:loc_transl}. The proof of (1)
follows that of \cite[Proposition 9.6]{BL} with some modifications.
The proof is in several steps.

{\it Step 1}. Suppose that abelian localization holds for $(\lambda,\theta)$.
Our goal for time being is to show that the long wall-crossing functor
$\WC_{\lambda^-\leftarrow \lambda}$ is a t-exact equivalence. By
\cite[Theorem 6.35]{BPW}, $\WC_{\lambda^-\leftarrow \lambda}$ decomposes into
the composition of short wall-crossing functors, each crossing a single wall.
So we need to prove that each short wall-crossing functor is t-exact.

{\it Step 2}. Let $\Gamma$ be a wall for the classical  chamber $C$ containing
$\theta$ and let $\theta'$ be in the chamber opposite to $C$ with respect to
$\Gamma$. Suppose that $\chi$ is such that abelian localization
holds for  $(\lambda+\chi,\theta')$. As was discussed in Section
\ref{SS_Perv}, the functor $\WC_{\lambda+\chi\leftarrow \lambda}$
is a perverse equivalence, moreover, if $\WC_{\lambda+\chi\leftarrow \lambda}$
is not t-exact, then the support of  $\A_{\paramq_0}/I^1_{\paramq_0}$ in $\paramq_0:=\lambda+\Gamma$
is Zariski dense. On the other hand, pick $\zeta\in \Gamma$
as in Proposition \ref{Prop:rank1} and apply the restriction functor $\bullet_{\dagger}$ associated to a point
$y_i\in Y_{\zeta}$ be as in Section \ref{SSS_essential}.
Since $\lambda+\Gamma$ is not an essential hyperplane, thanks to (v)
of Section \ref{SSS_HC_restr}, the bimodules
$(\A_{\lambda,\chi})_{\dagger}, (\A_{\lambda+\chi,-\chi})_{\dagger}$
are mutually inverse Morita equivalences. It follows that
$(I^1_{\paramq_0})_{\dagger}=0$. This contradicts the claim
that the support of  $\A_{\paramq_0}/I^1_{\paramq_0}$ in $\paramq_0:=\lambda+\Gamma$
is Zariski dense. This completes the proof of the claim that $\WC_{\lambda+\chi\leftarrow\lambda}$
is a t-exact equivalence. Hence $\WC_{\lambda^-\leftarrow \lambda}$ is a t-exact equivalence.

{\it Step 3}. Thanks to the results recalled in Section \ref{SS_Perv},
$\WC_{\lambda^-\leftarrow \lambda}$ is a perverse equivalence with respect
to the filtration by the dimensions of the associated varieties of the
annihilators. Since $\A_\lambda$ is a prime Noetherian algebra, the claim
that this perverse equivalence is t-exact, means that there are no proper
two-sided ideals in $\A$, i.e., $\A$ is simple.
\end{proof}

%
%

\subsection{Abelian localization and finite homological dimension}
In what follows we will need to understand when the specialization $\A_{\lambda,z}$ for $(\lambda,z)\in \param\oplus \C$ has finite homological dimension and has some related properties. We start
with abelian localization.

\subsubsection{Abelian localization}
\begin{Lem}\label{Lem:cor_localn}
There is a finite collection of essential hyperplanes
such that for any $\lambda$ away from the union of these hyperplanes
abelian localization holds for $(\lambda,\theta)$ for
$\theta$ that depends on $\lambda$.
\end{Lem}
\begin{proof}
For each classical chamber $C$ choose  $\chi\in \operatorname{Pic}(X)$
such that $c_1(\chi)$ is  inside
$C$, equivalently, $O(\chi)$ is ample on the resolution $X$
corresponding to $C$. Recall, Lemma \ref{Lem:Morita_open}, that the locus  of all $\lambda$ such that
one of the natural homomorphisms $$\A_{\lambda,\chi}\otimes_{\A_{\lambda+\chi}}\A_{\lambda+\chi,-\chi}
\rightarrow \A_\lambda,
\A_{\lambda+\chi,-\chi}\otimes_{\A_\lambda}\A_{\lambda,\chi}
\rightarrow \A_{\lambda+\chi}$$
is not an isomorphism is Zariski closed in $\param$.
Let $Z_\chi$ denote this locus. Note that $Z_\chi$
must be contained in the union of finitely many essential hyperplanes
by Lemma \ref{Loc:localization_very_generic}.
We can choose a finite collection $\Psi$ of essential hyperplanes containing all subsets
$Z_{\chi}$ and subject to the following condition:
\begin{itemize}
\item[(*)] For any of the chosen elements $\chi$, any of the walls $\Gamma$
and any $\lambda$ such that
neither $\lambda,\lambda+\chi$ lie in a hyperplane from $\Psi$, we have
that $\lambda$ and $\lambda+\chi$ lie in the same half space with respect to
any of the hyperplanes in $\Psi$ that is relevant for $\lambda$.
\end{itemize}
Now take $\lambda$ that does not lie in any of the hyperplanes in $\Psi$.
We can find $C$ such that, for the corresponding element $\chi$, we have
that $\lambda+n\chi$ does not lie in a hyperplane from $\Psi$ for any $n\geqslant 0$.
In particular, $\lambda+n\chi\not \in Z_\chi$ for any $n\geqslant 0$.
Now we can use Proposition \ref{Prop:loc_transl} to conclude that abelian localization
holds for $(\lambda,\chi)$ finishing the proof.
\end{proof}

In what follows we enlarge $\Psi$ so that for $\lambda\not\in \Psi$ abelian localization
holds for $(\lambda,\theta)$ and $(-\lambda,\theta')$ for some $\theta,\theta'$. It follows
that abelian localization holds for the algebra $\A_\lambda\otimes \A_{\lambda}^{opp}$
and a suitable resolution of $Y\times Y$.

\subsubsection{Finite projective dimension}
As before, we consider $\paramq$ as the affine subspace $\param\times \{1\}\subset \param \times \C$.
So  any hyperplane in $\Psi$ spans a codimension $1$ subspace in $\param \times \C$.
Let $(\param\times \C)^{0}$ denote the complement to the union of all codimension $1$
subspaces that arise in this way. In particular, $(\param\times \C)^{0}\cap (\param \times\{0\})=\param \setminus \param^{reg}$. This discussion and Lemma \ref{Lem:cor_localn} have the following
corollary.

\begin{Cor}\label{Cor:fin_hom_dim}
For every $(\lambda,z)\in (\param\times \C)^{0}$ the specialization $\A_{\lambda,z}$
has finite homological dimension.
\end{Cor}

The main result of this section is as follows.

\begin{Prop}\label{Prop:fin_proj_dim}
The $\C[(\param\times \C)^0]\otimes_{\C[\param,\hbar]}(\A_{\param,\hbar}\otimes_{\C[\param,\hbar]}\A_{\param,\hbar}^{opp})$-module
$\C[(\param\times\C)^0]\otimes_{\C[\param,\hbar]}\A_{\param,\hbar}$ has finite projective dimension not exceeding $2\dim X$.
\end{Prop}
\begin{proof}
First, we claim that for each $(\lambda,z)\in (\param\times \C)^{0}$ the regular
$\A_{\lambda,z}$-bimodule has projective dimension not exceeding $2\dim X$. If $z=0$,
then $\A_{\lambda,0}=\C[X_\lambda]$ is the algebra of functions on a smooth affine
variety and our claim follows. If $z\neq 0$, we can assume it is equal to $1$.
Then abelian localization holds for the quantization $\A_\lambda\otimes_\C \A_{\lambda}^{opp}$
of $Y\times Y$
and a suitable choice of a resolution, so the projective dimension of the regular
bimodule does not exceed that of the category of coherent sheaves on the resolution
of $Y\times Y$, which equals $2\dim X$ (compare with the proof of \cite[Lemma 9.2]{BL}).

We claim that the fiberwise bounds on the projective dimensions in the previous paragraph imply the claim of the proposition.
Namely, set $k=2\dim X$, $\mathsf{A}:=\A_{\param,\hbar}\otimes_{\C[\param,\hbar]}\A_{\param,\hbar}$, and consider an $\mathsf{A}$-module resolution
$$0\rightarrow M\rightarrow P_{k-1}\rightarrow \ldots \rightarrow P_0$$
of $\A_{\param,\hbar}$, where all $P_{i}$ are projective. Since $\A_{\param,\hbar}$
is a free $\C[\param,\hbar]$-module, we see that $M$ is free over $\C[\param,\hbar]$.
Also if we specialize this resolution to any point $(\lambda,z)$ we still get a resolution.
In particular, for any $(\lambda,z)\in (\param\times \C)^0$, the specialization
$M_{\lambda,z}$ is a projective $\A_{\lambda,z}$-bimodule. We claim that this implies
that $M^0:=\C[(\param\times \C)^0]\otimes_{\C[\param,\hbar]}M$ is a projective
$\mathsf{A}^0:=\C[\param,\hbar]^0\otimes_{\C[\param,\hbar]}\mathsf{A}$-module.

Note that $\mathsf{A}^0$ is a Noetherian ring. So to show that $M^0$
is projective, it is enough to show that $M^0$
is flat, equivalently, for any right ideal $J\subset \mathsf{A}^0$, the natural
map $J\otimes_{\mathsf{A}^0}M^0\rightarrow M^0$ is injective. Assume the contrary,
let $K$ denote the kernel.

Now consider the completion $\mathsf{A}^{\wedge_{\lambda,z}}$. This algebra is flat
over $\mathsf{A}^0$. Moreover, on the category of finitely generated
$\mathsf{A}^0$-modules the functor $\mathsf{A}^{\wedge_{\lambda,z}}\otimes_{\mathsf{A}^0}\bullet$
coincides with the completion of $\C[\param,\hbar]$-modules with respect to the maximal
ideal of $(\lambda,z)$. Applying this functor to the homomorphism
$J\otimes_{\mathsf{A}^0}M^0\rightarrow M^0$ we get
\begin{equation}\label{eq:completed_homom} J^{\wedge_{\lambda,z}}\otimes_{\mathsf{A}^{\wedge_{\lambda,z}}}M^{\wedge_{\lambda,z}}\rightarrow
M^{\wedge_{\lambda,z}}.\end{equation}
The kernel of this homomorphism is $K^{\wedge_{\lambda,z}}$.
Note that $M^{\wedge_{\lambda,z}}$ is a formal deformation of $M_{\lambda,z}$.
Since $M_{\lambda,z}$ is flat over $\mathsf{A}_{\lambda,z}$, we conclude that
$M^{\wedge_{\lambda,z}}$ is flat over $\mathsf{A}^{\wedge_{\lambda,z}}$.
It follows that (\ref{eq:completed_homom}) is injective. Hence $K^{\wedge_{\lambda,z}}=0$
for all $(\lambda,z)\in (\param\times \C)^0$. Using Lemma \ref{Lem:gen_freeness},
we conclude that $K=0$.
\end{proof}

\subsection{Translation and wall-crossing functors}
In this section we will study the behavior of the translation and wall-crossing functors
on the category $\mathcal{O}$. Recall that it makes sense to speak about the wall-crossing
functor $\WC_{\lambda+\chi\leftarrow \lambda}$ when abelian localization holds for
$(\lambda+\chi,\theta)$ and $(\lambda,\theta')$ for generic $\theta,\theta'$.
Note that $\WC_{\lambda+\chi\leftarrow \lambda}$
restricts to an equivalence
$D^b_{\mathcal{O}_\nu}(\A_\lambda\operatorname{-mod})
\rightarrow D^b_{\mathcal{O}_\nu}(\A_{\lambda+\chi}\operatorname{-mod})$.

\subsubsection{Behavior on $K_0(\mathcal{O}_\nu^T)$}\label{SSS_equiv_K_0_behavior}
Pick a $T$-equivariant structure on $O(\chi)$ (defined up to a twist with a character).
This gives a $T$-equivariant structure on $\A_{\paramq,\chi}^\theta$ and on $\A_{\lambda,\chi}$.
So $\WC_{\lambda+\chi\leftarrow \lambda}$ upgrades to an equivalence $D^b_{\mathcal{O}_\nu}(\A_\lambda\operatorname{-mod}^T)
\xrightarrow{\sim} D^b_{\mathcal{O}_\nu}(\A_{\lambda+\chi}\operatorname{-mod}^T)$.

Let $\mathsf{wt}_\chi(\mathsf{x})$ stand for the character of $T$
at the fiber of $O(\chi)$ at the fixed point $\mathsf{x}$.

\begin{Prop}\label{Prop:transl_grad_O}
Suppose that
\begin{itemize}
\item[(i)] $\lambda,\lambda+\chi\in \paramq^{fl}$,
\item[(ii)] and abelian localization holds for $(\lambda+\chi,\theta)$ and $(\lambda,\theta')$
for some generic $\theta,\theta'$.
\end{itemize}
Then we have
\begin{enumerate}
\item For $?=\lambda,\lambda+\chi$, the classes $[\Delta_{\nu,?}(\mathsf{x})]$ form a basis in $K_0(\mathcal{O}_\nu(\A_?))$.
\item  $[\WC_{\lambda+\chi\leftarrow \lambda} \Delta_{\nu,\lambda}(\mathsf{x},\kappa)]=[\Delta_{\nu,\lambda+\chi}(\mathsf{x}, \kappa+\mathsf{wt}_\chi(\mathsf{x}))]$.
\item Moreover, the characteristic cycle of $\operatorname{Loc}_\lambda^{\theta'}\Delta_{\lambda,\nu}(\mathsf{x})$ has the contracting component
    of $\mathsf{x}$ with multiplicity $1$ and the other contracting components appearing
    with nonzero multiplicities correspond to the fixed points that are less than $\mathsf{x}$
    in the order induced by $\nu$.
\end{enumerate}
\end{Prop}
\begin{proof}
For $(\lambda,z)\in \param\times \C$ we define the category
$\mathcal{O}^{T}_\nu(\A^\theta_{\lambda,z})$  of coherent $T$-weakly equivariant $\A^\theta_{\lambda,z}$-modules that are set theoretically supported on $X^+$.
Similarly, we can define the category $\mathcal{O}^{T}_\nu(\A^\theta_{\param,\hbar})$.
We have the degeneration map $K_0(\mathcal{O}^{gr}_\nu(\A^\theta_{\lambda,z}))\rightarrow K_0(\operatorname{Coh}^T_{X^+}(X))$. Thanks to this map we can talk about the characteristic
cycle of an object in $\mathcal{O}^{T}_\nu(\A^\theta_{\lambda,z})$.

Note that we have a functor
$$\operatorname{Loc}^\theta:=\A_{\param,\hbar}^\theta\otimes^L_{\A_{\param,\hbar}}\bullet:
D^b(\A_{\param,\hbar}\operatorname{-mod})\rightarrow
D^-(\Coh(\A^\theta_{\param,\hbar})).$$
This functor upgrades to the weakly $T$-equivariant categories.
Hence we also get a functor
$$\operatorname{Loc}^\theta:=\A_{\param,\hbar}^\theta\otimes^L_{\A_{\param,\hbar}}\bullet:
D^b_{\mathcal{O}_\nu}(\A_{\param,\hbar}\operatorname{-mod}^T)\rightarrow
D^-_{\mathcal{O}_\nu}(\Coh^T(\A^\theta_{\param,\hbar})).$$

Set $\A_{(\param\times \C)^0}:=\C[(\param\times \C)^0]\otimes_{\C[\param,\hbar]}\A_{\param,\hbar}$.
Recall, Proposition \ref{Prop:fin_proj_dim}, that the projective dimension of the $\A_{(\param\times \C)^0}\otimes_{\C[(\param\times \C)^0]}
\A_{(\param\times \C)^0}^{opp}$-module $\A_{(\param\times \C)^0}$ does not exceed $2\dim X$.
It follows that the projective dimension of the regular $\A_{(\param\times \C)^0}$-bimodule
does not exceed $k:=2\dim X+\dim \param+1$. Consider the functor
$$\operatorname{Loc}^\theta_{\geqslant}:
D^b_{\mathcal{O}_\nu}(\A_{\param,\hbar}\operatorname{-mod}^T)\rightarrow
D^b_{\mathcal{O}_\nu}(\Coh^T(\A^\theta_{\param,\hbar})),$$
the composition of $\operatorname{Loc}^\theta$ and the truncation functor
$\tau_{\geqslant -k}$.  Note that the localization of $\Loc^\theta_{\geq}$ to
$(\param\times \C)^0$ coincides with the localization of $\Loc^\theta$.

Now consider the objects $$\M^1_{\param,\hbar}:=\A^\theta_{\param,\hbar,\chi}\otimes_{\A^\theta_{\param,\hbar}}\Loc^\theta_{\geqslant}(\Delta_{\nu,\param,\hbar}(\mathsf{x},\kappa)),
\M^2_{\param,\hbar}:= \Loc^\theta_{\geqslant}(\mathsf{x}, \kappa+\mathsf{wt}_\chi(\mathsf{x}))\in
D^b_{\mathcal{O}_\nu}(\Coh^T(\A^\theta_{\param,\hbar})).$$
For $(\lambda,z)\in \param\times \C$, we set $\M^1_{\lambda,z}=\C_{\lambda,z}\otimes^L_{\C[\param,\hbar]}
\M^1_{\param,\hbar}$ and define $\M^2_{\lambda,z}$ similarly.
 Since $\lambda+\chi,\lambda\in \paramq^{fl}$ we have the following isomorphisms
\begin{equation}\label{eq:specializations1}
\M^1_{\lambda+\chi}\cong\A^\theta_{\lambda,\chi}\otimes_{\A_\lambda^\theta}
\Loc_\lambda^\theta(\Delta_{\nu,\lambda}(\mathsf{x},\kappa)),
\M^2_{\lambda+\chi}\cong\Loc_{\lambda+\chi}^\theta(\Delta_{\nu,\lambda+\chi}(\mathsf{x},
\kappa+\mathsf{wt}_\chi(\mathsf{x}))).
\end{equation}
Now let $\zeta\in \param^{reg}$.
We write $X_{\zeta,\mathsf{x}}^+$ for the contracting locus of the $T$-fixed point corresponding to
$\mathsf{x}$ in $X_{\zeta}$. Note that the restriction of
$O_{X_{\zeta}}(\chi)$ to $X_{\zeta,\mathsf{x}}^+$ is the trivial line bundle
with $T$ acting via $\mathsf{wt}_\chi(\mathsf{x})$. It follows that
\begin{equation}\label{eq:specializations2}
\M^1_{\zeta,0}\cong \M^2_{\zeta,0}\cong \mathcal{O}_{X^+_{\zeta,\mathsf{x}}}(\kappa+\mathsf{wt}_\chi(\mathsf{x}))).
\end{equation}

For an object $\mathcal{M}\in K_0(\Coh^T_{\mathcal{O}_\nu}(\A^\theta_{\param,\hbar}))$
the image of its specialization to $(\lambda,z)$ in $K_0(\operatorname{Coh}^T_{X^+}(X))$
is independent of the choice of $(\lambda,z)$. This is because this image coincides with
the class of $\C_{0,0}\otimes^L_{\C[\param,\hbar]}\M$.
In particular, the images of $[\M^1_{\lambda+\chi}], [\M^2_{\lambda+\chi}]$
coincide with that of $\mathcal{O}_{X^+_{\zeta,\mathsf{x}}}(\kappa+\mathsf{wt}_\chi(\mathsf{x})))$.
The characteristic cycle of the latter has the property listed in (3). This implies (3).

This argument also implies that the classes of $[\Delta_{\nu,\lambda}(\mathsf{x}, \kappa)]$ for $\mathsf{x}\in X^T,\kappa\in \mathfrak{X}(T),$
in $K_0(\OCat^{T}_{\nu}(\A_{\lambda+\chi}))$ are linearly independent because their images in
$K_0(\operatorname{Coh}^T_{X^+}(X))$ are. Therefore the classes $[\Delta_{\nu,\lambda}(\mathsf{x})]\in K_0(\OCat_{\nu}(\A_{\lambda+\chi}))$ are linearly independent.
Note that $$\dim K_0(\OCat_{\nu}(\A_{\lambda+\chi}))=
\dim K_0(\OCat_{\nu}(\A^\theta_{\lambda+\chi})).$$
The right hand side equals $|X^T|$ by Corollary \ref{Cor:local_hw}.  We get (1). Since the classes
$[\Delta_{\nu,\lambda}(\mathsf{x}, \kappa)]$ are linearly independent, we see that (\ref{eq:specializations1}) and (\ref{eq:specializations2}) imply
$[\M^1_{\lambda+\chi}]=[\M^2_{\lambda+\chi}]$. This, in turn, implies
$[R\Gamma\M^1_{\lambda+\chi}]=[R\Gamma\M^2_{\lambda+\chi}]$, which is (2).
\end{proof}

\subsubsection{Translations of Verma modules}
Our goal here is to prove the following result.

\begin{Prop}\label{Prop:transl_Verma}
Suppose that $\lambda\in \paramq$ and $\chi\in \operatorname{Pic}(X)$. Assume that the following
hold:
\begin{itemize}
\item[(i)] Conditions of Lemma \ref{Lem:Ext_vanishing} hold for $\lambda,\lambda+\chi$,
\item[(ii)]  $\WC_{\lambda+\chi\leftarrow \lambda}$ is an abelian equivalence,
\item[(iii)] $\lambda,\lambda+\chi\in \paramq^{fl}$.
\end{itemize}
Then the following claims hold.
\begin{enumerate}
\item $\A_{\lambda,\chi}\otimes_{\A_\lambda}\Delta_{\nu,\lambda}(\mathsf{x})\cong \Delta_{\nu,\lambda+\chi}(\mathsf{x})$.
\item Let $\paramq'\subset \paramq$ be a principal open subset such that conditions
(i)-(iii) hold for all $\lambda'$. Then we have
$\A_{\paramq',\chi}\otimes_{\A_\paramq}\Delta_{\nu,\paramq}(\mathsf{x})\cong \Delta_{\nu,\paramq'+\chi}(\mathsf{x})$.
\end{enumerate}
\end{Prop}
\begin{proof}
We prove part (1).
Note that by Proposition \ref{Prop:transl_grad_O}  (more precisely, its weaker version that does not take the equivariant structures into account), we have $[\A_{\lambda,\chi}\otimes_{\A_\lambda}\Delta_{\nu,\lambda}(\mathsf{x})]\cong [\Delta_{\nu,\lambda+\chi}(\mathsf{x})]$. But both $\A_{\lambda,\chi}\otimes_{\A_\lambda}\Delta_{\nu,\lambda}(\mathsf{x}), \Delta_{\nu,\lambda+\chi}(\mathsf{x})$ are standard objects for highest weight structures on
$\mathcal{O}_\nu(\A_{\lambda+\chi})$. As was checked in \cite[Lemma 4.3.2]{GL}, this implies that the modules are isomorphic.

Now we proceed to proving (2). The specializations of the both sides to any point
$\lambda\in \paramq'$ are isomorphic. We can equip $\A_{\paramq',\chi}\otimes_{\A_\paramq}\Delta_{\nu,\paramq}(\mathsf{x})$ and $ \Delta_{\nu,\paramq'+\chi}(\mathsf{x})$ with gradings compatible with the Hamiltonian
$\C^\times$-actions: we grade  $\Delta_{\nu,\paramq'+\chi}(\mathsf{x}), \Delta_{\nu,\paramq'}(\mathsf{x})$
by putting the highest weight space $\C[\paramq']$ in degree $0$ and
$\A_{\paramq',\chi}\otimes_{\A_\paramq}\Delta_{\nu,\paramq}(\mathsf{x})$ gets the tensor
product grading. Note that for any compatible grading on a Verma modules that highest weight
component coincides with the highest degree components so the grading is uniquely recovered
up to a shift. The specializations of $\A_{\paramq',\chi}\otimes_{\A_\paramq}\Delta_{\nu,\paramq}(\mathsf{x})$ and $ \Delta_{\nu,\paramq'+\chi}(\mathsf{x})$ to any $\lambda'$ inherit the gradings.
It follows that we can shift the grading on $\A_{\paramq',\chi}\otimes_{\A_\paramq}\Delta_{\nu,\paramq}(\mathsf{x})$ so that the
highest weight component in all specializations is in degree $0$. This highest weight
component is a projective rank one $\C[\paramq']$-module. It is therefore free. It is annihilated by
$\A_{\paramq'}^{>0}$. Also it is annihilated by the kernel of $\Ca_{\nu}(\A_{\paramq'})\rightarrow
\C[\paramq']_{\mathsf{x}}$ because this is so in every specialization by part (1).
This gives rise to a homomorphism
$\Delta_{\nu,\paramq'+\chi}(\mathsf{x})\rightarrow
\A_{\paramq',\chi}\otimes_{\A_\paramq'}\Delta_{\nu,\paramq}(\mathsf{x})$. This homomorphism
is an isomorphism after every specialization, hence an isomorphism. This finishes the proof.
\end{proof}

\subsubsection{Long wall-crossing  as  Ringel duality functor}
Let us recall a result from \cite[Section 7.3]{CWR}. Let abelian localization hold for $(\lambda,\theta),(\lambda+\chi,-\theta)$.

\begin{Lem}\label{Lem:long_WC_Ringel}
The wall-crossing functor $\WC_{\lambda+\chi\leftarrow \lambda}:D^b(\mathcal{O}_{\nu}(\A_\lambda))
\xrightarrow{\sim} D^b(\mathcal{O}_\nu(\A_{\lambda+\chi}))$ is a Ringel duality functor.
\end{Lem}

\subsection{Category $\mathcal{O}$ for very generic parameters}
We are going to prove a number of results on the category $\mathcal{O}$ at very generic (=lying outside of all essential hyperplanes) parameters.

\begin{Prop}\label{Prop:cat_O_simpl}
Let $\lambda$ be a very generic parameter. Then the following claims are true:
\begin{itemize}
\item[(i)] The category $\mathcal{O}_\nu(\A_\lambda)$ is semisimple.
\item[(ii)] $\lambda\in \paramq^{fl}$.
\item[(iii)] The natural homomorphism $\Ca_{\nu}(\A_\lambda)\rightarrow \C[X^T]$
is an isomorphism.
\item[(iv)] The conclusion of Lemma \ref{Lem:Ext_vanishing} holds, equivalently, we have $$\dim\operatorname{Tor}_i^{\A_\lambda}(\Delta_\nu(\mathsf{x}), \Delta^r_{-\nu}(\mathsf{x}'))=\delta_{i0}\delta_{\mathsf{x},\mathsf{x}'}.$$
\end{itemize}
\end{Prop}
\begin{proof}
Let us prove (i). By Lemma \ref{Loc:localization_very_generic}, the categories $\mathcal{O}_\nu(\A_\lambda)$ and $\mathcal{O}_\nu(\A_\lambda^\theta)$
are equivalent for any generic $\theta$ and all wall-crossing functors
are t-exact equivalences. This equips $\mathcal{O}_\nu(\A_\lambda)$ with a highest
weight structure.
It follows from Lemma \ref{Lem:long_WC_Ringel} that the Ringel duality equivalences
%
between $\mathcal{O}_\nu(\A_\lambda),\mathcal{O}_{\nu}(\A_{\lambda+\chi})$,
where $\chi$ has the same meaning as in Lemma \ref{Lem:long_WC_Ringel},
are $t$-exact. Consider the self-equivalence $\mathcal{F}$ of $\OCat_\nu(\A_\lambda)$
obtained by composing the Ringel duality functors
$\OCat_\nu(\A_\lambda)\rightarrow \OCat_\nu(\A_{\lambda+\chi})\rightarrow
\OCat_\nu(\A_\lambda)$. This equivalence sends projectives to injectives and is
t-exact. So every projective in $\OCat_\nu(\A_\lambda)$ is also injective. But since
$\OCat_\nu(\A_\lambda)$ is highest weight, an easy induction on the highest weight order shows that  this is only possible when $\OCat_\nu(\A_\lambda)$ is semisimple.

Let us prove (ii). We need to show that the $T$-character of $\Delta_{\nu,\lambda}(\mathsf{x})$
coincides with that of $\C[(T_{\mathsf{x}}X)^{>0}]$. Note that there is an object $\Theta_{\mathsf{x}}\in
\OCat_{\nu}(\A_\lambda)$
(constructed in \cite[Section 5.3]{BLPW}) with that $T$-character and a nonzero homomorphism
$\Delta_{\nu,\lambda}(\mathsf{x})\rightarrow \Theta_{\mathsf{x}}$. The module
$\Delta_{\nu,\lambda}(\mathsf{x})$ is indecomposable, hence simple. Therefore the homomorphism
is injective. Also the dimensions of
graded components of $\Delta_{\nu,\lambda}(\mathsf{x})$ are minimized for $\lambda\in
\paramq^{fl}$. We conclude that $\lambda\in \paramq^{fl}$.

Let us prove (iii). First, let us show that the algebra $\Ca_{\nu}(\A_\lambda)$  is semisimple. Indeed, otherwise there
is a non-split exact sequence $0\rightarrow N_1\rightarrow N\rightarrow N_2\rightarrow 0$, where
$N_1,N_2$ are simple $\Ca_\nu(\A_\lambda)$-modules necessarily with the
same action of $h$, say, by scalar $\alpha$. The object $\Delta_\nu(N)$ is completely reducible because
the category $\OCat_\nu(\A_\lambda)$ is semisimple. The $\alpha$-eigenspace for $h$ in $\Delta_\nu(N)$
is $N$. Since $\Delta_\nu(N)$ has no homomorphisms to modules with zero $\alpha$-eigenspace, we see that
any nonzero direct summand of $\Delta_\nu(N)$ must have a nonzero intersection with $N$. This intersection is
$\Ca_{\nu}(\A_\lambda)$-stable. Since $N$ is indecomposable, it follows that $\Delta_\nu(N)$ is indecomposable,
hence simple. But $\Delta_\nu(N)\twoheadrightarrow \Delta_\nu(N_2)$, therefore, this epimorphism is an iso.
Since the $\alpha$-eigenspace in $\Delta_\nu(N_2)$ is $N_2$ we arrive at a contradiction. We conclude
that the algebra $\Ca_\nu(\A_\lambda)$ is semisimple.

Proposition \ref{Prop:transl_grad_O} applies to $\lambda$ thanks to (ii) and Lemma
\ref{Loc:localization_very_generic}. We conclude that the classes $\Delta_{\nu,\lambda}(\mathsf{x})$ form a basis in $K_0(\OCat_\nu(\A_\lambda))$. On the
other hand the classes of $[\Delta_{\nu,\lambda}(N)]$ for $N\in \operatorname{Irr}(\Ca_\nu(\A_\lambda))$
form a basis in $K_0(\OCat_\nu(\A_\lambda))$. This implies that the pullback under the homomorphism
$\Ca_\nu(\A_\lambda)\rightarrow \Ca_\nu(\A_\lambda^\theta)$ is an identification
of the sets of irreducibles for these algebras. Since $\Ca_\nu(\A_\lambda)$ is semisimple,
we conclude that the homomorphism is an isomorphism. This proves (iii).

Let us prove (iv). Recall, Lemma \ref{Lem:Ext_Tor},
that $$\operatorname{Tor}_i^{\A_\lambda}(\Delta_\nu(\mathsf{x}), \Delta_{-\nu}^r(\mathsf{x}))^*=
\operatorname{Ext}^i_{\A_\lambda}(\Delta_\nu(\mathsf{x}),\nabla_\nu(\mathsf{x}')).$$ Since abelian localization holds
for $(\lambda,\theta)$, we see that this Ext is the same as
$$\operatorname{Ext}^i_{\Coh(\A^\theta_\lambda)}(\operatorname{Loc}_\lambda^\theta\Delta_\nu(\mathsf{x}),
\operatorname{Loc}_\lambda^\theta\nabla_\nu(\mathsf{x}')).$$  But the latter Ext is the same as in the category
$\mathcal{O}_\nu(\A_\lambda^\theta)$, see Section \ref{SSS_Ext_vanish}.
This category is semisimple and the objects
we consider are simple. So the latter Ext is $\delta_{i,0}\delta_{\mathsf{x},\mathsf{x}'}$ we get the claim of (iii).
\end{proof}

\subsection{Regular parameters and quantum chambers}\label{SS_reg_param}
\subsubsection{Terminology}\label{SSS_reg_terminology}
Let us start by introducing some terminology.

Pick $\lambda\in \paramq$. Consider the set of all classical walls $\Gamma$
such that there is an essential hyperplane parallel to $\Gamma$ containing $\lambda$. These walls $\Gamma$ will be called {\it integral walls for $\lambda$}.
They split $\param_\R$ into the set of polyhedral chambers to be called {\it integral
chambers} for $\lambda$.

A subset $\Sigma$ of $\C$ will be called {\it saturated} if for any two elements $z,z+n\in \Sigma$
with $n\in \Z_{>0}$, we have $z+1,z+2,\ldots,z+n-1\in \Sigma$.

Now suppose that we have fixed saturated subsets $\tilde{\Sigma}_\Gamma\subset \Sigma_\Gamma$ for each wall $\Gamma$.
Suppose $\lambda\in \paramq$ satisfies $\langle \lambda,\alpha_\Gamma\rangle\not\in \tilde{\Sigma}_\Gamma$
(note that this is automatic if $\Gamma$ is not an integral wall for $\lambda$). Then there is a unique
integral chamber $C^{int}$ for $\lambda$ with the following property: for $\lambda'\in \lambda+(C^{int}\cap \param_\Z)$
we have $\langle \lambda',\alpha_\Gamma\rangle\not\in \tilde{\Sigma}_\Gamma$ for all $\Gamma$.
We will say that $C^{int}$ is a {\it positive chamber for $\lambda$}.
%

Now let us define {\it quantum chambers}. Let $\lambda\in \paramq$ and pick an integral chamber $C^{int}$. Quantum chambers will be subsets in $\lambda+\paramq_\Z$ defined by linear inequalities.
Namely, let $\Gamma_1,\ldots,\Gamma_k$ be all walls such that
$\langle \lambda,\alpha_{\Gamma_i}\rangle\in \Sigma_{\Gamma_i}$. Let $\epsilon_{\Gamma_i}$ be the sign of $\alpha_{\Gamma_i}$ on
$C^{int}$.
Pick minimal integers $m_i, i=1,\ldots,k$ such that
$\tilde{m}_i:=\epsilon_{\Gamma_i}\langle \lambda,\alpha_{\Gamma_i}\rangle+m_i\not\in \tilde{\Sigma}_{\Gamma_i}$.
Then we define the  quantum chamber $C^{q}$ as the subset of the set of regular parameters
in  $\lambda+\paramq_\Z$ given by the inequalities
$\epsilon\langle\alpha_\Gamma,\bullet\rangle\geqslant \tilde{m}_i$. We say that it is {\it shifted}
from $C^{int}$.

By the definition, different  quantum chambers are disjoint. Also  $C^{int}$ is a positive integral
chamber for $\lambda'\in \lambda+\paramq_{\Z}$ if and only if $\lambda'$ lies in some
shifted (from $C^{int}$) quantum chamber. If $\lambda'$ lies in a quantum chamber, then it is
unique.

%

\begin{Ex}\label{Ex:chamber_weird}
Let $X=T^*\mathcal{B}$ for $G=\operatorname{SL}_3(\C)$.
Let $\alpha_1,\alpha_2$ denote the simple roots and let $\alpha_{12}=\alpha_1+\alpha_2$. Then we have three
walls: $\Gamma_1:=\ker\alpha_1^\vee, \Gamma_2:=\ker\alpha_2^\vee, \Gamma_{12}:=\ker\alpha_{12}^\vee$.
Set $\tilde{\Sigma}_{\Gamma_1}=\{0\}, \tilde{\Sigma}_{\Gamma_2}=\{0\}, \tilde{\Sigma}_{\Gamma_{12}}=\{-2,-1,0,1,2\}$.
Let us take an integral weight $\lambda$. One of the integral chambers is the positive (in the usual Lie-theoretic sense)
chamber $C^{int}$, it is given by $\langle\alpha_1^\vee,\bullet\rangle\geqslant 0, \langle\alpha_2^\vee,\bullet\rangle\geqslant 0$. The corresponding quantum chamber is given by $\{(x_1,x_2)| x_1\geqslant 1, x_2\geqslant 1, x_1+x_2\geqslant 3\}$, where $x_1:=\langle\alpha_1^\vee,\lambda\rangle,
x_2:=\langle\alpha_2^\vee, \lambda\rangle$. In particular, $C^{q}$ is not a cone.
\end{Ex}


\subsubsection{Result}
\begin{Prop}\label{Prop:generic}
For each wall $\Gamma$, there is a finite saturated subset $\widetilde{\Sigma}_\Gamma\subset \Sigma_\Gamma$
such that for union $\tilde{\Psi}$ of the hyperplanes $\langle \alpha_{\Gamma},\bullet\rangle
\in \tilde{\Sigma}_{\Gamma}$, where $\Gamma_i$ runs over the set of all classical walls,
the following hold:
\begin{enumerate}
\item If $\lambda\not\in \tilde{\Psi}$, then abelian localization holds for
$(\lambda',\theta)$, where $\theta$ is in the integral chamber $C^{int}$ of $\lambda$,
and any $\lambda'\in \lambda+(C^{int}\cap \param_{\Z})$.
\item If $\lambda\not\in \tilde{\Psi}$, then $\lambda\in \paramq^{fl}$,
\item  If $\lambda\not\in \tilde{\Psi}$, then $\Ca_\nu(\A_\lambda)\xrightarrow{\sim}\Ca_\nu(\A_\lambda^\theta)$,
\item if $\lambda\not\in \tilde{\Psi}$, then
$\dim \operatorname{Ext}^i_{\A_\lambda}(\Delta_{\nu,\lambda}(\mathsf{x}),\nabla_{\nu,\lambda}(\mathsf{x}'))=
\delta_{i,0}\delta_{\mathsf{x},\mathsf{x}'}$.
\end{enumerate}
\end{Prop}
\begin{proof}
Taking $\Psi$ from Lemma \ref{Lem:cor_localn} and including it into a saturated
collection $\tilde{\Psi}$, we get $\tilde{\Psi}$ satisfying (1). The conditions
 (2)-(4) hold for very generic $\lambda$ thanks to Proposition
\ref{Prop:cat_O_simpl}. On the other hand, the loci where conditions (2) and (3) fail
are Zariski closed. For (2), this follows from Corollary \ref{Cor:Verma_flat_open}.
For (3), the locus of $\lambda$, where $\Ca_\nu(\A_\lambda)\rightarrow \Ca_{\nu}(\A_\lambda^\theta)$
is precisely the locus $\paramq'$ in $\paramq$, where $\Ca_\nu(\A_{\paramq})\rightarrow
\Ca_\nu(\A^\theta_{\paramq})$ is an isomorphism. This follows because
$\Ca_\nu(\A^\theta_{\paramq})$ is projective over $\C[\paramq]$. And since both
$\Ca_\nu(\A_{\paramq}), \Ca_\nu(\A^\theta_{\paramq})$ are finitely generated
$\C[\paramq]$-modules, the locus $\paramq'$ is Zariski open.

Now we can find $\tilde{\Psi}$ satisfying (2) and (3)
as well. It remains to show that we can find $\tilde{\Psi}$ so that (4) is satisfied as well.

Set $\paramq^0:=\paramq\cap (\param\times \C)^0$. We claim that the locus in
$\paramq^0\cap \paramq^{fl}$ where (4) holds is Zariski open. Consider
$\Delta^{opp}_{\nu,\paramq}(\mathsf{x}')\otimes^L_{\A_\paramq}
\Delta_{\nu,\paramq}(\mathsf{x})$, this is an object in
$D^-(\C[\paramq]\operatorname{-mod})$. The localization to
$\paramq^0$ is bounded. Moreover, for $\lambda\in \paramq^0\cap \paramq^{fl}$, we have
$$\C_\lambda\otimes^L_{\C[\paramq]}\Delta^{opp}_{\nu,\paramq}(\mathsf{x}')\otimes^L_{\A_\paramq}
\Delta_{\nu,\lambda}(\mathsf{x})\cong\Delta^{opp}_{\nu,\lambda}(\mathsf{x}')\otimes^L_{\C}
\Delta_{\nu,\lambda}(\mathsf{x}).$$
So the locus in $\paramq^0\cap \paramq^{fl}$, where the right hand side is in
homological degree $0$ and vanishes for $\mathsf{x}\neq \mathsf{x}'$ is Zariski open.
Now we are done by Lemma \ref{Lem:Ext_Tor}.

\end{proof}

\begin{defi}\label{defi:regular_parameters}
Let $\paramq^{reg}$ denote the complement to $\tilde{\Psi}$ in $\paramq$.
The elements of $\paramq^{reg}$ will be called regular.
\end{defi}

\begin{Rem}\label{Rem:hom_dim}
Note that we can make (1) hold even in the case when there is no torus action.
Also it follows from Proposition \ref{Prop:fin_proj_dim}
that the homological dimension of $\A_{\paramq^{reg}}$
does not exceed $2\dim X+\dim \param$.
\end{Rem}

\subsubsection{Examples}\label{SSS_quantum_walls_examples}
Let us explain what being regular means in the examples we consider: the cotangent bundles
$T^*\mathcal{B}$ and  the Hilbert schemes $\operatorname{Hilb}_n(\C^2)$.

Let us start with the case of $X=T^*\mathcal{B}$. 

For $X=T^*\mathcal{B}$, it is easy to see that we can take $\tilde{\Sigma}_\Gamma=\{0\}$ for all $\Gamma$ so that
regular in our sense is the same as regular in the usual sense ($\langle\alpha^\vee,\lambda\rangle \neq 0$
for all roots $\alpha$). Integral chambers for $\lambda$ are given by $\langle \alpha_i^\vee,\bullet\rangle
\geqslant 0, i=1,\ldots,k$, where $\alpha_1,\ldots,\alpha_k$ is a system of simple roots for the integral
Weyl group $W_{[\lambda]}$. Quantum chambers are given by
$\langle \alpha_i^\vee,\bullet\rangle\geqslant 1$.

For $X=\operatorname{Hilb}_n(\C^2)$, we have a single wall $\Gamma=\{0\}$ and $\Sigma_{\Gamma}=\{-\frac{a}{b}+1/2|
1\leqslant a<b\leqslant n\}$. 
We will take $\tilde{\Sigma}_\Gamma$ of the form $\bigcup_{k=-\ell}^\ell(\Sigma_\Gamma+k)$, where $\ell$ is a suitable (sufficiently big) positive integer. Clearly,
$\tilde{\Sigma}_\Gamma$ is saturated. When $c=\lambda-1/2$ is irrational, integer, or
rational with denominator bigger than $n$, then there is only one integral chamber, while otherwise there are two:
$\mathbb{R}_{\geqslant 0}, \mathbb{R}_{\leqslant 0}$. The  quantum chambers (for $c\in -\frac{a}{b}+\Z$) are of
the form $\{-\frac{a}{b}+\ell+m| m\in \Z_{>0}\}$ and $\{-\frac{a}{b}-\ell-m| m\in \Z_{>0}\}$.

%

\section{Alcoves}
This is a technical section, our goal here is to discuss various things related to alcoves defined
by the walls $\Gamma$ and the subsets $\tilde{\Sigma}_\Gamma\subset \Sigma_\Gamma$. Here we only deal with
questions from elementary combinatorial geometry.

We start with a finitely generated lattice $\param_{\Z}$, a finite  collection of codimension one sublattices
to be called {\it walls}. For each wall $\Gamma$, we have a primitive element $\alpha_\Gamma\in \param_\Z^*$ defined up
to a sign with $\Gamma=\ker \alpha_\Gamma$. Further, for each $\Gamma$ we fix a finite saturated subset $\tilde{\Sigma}_\Gamma\subset \Q$. Let $\Sigma_\Gamma$ denote the set of classes of elements from $\tilde{\Sigma}_\Gamma$
in $\Q/\Z$.

\subsection{Real alcoves and $p$-alcoves}\label{SS_alcoves}
We define two closely related sets of {\it alcoves}. One of them, {\it real alcoves}, will be in $\param_\R$, while the other, {\it p-alcoves}, will be in $\param_\Z$.

Let us start with real alcoves. Consider the hyperplanes of the form $\langle\alpha_\Gamma,\bullet\rangle=
\sigma$ for $\sigma\in \Sigma_\Gamma$ for all walls $\Gamma$. By  real alcoves we mean the closures of the connected
components of $\param_\R$ with these hyperplanes removed. Note that the set of real alcoves is stable under translations by $\param_\Z$.

\begin{Ex}\label{Ex_real_alcoves_cotang}
Let us consider the very classical case of $X=T^*\mathcal{B}$ and the corresponding data of $\param_\Z$, walls, and sets
$\tilde{\Sigma}_\Gamma=\{0\}$. The hyperplanes we consider are of the form  $\langle\alpha^\vee, \bullet\rangle=m$, where $\alpha^\vee$ is a coroot. Let $\alpha^\vee_1,\ldots,\alpha^\vee_r$ denote the simple coroots
and $\alpha_0^\vee$ be the minimal coroot. Then we have the so called fundamental alcove defined by
$\langle\alpha^\vee_i,\bullet\rangle\geqslant 0$ for $i=1,\ldots,r,$ and $\langle\alpha_0^\vee,\bullet\rangle
\geqslant -1$. All other alcoves are obtained from this one by the standard action of the affine
Weyl group of $W$ on $\param_\R$.
\end{Ex}

\begin{Ex}\label{Ex_real_alcoves_Hilb}
Now let us consider the case of $X=\operatorname{Hilb}_n(\C^2)$ (and the parameter $c$).
Here, recall, $\param_\Z=\Z,\Gamma=\{0\}$
and $\tilde{\Sigma}_{\Gamma}=\bigsqcup_{i=-\ell}^{\ell} \{-\frac{a}{b}+i| 1\leqslant
a<b\leqslant n\}$. It follows that
the real alcoves are intervals between consecutive rational numbers with denominators from $2$ to $n$.
\end{Ex}

Alternatively, the real alcoves can be described as follows. For $\sigma\in \Sigma_\Gamma$, we can consider the half spaces
\begin{equation}\label{eq:half_spaces}
\Upsilon_{\sigma}^+=\{\lambda\in \paramq| \langle\alpha_\Gamma,\lambda\rangle\geqslant
\sigma\},
\Upsilon_{\sigma}^-=\{\lambda\in \paramq| \langle\alpha_\Gamma,\lambda\rangle\leqslant
\sigma\}
\end{equation}
Every real alcove is an intersection of finitely many half spaces of this form.
Moreover, for a real alcove $A$ there is a unique minimal collection of half-spaces
with this property. We call them the half-spaces {\it associated} to $A$.

Now we proceed to defining the $p$-alcoves. Fix an integer $p\gg 0$ such that $p+1$ is divisible by the denominator of any element in $\tilde{\Sigma}_\Gamma$
for all $\Gamma$. Later on, $p$ will assumed to be prime but at this point we do not require that. Consider
the hyperplanes of the form $\langle\alpha_\Gamma,\bullet\rangle=(p+1)\sigma+pm$, where $\sigma\in \tilde{\Sigma}_\Gamma$
and $m\in \Z$. 

Let $\tilde{A}_\R$ denote a connected component of the complement of the union of these
hyperplanes in $\paramq_{\R}$. By a $p$-alcove we mean the intersection of an open subset of the form
$\tilde{A}_\R$ with $\paramq_{\Z}$. So every  $p$-alcove is the set of integral points $\lambda$ subject to
the inequalities of one of  the following forms
\begin{itemize}
\item
$\langle\alpha_\Gamma,\lambda\rangle\geqslant (p+1)\sigma+pm+1$ for $\sigma$ maximal in
$\tilde{\Sigma}_{\Z}\cap (\sigma+\Z)$ (and all possible $m$),
\item $\langle\alpha_\Gamma,\lambda\rangle\leqslant (p+1)\sigma+pm-1$ for $\sigma$ minimal in
$\tilde{\Sigma}_{\Z}\cap (\sigma+\Z)$ (and all possible $m$).
\end{itemize}

Alternatively, we can describe the $p$-alcoves as follows. Let $\Gamma$ be a classical wall.
Choose $\sigma\in \Sigma_\Gamma$ (defined up to an integral summand)
and let $\sigma_+,\sigma_-$ denote the maximal and minimal elements in $\sigma+\mathbb{Z}$.
For $m\in \Z$, we can consider the affine half spaces $\Upsilon_{\sigma,m}^+, \Upsilon_{\sigma,m}^-\subset \paramq_{\R}$ defined as follows:
\begin{equation}\label{eq:subspaces}
\tilde{\Upsilon}_{\sigma,m}^+=\{\lambda\in \paramq| \langle\alpha_\Gamma,\lambda\rangle>
(p+1)\sigma_++pm\},
\tilde{\Upsilon}_{\sigma,m}^-=\{\lambda\in \paramq| \langle\alpha_\Gamma,\lambda\rangle<
(p+1)\sigma_-+pm\}.
\end{equation}
Any $p$-alcove is the intersection of $\param_{\Z}$ with a finite collection of the half spaces of this kind. Moreover, for any $p$-alcove $\tilde{A}$ there is a unique minimal collection
of half spaces as above whose intersection coincides with $A$. Similarly
to the real case, we call them the half-spaces {\it associated} to $\tilde{A}$.

\begin{Ex}\label{Ex_p_alcoves_cotang}
Let $X=T^*\mathcal{B}$. Then the hyperplanes are of the form $\langle\alpha,\bullet\rangle=pm$
for $m\in \Z$. We have the fundamental $p$-alcove, whose points are all integral $\lambda$
such that $\langle\alpha_i^\vee,\lambda\rangle\geqslant 1, \langle \alpha_0^\vee,\lambda\rangle
\geqslant 1-p$. All other $p$-alcoves are obtained from the fundamental one by the affine
Weyl group action, where an element $\mu$ of the root lattice acts by the shift by $p\mu$.
\end{Ex}

\begin{Ex}\label{Ex_p_alcoves_Hilb}
Let $X=\operatorname{Hilb}_n(\C^2)$. Let us describe the $p$-alcoves  for the parameter
$c=\lambda-1/2$.  By our assumption we need to take $p$ with $p+1$ divisible by $n!$.
The $p$-alcoves have form $[\frac{(p+1)a'}{b'}+s+1, \frac{(p+1)a}{b}-s-1]$, where
$\frac{a'}{b'}<\frac{a}{b}$ are rational numbers with denominators between $2$ and
$n$ such that $(\frac{a'}{b'},\frac{a}{b})$ has no rational numbers with these
denominators.
\end{Ex}

Now we discuss a connection between the $p$-alcoves and the real alcoves.
Recall that $p$ is very large. Take (the real form of) a $p$-alcove $\tilde{A}_\R$ and divide it
by $p$. There is a unique real alcove $A$ such that the volumes of
both $A\setminus (\tilde{A}_\R/p)$ and $(\tilde{A}_\R/p)\setminus A$ are small
comparing to $p$. This defines a bijection. We write $^p\! A$
for the unique $p$-alcove corresponding to $A$. Under this correspondence,
to each half space $\Upsilon$ associated to $A$ we can assign a unique
half space associated to $^p\! A$ that is a shift of $\Upsilon$. We denote it
by $^p\!\Upsilon$. The assignment $\Upsilon\mapsto \,^p\!\Upsilon$ is injective but may
fail to be surjective. One can describe the image as follows.
The image of $\Upsilon\mapsto\,
^p\!\Upsilon$ consists precisely of the associated half spaces $\tilde{\Upsilon}$
satisfying the following condition:
the volume of $\dim \param-1$-dimensional polytope
$\partial\tilde{\Upsilon}\cap\, ^p\! \overline{A}$ is of order $p^{\dim \param-1}$.
Here $\partial\tilde{\Upsilon}$ is the boundary hyperplane and
$^p\! \overline{A}$ stands for the closure.

On the other hand, to each half-space $\tilde{\Upsilon}$ associated to $^p\!A$ we can assign
a unique half-space $^{\R}\!\tilde{\Upsilon}$ with the following properties:
\begin{itemize}
\item $A\subset \,^{\R}\!\tilde{\Upsilon}$,
\item $\,^{\R}\!\tilde{\Upsilon}$ is obtained from $\tilde{\Upsilon}$
by a shift,
\item The boundary of $A$ intersects the boundary hyperplane of $\,^{\R}\!\tilde{\Upsilon}$.
\end{itemize}

\begin{Ex}\label{Ex:alcove_weird}
We use the setting of Example \ref{Ex:chamber_weird}. Consider the real alcove
$A$ given by $\alpha_1^\vee>0, \alpha_2^\vee>0, \alpha_1^\vee+\alpha_2^\vee<1$.
The corresponding $p$-alcove $^p\!A$ is given by $\alpha_1^\vee>0, \alpha_2^\vee>0,
\alpha_1^\vee+\alpha_2^\vee>2$ and $\alpha_1^\vee+\alpha_2^\vee<p-2$. There are four
associated half spaces for $^p\! A$. All but one, $\alpha_1^\vee+\alpha_2^\vee>2$,
correspond to the associated half spaces of $A$. For this half-space $\tilde{
\Upsilon}$ we have $\R^!\tilde{\Upsilon}=\{\lambda| \langle\alpha_1^\vee+\alpha_2^\vee,\lambda\rangle>0\}$.
\end{Ex}


\begin{Rem}\label{Rem_p_choice}
The condition that $p+1$ is divisible by the denominators of all elements in $\tilde{\Sigma}_\Gamma$
is motivated by the following: we want to have in the same order as the real alcoves, compare
Example \ref{Ex_p_alcoves_Hilb} and Example \ref{Ex_real_alcoves_Hilb}.
\end{Rem}

\subsection{Compatible elements}\label{SS_compat_elements}
In this section we will choose a finite collection of points in $\paramq_\Q$ with a
prescribed behavior  mod $p$.

Namely, pick a real alcove $A$ and its face $\Theta$. Let $\Upsilon_1,\ldots,\Upsilon_k$
be all half spaces of the form $\,^\R\! \tilde{\Upsilon}$, where $\tilde{\Upsilon}$
is a half space associated to $^p A$ that contain $\Theta$. Let $\tilde{\Upsilon}_1,\ldots,
\tilde{\Upsilon}_k$ be the corresponding associated half spaces of $^p\!A$.
Let $\tilde{\Upsilon}_1',\ldots,\tilde{\Upsilon}'_\ell$ be the remaining associated subspaces
of $^p\! A$.  For instance, in Example \ref{Ex:alcove_weird}, we can take $\Theta=\{0\}$ and
we get $k=3,\ell=1$ with $\tilde{\Upsilon}_1,\tilde{\Upsilon}_2,\tilde{\Upsilon}_3$ being $\alpha_1^\vee>0, \alpha_2^\vee>0,
\alpha_1^\vee+\alpha_2^\vee>2$, and $\tilde{\Upsilon}'_1$ is $\alpha_1^\vee+\alpha_2^\vee<p-2$.

For $\tilde{\lambda}\in \,^p\! A$ we can define numbers $d_i(\tilde{\lambda}),
d'_j(\tilde{\lambda})$ as follows. Let $\Gamma_i$ be the classical wall parallel
to the boundary hyperplane $\tilde{\Gamma}_i$ of $\tilde{\Upsilon}_i$. Set $\alpha_i=\pm \alpha_{\Gamma_i}$,
where the sign is chosen so that $\alpha_i$ is bounded from below on $\tilde{\Upsilon}_i$.
We set $$d_i(\tilde{\lambda}):=\alpha_i(\tilde{\lambda})-\alpha_i|_{\tilde{\Gamma}_i},$$
this is a positive number. Define $d_j'(\tilde{\lambda})$ in a similar way.


Now let $\lambda\in \param_{\Q}$. Below in this section we consider $p$ that is large enough and is such that
$(p+1)\lambda\in \param_{\Z}$ and $p+1$ is divisible by the denominators of all elements
in $\tilde{\Sigma}_\Gamma$ for all $\Gamma$.

\begin{defi}\label{defi:compat_elements}
We say that $\lambda$ is compatible with $(A,\Theta)$  if there is $^p\lambda\in \,^p\!A$
depending on $p$ in an affine way such that, for all $p\gg 0$ satisfying the congruence conditions
in the previous paragraph, we have
\begin{itemize}
\item[(i)] $\,^p\lambda=\lambda+p\mu$, where $\mu\in \lambda+ \param_{\Z}$,
\item[(ii)] the numbers $d_i(\,^p\!\lambda)$ are independent
of $p$ for all $i=1,\ldots,k$,
\item[(iii)] the numbers $d_j'(\,^p\!\lambda)$ are affine
functions in $p$ with nonzero linear coefficient (that is automatically a positive number)
for all $j=1,\ldots,\ell$.
\end{itemize}
\end{defi}

The conditions (ii) and (iii) say, in particular, that $\,^p\lambda$ is close to the hyperplanes
$\tilde{\Gamma}_i$ for $i=1,\ldots,k$ but is far from the hyperplanes $\tilde{\Gamma}_j'$
for all $j=1,\ldots,\ell$.

\begin{Lem}\label{Lem:alcove_compat}
The following claims are true:
\begin{enumerate}
\item
For every $(A,\Theta)$, there is $\lambda$ compatible with $(A,\Theta)$.
\item There is a finite subset $\Lambda\in \param_{\Q}$ such that for every
$(A,\Theta)$, there is an element $\lambda\in \Lambda$ compatible
with $(A,\Theta)$.
\end{enumerate}
\end{Lem}
\begin{proof}
We prove (1).
We will be looking for $^p\lambda$ in the form $\lambda+p\mu$ for $\mu$ such that
\begin{equation}\label{eq:cond1}\lambda-\mu\in \param_{\Z}\end{equation}
so that (i) holds.

Let us see what the condition $^p\lambda\in \param_{\Z}$ means in terms of $\lambda,\mu$.
Recall that we have chosen $p$ so that $(p+1)\lambda\in \param_\Z$. So $\lambda+p\mu=
(p+1)\lambda+ p(\mu-\lambda)$. This expression is integral for $p$ satisfying the congruence conditions
from Section \ref{SS_alcoves} provided (\ref{eq:cond1}) holds.

Let $m_i\in \mathbb{Q}$ (resp. $m_j'$) be such that $\Upsilon_i$ (resp.
$\Upsilon_j'$) is given by $\alpha_i\geqslant m_i$ (resp., $\alpha_j'\geqslant m_j'$).
Then the inequalities for $\tilde{\Upsilon}_i$ (resp., $\tilde{\Upsilon}_j'$) are of
the form $\alpha_i\geqslant pm_i+\tilde{m}_i$ for a suitable element
$\tilde{m}_i\in m_i+\Z$ (resp., $\alpha_j'\geqslant pm_j'+\tilde{m}_j'$).

In (ii) we have
\begin{align*}
d_i(\,^p\!\lambda)=p(\langle\alpha_{i},\mu\rangle-m_i)
+(\langle\alpha_{i},\lambda\rangle-\tilde{m}_i).
\end{align*}
So (ii) means
\begin{equation}\label{eq:cond2}
\langle\alpha_{i},\mu\rangle=m_i,\, \langle\alpha_{i},\lambda\rangle>\tilde{m}_i.
\end{equation}
Similarly (iii) means that
\begin{equation}\label{eq:cond3}
\langle\alpha_{j}',\mu\rangle>m_j'.
\end{equation}
So for $\mu$ we can take any rational point lying inside $\Theta$, while
for $\lambda$ we can take any element in $\mu+\param_{\Z}$ such that
$\langle\alpha_{i},\lambda\rangle>\tilde{m}_i$. It is clear by the choice
of $\alpha_{1},\ldots,\alpha_{k}$ that such $\lambda$ exists.
This finishes the proof of (1).

Let us prove (2). For this we just need to notice that
an element compatible with $(A,\Theta)$ is also compatible
with $(A+\chi,\Theta)$ for every $\chi\in\param_{\Z}$. The set of orbits for
the corresponding action of $\param_\Z$ on the set of pairs $(A,\Theta)$
is finite, which finishes the proof.
\end{proof}

\begin{Rem}\label{Rem:compat}
Let  $\lambda$ be compatible with $(A,\Theta)$ and let  $\chi\in \param_{\Z}$
satisfy $\langle \alpha_{i},\chi\rangle\geqslant 0$.  Then $\lambda+\chi$
is also compatible with $(A,\Theta)$.

Also let $A^-$ denote the alcove opposite to $A$ with  respect to $\Theta$.
Let $\lambda,\lambda^-$ be elements compatible with $(A,\Theta),(A^-,\Theta)$.
We can assume that $\lambda^--\lambda\in \param_{\Z}$.
\end{Rem}

\begin{Ex}
Suppose that we are in the setting of Example \ref{Ex_p_alcoves_Hilb}. Pick a real alcove $A=(a/b,a'/b')$.
Then, for any $m>s$, the element $a/b+m$ is compatible with $(A,\{a/b\})$ and the
element  $a'/b'-m$ is compatible with $(A,\{a'/b'\})$.
\end{Ex}

\section{$\Ring$-forms}\label{S_R_forms}
Below in this section $\Ring$ denotes a finite localization of $\Z$.

\subsection{Assumptions}\label{SS_assumptions}
\subsubsection{Assumptions on $X,Y$}\label{SSS_assum_XY}
We still assume that the formal slice to any symplectic leaf in $Y$ is conical, see Section \ref{SSS_conical_slices},
and that $X$ comes with a Hamiltonian action of a torus $T$ with finitely many fixed points.

We further assume that $X,Y,\rho:X\rightarrow Y$ and the $\C^\times\times T$-action are defined over some
finite localization $\Ring$ of $\Z$.  We have a natural homomorphism $\operatorname{Pic}(X_\Ring)\xrightarrow{\sim} \operatorname{Pic}(X)$
by changing the base. We assume that it becomes an isomorphism after tensoring with
$\mathbb{Q}$. We define the sublattice $\param'_\Z$  to be the image of
$\operatorname{Pic}(X_\Ring)$ in $\param_\Z$.

These two assumptions clearly hold in our examples of
the cotangent bundles to flag varieties and Hilbert schemes
(and, more generally, for Nakajima quiver varieties of affine type A
and Slodowy varieties associated to principal Levi nilpotent elements).

Localizing $\Ring$ further, we may assume that $H^i(X_\Ring,O)=0$ for $i>0$ and $H^0(X_\Ring,O)=\Ring[Y]$,
this algebra is automatically flat over $\Ring$.

Now we proceed to the deformations.
Since $H^i(X_{\Q},O)=0$, we see that the deformations $X_{\param},Y_{\param}$ are defined over $\Q$.
Replacing $\Ring$ with a finite localization, we achieve that $\param_{\Ring}:=H^2(X,\Ring)$ is an $\Ring$-lattice in $\param_{\Q}$.
We can pick $\Ring$-forms $X_{\param,\Ring},Y_{\param,\Ring}$ flat over $\param_\Ring$.
The schemes $X_{\param,\Ring},Y_{\param,\Ring}$ come with  contracting $\mathbb{G}_m$-actions
and also with  Hamiltonian $T_\Ring$-actions
(perhaps after a finite localization of $\Ring$). We assume that all $T$-fixed  points in $X$
are defined over $\Q$. We also assume that the fixed point subvariety in $X_{\param}$ is defined
over $\Q$ (and so are the isomorphisms of connected components of $X_{\param}^T$
with $\param$).   It is easy to see that these assumptions hold for $X=T^*\mathcal{B}$ and for $X=\operatorname{Hilb}_n(\mathbb{A}^2)$ (and, more generally, for all Slodowy varieties associated
to principal Levi nilpotent elements and for all quiver varieties of affine type $A$).

In particular, for a generic one-parameter subgroup $\nu$, we see that
$\Ca_{\nu}(O_{X_{\Q}})$ is the structure sheaf of $X_{\Q}^{T_{\Q}}$
and, similarly, $\Ca_{\nu}(O_{X_{\param,\Q}})$ is the structure sheaf
of $X_{\param,\Q}^{T_\Q}$. After a finite localization of $\Ring$, we achieve
that these properties hold over $\Ring$. We can also achieve that $\Ca_\nu(\Ring[Y_{\paramq}])$
is a finitely generated module over $\Ring[\paramq]$.

\subsubsection{Assumption (S) and set $\Lambda$}\label{SSS_Ring_assumptions}
One important assumption we are going to make is as follows.

\begin{itemize}
\item[(S)] For each classical wall $\Gamma$, the set $\Sigma_\Gamma$
defined in Section \ref{SSS_essential} consists of rational numbers.
\end{itemize}

Thanks to Examples \ref{Ex:cotangent_S} and \ref{Ex:Gieseker_S}, (S) holds
in the case when $X=T^*(G/B)$ and $X=\operatorname{Hilb}_n(\C^2)$. More generally,
one can show that (S) holds for Slodowy varieties and Nakajima quiver varieties.

Now we pick a finite set of elements $\Lambda\subset
\paramq_{\Q}$. 
We choose $\Lambda$
so that for every pair $(A,\Theta)$ of a real alcove and its face and every coset in $\param_\Z/\param'_\Z$, there is an element $\lambda\in \Lambda$ in that coset that is compatible
with $(A,\Theta)$. Let $\tilde{\Upsilon}_1,\ldots,\tilde{\Upsilon}_k$ have the same
meaning as in Section \ref{SS_compat_elements}, $\tilde{\Gamma}_1,\ldots,
\tilde{\Gamma}_k$ be their boundary hyperplanes, and let $\Gamma_1,\ldots,\Gamma_k$
be the classical walls. Thanks to Remark \ref{Rem:compat} and results recalled in Section \ref{SS_Perv} we can, in addition, assume that the following hold:
\begin{itemize}
\item[(a)] Each $\lambda\in \Lambda$ lies in a quantum chamber $C^{q}$ shifted from an integral chamber $C^{int}$ such that $\bigcap_{i=1}^k \Gamma_i$ (where the meaning of $\Gamma_i$'s is as above) cuts a face of $C^{int}$.
\item[(b)] For each $\lambda\in \Lambda$ compatible with $(A,\Theta)$, there is
$\lambda^-\in \Lambda$ compatible with $(A^-,\Theta)$ such that
\begin{itemize}
\item[(b1)]
$\lambda^-=\lambda+\chi$ for  $\chi\in \param'_\Z$,
\item[(b2)] $\lambda^-$ lies in the  quantum chamber $C^{q,-}$ shifted from
$C^{int,-}$ (the integral chamber that is opposite to $C^{int}$ with respect to the face defined
by $\langle\alpha_{\Gamma_i},\bullet\rangle=0, i=1,\ldots,k$),
\item[(b3)] and $\WC_{\lambda^-\leftarrow \lambda}$ is a perverse derived equivalence
$D^b(\A_\lambda\operatorname{-mod})\xrightarrow{\sim} D^b(\A_{\lambda^-}\operatorname{-mod})$.
\end{itemize}
\end{itemize}

We can assume that $\Z[1/N!]\subset \Ring$,
where $N$ is the maximum of the denominators of elements of all sets $\Sigma_\Gamma$
and of all elements in $\Lambda$.

Now we define two finite subsets $\mathcal{P}_1,\mathcal{P}_2\subset \operatorname{Pic}(X_\Ring)$.
For $\mathcal{P}_1$ we take a finite collection of line bundle whose images
include all $\chi$ that appear in (b1) of
Section \ref{SSS_Ring_assumptions}. We can assume that  $-\mathcal{P}_1=\mathcal{P}_1$.  For $\mathcal{P}_2$ a finite collection of line bundles
that map to the union of generating sets for the
monoids $C_{\Z}=C\cap \param'_{\Z}$, where $C$ runs over all possible classical chambers.
Note that $\mathcal{P}_2$ contains a generating set of
$C^{int}_{\Z}$ for every integral chamber $C^{int}$.





\subsubsection{Assumptions on quantizations}
We assume that (S) is satisfied for quantizations of $X$ (over $\C$).

We suppose that the canonical quantization $\A_{\paramq,\Q}^\theta$ of $X_{\param,\Q}^\theta$
has an $\Ring$-form $\A_{\paramq,\Ring}^\theta$ that is a microlocal quantization of $X_{\param,\Ring}$. Let us write $\A_{\paramq,\Ring}$ for the algebra of global sections. Note that $\A_{\paramq}=\C\otimes_\Ring \A_{\paramq,\Ring}$ and $\A_{\paramq}^\theta=\C\hat{\otimes}_{\Ring}\A_{\paramq,\Ring}^\theta$
(where $\hat{\otimes}$ stands for the completed tensor product with respect to the induced filtration).

This clearly holds in the example of $T^*\mathcal{B}$, where the quantization $\A^\theta_{\paramq,\Ring}$ is obtained as
(the microlocalization of) $\mathcal{D}_{G/U,\Ring}^{T_\Ring}$. This also holds for the Hilbert schemes, where
the quantization $\A^{\theta}_{\paramq,\Ring}$ is obtained as
$$\left([D(V_\Ring)/D(V_{\Ring})\Phi([\g_\Ring,\g_{\Ring}])]|_{T^*V_\Ring^{\theta-ss}}\right)^{G_\Ring}.$$
More generally, the conditions hold for the quantizations of Slodowy varieties and of Nakajima quiver varieties.

Now let us discuss the homological dimension for the algebras $\A_{\lambda,\Ring}$.

\begin{Lem}\label{Lem:fin_hom_dim1}
After a finite localization of $\Ring$, for all $\lambda\in \Lambda$,
the algebra $\A_{\lambda,\Ring}$ has finite homological dimension.
\end{Lem}
\begin{proof}
The algebra $\A_\lambda$ has finite homological dimension. This is because,
by the construction of $\Lambda$, see (a) in Section
\ref{SSS_Ring_assumptions}, $\lambda\in \paramq^{reg}$. To deduce that $\A_{\lambda,\Ring}$ has finite homological
dimension (after a finite localization of $\Ring$) we can argue as in the proof of
\cite[Lemma 7.8]{BL}.
\end{proof}

\subsection{Forms of translation bimodules}
We will further replace $\Ring$ with its finite localization so that the translation bimodules
$\A_{\Ring,\chi}$ have good properties as explained below in this section.

Recall the finite sets $\mathcal{P}_1,\mathcal{P}_2\subset \operatorname{Pic}(X_\Ring)$ from
Section \ref{SSS_Ring_assumptions}.
Set $\mathcal{P}=\mathcal{P}_1\cup \mathcal{P}_2$.

As above, we can quantize the line bundles $O(\chi)$ on $X_{\param,\Ring}$ getting bimodules $\A^\theta_{\paramq,\chi,\Ring}$. We note that we have natural homomorphisms
$$\A^\theta_{\paramq,\chi_2,\Ring}\otimes_{\A^\theta_{\paramq,\Ring}}\A^\theta_{\paramq,\chi_1,\Ring}
\rightarrow \A^\theta_{\paramq,\chi_1+\chi_2,\Ring},\quad
\Gamma(\A^\theta_{\paramq,\chi_2,\Ring})\otimes_{\A_{\paramq,\Ring}}\Gamma(\A^\theta_{\paramq,\chi_1,\Ring})
\rightarrow \Gamma(\A^\theta_{\paramq,\chi_1+\chi_2,\Ring}).$$

Note that the base changes of $\Gamma(\A^\theta_{\paramq,\chi,\Ring})$ to $\Q$ are independent of
$\theta$ for the same reason as the $\C$-forms are independent of $\theta$.
Each of the bimodules
$\Gamma(\A^\theta_{\paramq,\chi,\Ring})$ is finitely generated over $\A_{\paramq,\Ring}$. It follows
that after a finite localization of $\Ring$, the bimodule  $\Gamma(\A^\theta_{\paramq,\chi,\Ring})$
is independent of $\theta$ for any $\chi\in \mathcal{P}$. We denote this bimodule by $\A_{\paramq,\chi,\Ring}$. Note that
$\C\otimes_{\Ring}\A_{\paramq,\chi,\Ring}=\A_{\paramq,\chi}$.

Note that each  bimodule $\A_{\paramq,\chi,\Ring},\chi\in \mathcal{P},$ is Harish-Chandra. From here one deduces
that after a finite localization of $\Ring$ we can achieve that $\A_{\paramq,\chi,\Ring}$ is flat over $\Ring$.
This is a consequence of the following more general result that is proved completely analogously to \cite[Lemma 3.5]{BL}.

\begin{Lem}\label{Lem:gen_flatness}
Let $M$ be a finitely generated $\A_{\paramq,\Ring}$-module.
Then after a finite localization of $\Ring$, the $\Ring$-module
$M$ becomes flat.
\end{Lem}

\begin{Cor}\label{Cor:translation_flatness}
Perhaps, after a finite localization of $\Ring$, we achieve that $\A_{\lambda,\chi,\Ring}$
is flat over $\Ring$ for all $\lambda\in \Lambda, \chi\in \mathcal{P}_1$.
\end{Cor}

We will also need some properties of bimodules $\A_{\paramq,\chi,\Ring}$ for $\chi\in \mathcal{P}_2$
and $\A_{\lambda,\chi,\Ring}$, where $\lambda\in \Lambda$ and $\lambda^-=\lambda+\chi$.

For $\chi\in \mathcal{P}_2$, let $\paramq^{reg[\chi]}$ denote the subset of $\lambda\in \param$ such that
$\lambda,\lambda+\chi$ lie in the same quantum chamber. By the definition of quantum chambers, $\paramq^{reg[\chi]}$
is the complement to finitely many essential
hyperplanes in $\paramq$. Also note that $\paramq^{reg[-\chi]}=
\paramq^{reg[\chi]}+\chi$. Consider the bimodules $\A_{\paramq^{reg[\chi]},\chi,\Ring}:=
\A_{\paramq,\chi,\Ring}\otimes_{\Ring[\paramq]}\Ring[\paramq^{reg[\chi]}]$ (this is
a $\A_{\paramq^{reg[-\chi]},\Ring}$-$\A_{\paramq^{reg[\chi]},\Ring}$-bimodule) and
$\A_{\paramq^{reg[-\chi]},-\chi,\Ring}$.

\begin{Lem}\label{Lem:Morita_Ring}
After a finite localization of $\Ring$, the bimodules
$\A_{\paramq^{reg[\chi]},\chi,\Ring}, \A_{\paramq^{reg[-\chi]},\chi,\Ring}$
are mutually inverse Morita equivalence bimodules for all $\chi\in \mathcal{P}_2$.
\end{Lem}
\begin{proof}
By Corollary \ref{Cor:trans_equiv}, $\A_{\lambda,\chi},\A_{\lambda+\chi,-\chi}$ are mutually
dual Morita equivalence bimodules for $\lambda\in \paramq^{reg[\chi]}$. The natural homomorphisms
$$\A_{\lambda+\chi,-\chi}\otimes_{\A_{\lambda+\chi}}\A_{\lambda,\chi}\rightarrow \A_{\lambda},
\A_{\lambda,\chi}\otimes_{\A_{\lambda}}\A_{\lambda+\chi,-\chi}\rightarrow \A_{\lambda+\chi}$$
are obtained by specialization to $\lambda$ from
\begin{equation}\label{eq:bimod_homom_reg}
\begin{split}
&\A_{\paramq^{reg[-\chi]},-\chi}\otimes_{\A_{\paramq^{reg[-\chi]}}}\A_{\paramq^{reg[\chi]},\chi}\rightarrow
\A_{\paramq^{reg[\chi]}},\\
&\A_{\paramq^{reg[\chi]},\chi}\otimes_{\A_{\paramq^{reg[\chi]}}}\A_{\paramq^{reg[-\chi]},-\chi}\rightarrow
\A_{\paramq^{reg[-\chi]}}.
\end{split}
\end{equation}
It follows that homomorphisms (\ref{eq:bimod_homom_reg}) are isomorphisms. These homomorphisms
are defined over $\Ring$. It follows that the kernels and the cokernels of
\begin{align*}
&\A_{\paramq^{reg[-\chi]},-\chi,\Ring}\otimes_{\A_{\paramq^{reg[-\chi]},\Ring}}\A_{\paramq^{reg[\chi]},\chi,\Ring}\rightarrow
\A_{\paramq^{reg[\chi]},\Ring},\\
&\A_{\paramq^{reg[\chi]},\chi,\Ring}\otimes_{\A_{\paramq^{reg[\chi]},\Ring}}\A_{\paramq^{reg[-\chi]},-\chi,\Ring}\rightarrow
\A_{\paramq^{reg[-\chi]},\Ring}\end{align*}
are $\Ring$-torsion. A direct analog of Lemma \ref{Lem:gen_flatness} holds for the algebras of the form
$$\A_{\paramq^{reg[\chi]},\Ring}\otimes_{\Ring}
\A_{\paramq^{reg[-\chi]},\Ring}^{opp}$$ and so the kernels and the cokernels vanish after a
finite localization of $\Ring$.
\end{proof}

%


Now pick $\lambda\in \Lambda$ and let $\lambda^-\in\Lambda, \chi\in \mathcal{P}_1$ be as before.
The bimodule $\A_{\lambda,\chi}$ is defined over $\Q$. It follows that the filtrations
by Serre subcategories on $\A_{\lambda}\operatorname{-mod}, \A_{\lambda+\chi}\operatorname{-mod}$
that make $\WC_{\lambda^-\leftarrow \lambda}$ perverse are defined over $\Q$. Recall,
Section \ref{SS_Perv}, that these filtrations come from chains of ideals,
say $\{0\}=\I^0_{?}\subset \I^1_?\subset\ldots \subset \I^k_?\subset \A_{?}$, where
$?=\lambda$ or $\lambda^-$. So the ideals are defined over $\Q$ as well.
We set $\I^j_{?,\Ring}=\I^j_{?}\cap \A_{?,\Ring}$ so that $\I^j_?=\C\otimes_{\Ring}\A_{?,\Ring}$.


\begin{Lem}\label{Lem:perv_Ring}
After a finite localization of $\Ring$ we achieve that for all $\lambda\in \Lambda$
the following hold:
\begin{enumerate}
\item  $\I^j_{\lambda,\Ring}=(\I^j_{\lambda,\Ring})^2$ for all $j$.
\item  The functor $\A_{\lambda,\chi,\Ring}\otimes^L_{\A_{\lambda,\Ring}}\bullet:
D^b(\A_{\lambda,\Ring}\operatorname{-mod})\rightarrow D^b(\A_{\lambda+\chi,\Ring}\operatorname{-mod})$
is a perverse derived equivalence with respect to the filtrations given by
the ideals $\I^j_{\lambda,\Ring},\I^j_{\lambda+\chi,\Ring}$.
\end{enumerate}
\end{Lem}
\begin{proof}
We have $\I^j_{\lambda}=(\I^j_\lambda)^2$. From here we get $\I^j_{\lambda,\Ring}=(\I^j_{\lambda,\Ring})^2$
similarly to the proof of Lemma \ref{Lem:Morita_Ring}. Let us proceed to (2).

The functor $\A_{\lambda,\chi,\Ring}\otimes^L_{\A_{\lambda,\Ring}}\bullet$ is an equivalence
if and only if
$$\A_{\lambda+\chi,\Ring}\xrightarrow{\sim}R\operatorname{End}_{\A_{\lambda,\Ring}}(\A_{\lambda,\chi,\Ring}),\quad
\A_{\lambda,\chi,\Ring}\otimes^L_{\A_{\lambda+\chi,\Ring}}R\operatorname{Hom}(\A_{\lambda,\chi,\Ring},
\A_{\lambda,\Ring})\xrightarrow{\sim}\A_{\lambda,\Ring}$$
We know that these homomorphisms become iso after base change to $\C$. So they are iso
after a finite localization of $\Ring$.

We also know that (b1)-(b6) from Section \ref{SS_Perv} hold over $\C$.
From here we deduce that (b1)-(b6) hold over $\Ring$, perhaps after a
finite localization of $\Ring$. Now that implies that
$\A_{\lambda,\chi,\Ring}\otimes^L_{\A_{\lambda,\Ring}}\bullet$  is perverse as required.
\end{proof}

\subsection{Forms of Verma  modules}\label{SS_Verma_forms}
The goal of this and the next sections is to show that most results and constructions that we have for the
categories $O$ over $\C$ still carry over to $\Ring$, perhaps after some finite localization
of $\Ring$.

Recall the principal open subset $\paramq^{reg}\subset \paramq$ from
Definition \ref{defi:regular_parameters}. By our choice of $\Ring$, $\paramq^{reg}$
is defined over $\Ring$. Let $\paramq^{reg}_{\Ring}$ denote the natural $\Ring$-form
of $\paramq^{reg}$.

Recall, Section \ref{SSS_assum_XY}, that $\Ca_{\nu}(O_{X_{\paramq,\Ring}})\cong \Ring[\paramq][X^T]$.
Replacing $\Ring$ with its finite localization we achieve that the sheaf $\Ca_{\nu}(\A^\theta_{\paramq,\Ring})$ is a filtered deformation of $\Ca_{\nu}(O_{X_{\paramq,\Ring}})$.
This follows because the filtration on $\Ca_{\nu}(\A^\theta_{\paramq})$ is defined over $\Ring$,
$\Ca_{\nu}(O_{X_{\paramq,\Ring}})\cong \Ring[\paramq][X^T]\rightarrow \gr\Ca_{\nu}(\A^\theta_{\paramq,\Ring})$ and this homomorphism becomes an isomorphism after
tensoring with $\C$.

We get a homomorphism
$\Ca_{\nu}(\A_{\paramq,\Ring})\rightarrow \Ring[\paramq][X^T]$
similarly to the case of $\C$. The $\Ring[\param]$-module $\Ca_{\nu}(\Ring[X_{\param}])$
is finitely generated for the same reason as in the case of $\C$.
From $\Ca_{\nu}(\Ring[X_{\param}])\twoheadrightarrow
\gr \Ca_{\nu}(\A_{\paramq,\Ring})$ we deduce that
$\Ca_{\nu}(\A_{\paramq,\Ring})$ is a finitely generated $\Ring$-module.

Since the homomorphism $\Ca_{\nu}(\A_{\paramq^{reg},\Ring})\rightarrow \Ring[\paramq^{reg}][X^T]$
is an isomorphism after a base change to $\C$, it is still an isomorphism after a finite
localization of $\Ring$. We assume from now on that this is an isomorphism.

Thanks to the homomorphism $\Ca_{\nu}(\A_{\paramq,\Ring})\rightarrow \Ring[\paramq][X^T]$
we can define Verma module $\Delta_{\nu,\paramq,\Ring}(\mathsf{x})$ for $\mathsf{x}\in X^T$ similarly
to the case of $\C$. Clearly, $\C\otimes_{\Ring}\Delta_{\nu,\paramq,\Ring}(\mathsf{x})=\Delta_{\nu,\paramq}(\mathsf{x})$.
By Lemma \ref{Lem:gen_flatness}, after a finite localization of $\Ring$ we can assume that
$\Delta_{\nu,\paramq,\Ring}$ is flat over $\Ring$. It follows that all graded components of $\Delta_{\nu,\paramq,\Ring}(\mathsf{x})$ are finitely generated  $\Ring[\paramq]$-modules.

\begin{Lem}\label{Lem:Verma_R_flatness}
After a finite localization of $\Ring$, each  $\Delta_{\nu,\paramq^{reg},\Ring}(\mathsf{x})$
is projective over $\Ring[\paramq^{reg}]$.
\end{Lem}
\begin{proof}
We can filter the algebra $\A_{\paramq,\Ring}$ (with an ascending $T_\Ring$-stable $\Z_{\geqslant 0}$-filtration) in such a way that $\Ring[\paramq]$ is in degree $0$ and $\gr\A_{\paramq,\Ring}=\Ring[\paramq][X]$,
compare to \cite[Lemma 3.5]{BL}.
Equip $\Delta_{\nu,\paramq^{reg},\Ring}(\mathsf{x})$ with a compatible $T_{\Ring}$-stable good filtration.
Then $\gr \Delta_{\nu,\paramq^{reg},\Ring}(\mathsf{x})$ is a finitely generated $R[\paramq^{reg}][X]$-module.
Hence, by generic flatness, there is a finite localization $\Ring[\paramq^{reg}]^0$ of
$\Ring[\paramq^{reg}]$ such that  $\Ring[\paramq^{reg}]^0\otimes_{\Ring[\paramq^{reg}]}\gr \Delta_{\nu,\paramq^{reg},\Ring}(\mathsf{x})$ is flat over $\Ring[\paramq^{reg}]^0$. It follows that
$\Ring[\paramq^{reg}]^0\otimes_{\Ring[\paramq]}\Delta_{\nu,\paramq^{reg},\Ring}(\mathsf{x})$ is flat over $\Ring[\paramq^{reg}]^0$. So every graded component of $\Delta_{\nu,\paramq^{reg},\Ring}(\mathsf{x})$
becomes flat (hence projective) after localization to $\Ring[\paramq^{reg}]^0$.

Now let $\mathsf{S}$ be a quotient of $\Ring[\paramq^{reg}]$.
An argument similar to the previous paragraph show that
the specialization $$\mathsf{S}\otimes_{\Ring[\paramq^{reg}]}\Delta_{\nu,\paramq^{reg},\Ring}(\mathsf{x})$$
becomes flat after a finite localization
of the quotient $\mathsf{S}$.

We deduce that one can split $\paramq^{reg}_{\Ring}$ into a disjoint union locally closed irreducible
subschemes in such a way that the restriction of any graded component to any of the subschemes
is projective.  On the other hand, we know that  $\Q\otimes_{\Ring}\Delta_{\nu,\paramq^{reg},\Ring}(\mathsf{x})$ is flat over $\Q[\paramq^{reg}]$ (because the similar statement holds over $\C$). Combining these two observations,
we see that after a finite localization of $\Ring$  each graded component
of $\Delta_{\nu,\paramq^{reg},\Ring}(\mathsf{x})$ is flat over $\Ring[\paramq^{reg}]$.
\end{proof}

Now let us study translations of Verma modules. We can consider equivariant Verma modules
$\Delta_{\nu,\paramq,\Ring}(\mathsf{x},\kappa)$ and equivariant lifts
$\A_{\paramq,\chi,\Ring}$ with $\chi\in \mathcal{P}_1,$
as before, see Section \ref{SSS_equiv_K_0_behavior}. Recall that
for $\chi\in \operatorname{Pic}^T(X)$, we write
$\mathsf{wt}_\chi(\mathsf{x})$ for the weight of $T$ in the
fiber $O(\chi)$ at $\mathsf{x}$.

\begin{Lem}\label{Lem:Verma_transl_global}
After a finite localization of $\Ring$, we have a $T_\Ring$-equivariant
isomorphism
\begin{equation}\label{eq:Verma_transl_global}
\Delta_{\nu,\paramq^{reg[-\chi]},\Ring}(\mathsf{x},\kappa+\mathsf{wt}_\chi(\mathsf{x}))\cong \A_{\paramq^{reg[\chi]},\chi,\Ring}\otimes_{\A_{\paramq^{reg[\chi]},\Ring}}
\Delta_{\nu,\paramq^{reg[\chi]},\Ring}(\mathsf{x},\kappa).\end{equation}
\end{Lem}
\begin{proof}
By (2) of Proposition
\ref{Prop:transl_Verma}, we know that
\begin{equation}\label{eq:Verma_trans_global_C}
\Delta_{\nu,\paramq^{reg[-\chi]}}(\mathsf{x},\kappa+\mathsf{wt}_\chi(\mathsf{x}))\cong \A_{\paramq^{reg[\chi]},\chi}\otimes_{\A_{\paramq^{reg[\chi]}}}\Delta_{\nu,\paramq^{reg[\chi]}}(\mathsf{x},\kappa).
\end{equation}

Now it is enough to show that a version of (\ref{eq:Verma_transl_global}) holds over
$\Q$ (then, since our modules are finitely generated, we can take a finite localization
of $\Ring$ to achieve (\ref{eq:Verma_transl_global})). But an isomorphism in
(\ref{eq:Verma_transl_global}) is rational if and only if its highest degree component
is rational. Again, $\Q[\paramq^{reg[\chi]}]$ is a factorial algebra, so the highest degree component
of the right hand side is a free rank one module. The analog of
(\ref{eq:Verma_transl_global}) over $\Q$ follows.
\end{proof}

\subsubsection{Tor vanishing}
We can also consider the right-handed versions of Verma modules, $\Delta^r_{-\nu,\paramq,\Ring}(\mathsf{x})$.
Straightforward analogs of Lemmas \ref{Lem:Verma_R_flatness} and \ref{Lem:Verma_transl_global} hold
for these modules.

\begin{Lem}\label{Lem:Ext_vanish_global}
After a finite localization of $\Ring$, we have
\begin{align}\label{eq:global_Tor_vanishing}
&\operatorname{Tor}^i_{\A_{\paramq^{reg},\Ring}}(\Delta^r_{-\nu,\paramq^{reg},\Ring}(\mathsf{x}), \Delta_{\nu,\paramq^{reg},\Ring}(\mathsf{x}'))\cong\Ring[\paramq^{reg}]^{\oplus \delta_{i0}\delta_{\mathsf{x},\mathsf{x}'}}.
\end{align}
\end{Lem}
\begin{proof}
Note that, for $\lambda\in \paramq^{reg}_{\C}$, we have
$$\C_{\lambda}\otimes^L_{\Ring[\paramq^{reg}]}\left(\Delta^r_{-\nu,\paramq^{reg},\Ring}(\mathsf{x})\otimes^L_{\A_{\paramq^{reg},\Ring}} \Delta_{\nu,\paramq^{reg},\Ring}(\mathsf{x}')
\right)=\Delta_{-\nu,\lambda}^r(\mathsf{x})\otimes^L_{\A_\lambda}\Delta_{\nu,\lambda}(\mathsf{x}').$$
The right hand side vanishes if $\mathsf{x}\neq \mathsf{x}'$ and is the one-dimensional space in
homological degree $0$ if $\mathsf{x}=\mathsf{x}'$.

Let us show that the left hand side of
(\ref{eq:global_Tor_vanishing}) is a finitely generated $\Ring[\paramq^{reg}]$-module.
This will follow once we check that
$\operatorname{Tor}^i_{\A_{\paramq,\Ring}}(\A_{\paramq,\Ring}/\A^{<0}_{\paramq,\Ring}\A_{\paramq,\Ring},
\A_{\paramq,\Ring}/\A_{\paramq,\Ring}\A_{\paramq,\Ring}^{>0})$ is a finitely generated
$\Ring[\paramq]$-module for all $i$. The modules $\A_{\paramq,\Ring}/\A^{<0}_{\paramq,\Ring}\A_{\paramq,\Ring}, \A_{\paramq,\Ring}/\A_{\paramq,\Ring}\A_{\paramq,\Ring}^{>0}$ are filtered. The Tor inherits a filtration
whose associated graded is a $\Ring[\paramq]$-module subquotient of
$$\operatorname{Tor}^i_{\Ring[Y_{\param}]}(\gr(\A_{\paramq,\Ring}/\A^{<0}_{\paramq,\Ring}\A_{\paramq,\Ring}),
\gr(\A_{\paramq,\Ring}/\A_{\paramq,\Ring}\A_{\paramq,\Ring}^{>0})).$$ The intersection of the supports of the arguments in $\Ring[Y_{\param}]$ is finite over $\operatorname{Spec}(\Ring[\param])$, so the latter Tor is
a finitely generated $\Ring[\param]$-module.

Now let us recall that $$\Q\otimes_{\Ring}\operatorname{Tor}^i_{\A_{\paramq^{reg},\Ring}}(\Delta^r_{-\nu,\paramq^{reg},\Ring}(\mathsf{x}), \Delta_{\nu,\paramq^{reg},\Ring}(\mathsf{x}'))$$
is zero if $i\neq 0, \mathsf{x}\neq \mathsf{x}'$ and is $\Q[\paramq^{reg}]$
else. Similarly to  the proof of Lemma \ref{Lem:Verma_R_flatness} (see the last paragraph
there, in particular), it follows that
$$\operatorname{Tor}^i_{\A_{\paramq^{reg},\Ring}}(\Delta^r_{-\nu,\paramq,\Ring}(\mathsf{x}), \Delta_{\nu,\paramq^{reg},\Ring}(\mathsf{x}'))$$
becomes flat over $\Ring[\paramq^{reg}]$ after a finite localization of $\Ring$
(depending on $i$).

It remains to check  that the algebra $\A_{\paramq^{reg},\Ring}$ has finite homological dimension
after a finite localization of $\Ring$ (so that we only have finitely many
non-vanishing Tor's). Similarly to Lemma \ref{Lem:fin_hom_dim1}, this will follow if we check
that $\A_{\paramq^{reg}}$ has finite homological dimension. By Remark \ref{Rem:hom_dim}, the homological dimension of $\A_{\paramq^{reg}}$ does not exceed
$2\dim X+\dim \param$.


So after a finite localization
of $\Ring$, (\ref{eq:global_Tor_vanishing}) holds (for all $\mathsf{x},\mathsf{x}',i$).
\end{proof}

\subsection{Forms of simple, projective, etc. objects in $\mathcal{O}$}\label{SSS_simple_proj_reductions}
Let us discuss forms of various objects  in $\mathcal{O}_\nu(\A_\lambda),\lambda\in \Lambda,$ over $\Ring$.
First of all, let us note that, by our rationality assumptions on $\A_\paramq$ and $\Ca_\nu(\A_{\paramq})$, for $\lambda\in \paramq^{reg}_{\Q}$, the category $\mathcal{O}_\nu(\A_\lambda)$ is defined over $\Q$. For $\mathsf{x}\in X^T$, let $L_{\nu,\lambda,\Q}(\mathsf{x}),
P_{\nu,\lambda,\Q}(\mathsf{x}),I_{\nu,\lambda,\Q}(\mathsf{x}),\nabla_{\nu,\lambda,\Q}(\mathsf{x})$
denote the simple, the projective, the injective and the costandard objects in $\mathcal{O}_\nu(\A_{\lambda,\Q})$
labelled by $\mathsf{x}$.

Let us pick some $T_{\Ring}$-stable $\Ring$-forms $L_{\nu,\lambda,\Ring}(\mathsf{x})\subset
L_{\nu,\lambda,\Q}(\mathsf{x}), \nabla_{\nu,\lambda,\Ring}(\mathsf{x})
\subset \nabla_{\nu,\lambda,\Q}(\mathsf{x}), P_{\nu,\lambda,\Ring}(\mathsf{x})\subset
P_{\nu,\lambda,\Q}(\mathsf{x}), I_{\nu,\lambda,\Ring}(\mathsf{x})\subset
I_{\nu,\lambda,\Q}(\mathsf{x})$.

\begin{Lem}\label{Lem:simple_proj_inj_Ring}
After a finite localization of $\Ring$, for every $\lambda\in \Lambda$, we achieve the following:
\begin{enumerate}
\item The objects $\Delta_{\nu,\lambda,\Ring}(\mathsf{x}),
\nabla_{\nu,\lambda,\Ring}(\mathsf{x})$ are filtered by
$L_{\nu,\lambda,\Ring}(\mathsf{x}')$'s. Moreover, $\Delta_{\nu,\lambda,\Ring}(\mathsf{x})
\twoheadrightarrow L_{\nu,\lambda,\Ring}(\mathsf{x})\hookrightarrow
\nabla_{\nu,\lambda,\Ring}(\mathsf{x})$.
\item The objects $P_{\nu,\lambda,\Ring}(\mathsf{x})$ are filtered by
$\Delta_{\nu,\lambda,\Ring}(\mathsf{x}')$'s and the objects
$I_{\nu,\lambda,\Ring}(\mathsf{x})$ are filtered by
$\nabla_{\nu,\lambda,\Ring}(\mathsf{x}')$'s (in the increasing
order with respect to  $\leqslant_{\nu,\lambda}$).
\item Every object $L_{\nu,\lambda,\Ring}(\mathsf{x})$ has a finite resolution
$$\ldots\rightarrow P_{1,\Ring}\rightarrow P_{0,\Ring}\rightarrow L_{\nu,\lambda,\Ring}(\mathsf{x})\rightarrow 0,$$
where each $P_{i,\Ring}$ is the direct sum of $P_{\nu,\lambda,\Ring}(\mathsf{x}')$'s. Similarly,
$L_{\nu,\lambda,\Ring}(\mathsf{x})$ has a resolution
$$0\rightarrow L_{\nu,\lambda,\Ring}(\mathsf{x})\rightarrow I_{0,\Ring}\rightarrow
I_{1,\Ring}\rightarrow \ldots$$
where each $I_{i,\Ring}$ is the direct sum of $I_{\nu,\lambda,\Ring}(\mathsf{x}')$'s.
\end{enumerate}
\end{Lem}
\begin{proof}
The proofs of (1)-(3) are similar, let us prove (2). This claim clearly holds over $\C$ and the
filtrations are defined over $\Q$. So it holds over $\Q$ as well. After a finite localization,
it holds over $\Ring$ too.
\end{proof}

\begin{Lem}\label{Lem:fin_gen_Ext}
Let $M_\Ring,N_{\Ring}$ be graded $\Ring$-lattices in
$M_\Q,N_\Q\in \mathcal{O}^T_\nu(\A_{\lambda,\Q})$. Then, for all $i$,
$\operatorname{Ext}^i_{\A_{\lambda,\Ring},T}(M_\Ring,N_{\Ring})$ is a finitely generated
$\Ring$-module.
\end{Lem}
\begin{proof}
Let us pick a free graded resolution
$$\ldots\rightarrow\A_{\lambda,\Ring}\otimes_{\Ring} V_1\rightarrow
\A_{\lambda,\Ring}\otimes_{\Ring} V_0\rightarrow M_{\Ring}\rightarrow 0,$$
where the $V_i$'s a finite rank graded free $\Ring$-modules. The complex
$$\ldots\rightarrow V_1^*\otimes_{\Ring}N_\Ring\rightarrow V_0^*\otimes_{\Ring}N_{\Ring}\rightarrow 0$$
is graded. The Ext's
we want to compute are the degree $0$ components homology of this complex.
All graded components of $N_{\Ring}$ are finitely generated over $\Ring$
and our claim follows.
\end{proof}

Pick  $m\in\Z$ such that
$|c_{\nu,\lambda}(\mathsf{x}_1)-c_{\nu,\lambda}(\mathsf{x}_2)|\leqslant m$
for all $\mathsf{x}_1,\mathsf{x}_2$ in the same $h$-block.

\begin{Cor}\label{Lem:Ext_freeness}
After a finite localization of $\Ring$, we have the following. Let $M_{\Ring},N_{\Ring}$
be modules of the form $L_{\nu,\lambda,\Ring}(\mathsf{x}_i,\kappa_i), \Delta_{\nu,\lambda,\Ring}(\mathsf{x}_i,
\kappa_i),\nabla_{\nu,\lambda,\Ring}(\mathsf{x}_i,
\kappa_i) , P_{\nu,\lambda,\Ring}(\mathsf{x}_i,\kappa_i)$ or $I_{\nu,\lambda,\Ring}(\mathsf{x}_i,\kappa_i),
i=1,2$, where $\mathsf{x}_1,\mathsf{x}_2$ are in the same $h$-block and $|\kappa_1-\kappa_2|\leqslant m$.
Then for all $i$, the $\Ring$-modules $\Ext^i_{\A_{\lambda,\Ring},T}(M_{\Ring},N_{\Ring})$
are free and finitely generated over $\Ring$ (and, automatically,
$\C\otimes_{\Ring}\Ext^i_{\A_{\lambda,\Ring},T}(M_{\Ring},N_{\Ring})
\xrightarrow{\sim} \Ext^i_{\A_{\lambda,\Ring},T}(M_{\C},N_{\C})$).
\end{Cor}

To finish this section let us describe the images of $\Delta_{\nu,\lambda,\Ring}(\mathsf{x},\kappa)$
under the wall-crossing functors $\A_{\lambda,\chi,\Ring}\otimes^L_{\A_{\lambda,\Ring}}\bullet$.
Here we assume that $\lambda\in \Lambda$ and $\chi\in \mathcal{P}_1$ is the corresponding element.
We choose a $T_{\Ring}$-equivariant structure on $O_\Ring(\chi)$ that gives rise to an equivariant
structure on $\A_{\lambda,\chi,\Ring}$.

\begin{Lem}\label{Lem:WC_R}
After a finite localization of $\Ring$, for all $\lambda\in \Lambda$ and $\chi\in \mathcal{P}_1$,
we have
$$\A_{\lambda,\chi,\Ring}\otimes^L_{\A_{\lambda,\Ring}}\Delta_{\nu,\lambda,\Ring}(\mathsf{x},\kappa)\cong
\nabla_{\nu,\lambda,\Ring}(\mathsf{x},\kappa+\mathsf{wt}_\chi(\mathsf{x})).$$
\end{Lem}
\begin{proof}
We have such an isomorphism over $\C$, this follows from Lemma
\ref{Lem:long_WC_Ringel}. It follows that the degree
$\kappa+\mathsf{wt}_\chi(\mathsf{x})$ component in
$\A_{\lambda,\chi,\Q}\otimes^L_{\A_{\lambda,\Q}}\Delta_{\nu,\lambda,\Q}(\mathsf{x},\kappa)$
is 1-dimensional. Also it is the irreducible $\Ca_{\nu}(\A_{\lambda,\Q})$-module corresponding
to $\mathsf{x}$. So we get a homomorphism
$$\A_{\lambda,\chi,\Q}\otimes^L_{\A_{\lambda,\Q}}\Delta_{\nu,\lambda,\Q}(\mathsf{x},\kappa)
\rightarrow \nabla_{\nu,\lambda,\Q}(\mathsf{x}, \kappa+\mathsf{wt}_\chi(\mathsf{x}))$$
unique up to proportionality. So it is an isomorphism over $\C$, hence an isomorphism.
The claim of the lemma easily follows from here.
\end{proof}

Below in the paper we assume that $\Ring$ satisfies all results in Section
\ref{S_R_forms}.

\section{Reduction to characteristic $p$}
\subsection{Assumptions on $p$}\label{SS_p_assumptions}
Let $\Ring$ be as directly above. Also recall a finite subset
$\Lambda\subset \paramq_{\mathbb{Q}}$ from
Section \ref{SSS_Ring_assumptions} and finite subsets
$\mathcal{P}_1,\mathcal{P}_2\subset \operatorname{Pic}(X_{\Ring})$ from Section
\ref{SSS_assum_XY}.

Let $p$ be a prime number that is non-invertible in $\Ring$. Let us write $\F$ for $\bar{\F}_p$.
We will be interested in  categories of  representations
of the algebras $\A_{\F,\lambda}:=\F_{\lambda}\otimes_{R[\paramq]}\A_{\paramq,R}$ for $\lambda\in \paramq_{\F}$.

We will impose several additional assumptions on $p$. First of all, we assume that $p+1$
is divisible by the denominators of all elements in the sets $\Sigma_\Gamma$ (defined in
Section \ref{SSS_essential}) for all classical walls $\Gamma$. Second, we assume that  $(p+1)\lambda\in \paramq_\Z$
for all $\lambda\in \Lambda$ for all $(p+1)\mathsf{c}_{\nu,\lambda}(\mathsf{x})\in \Z$
for all $\lambda\in \paramq_\Ring$, generic $\nu$ and all $\mathsf{x}\in X^T$.  There are infinitely many such primes.

Also we need to assume that $p$ is large enough. More precisely,
this means the following
\begin{itemize}
\item The map $A\mapsto \,^p\! A$ from Section \ref{SS_alcoves} is a bijection
between the set of real alcoves and the set of $p$-alcoves.
\item  $p$ is large enough to satisfy the conditions of the
next lemma.
\end{itemize}

\begin{Lem}\label{Lem:alcove_point_translation}
For $p$ sufficiently large, for every two points $\lambda_1,\lambda_2$ in the same
$p$-alcove $\tilde{A}$ such that $\lambda_2-\lambda_1\in \param'_{\Z}$, there is a sequence $\chi_1,\ldots,\chi_k\in \mathcal{P}_2$ such that
\begin{itemize}
\item[(*)]
The elements $\lambda_1+\sum_{i=1}^j \chi_i$
are in $\tilde{A}$ for all $j=0,\ldots,k$.
\end{itemize}
\end{Lem}
\begin{proof}
Let $\tilde{\Upsilon}_1,\ldots,\tilde{\Upsilon}_m$ be the associated half spaces of $\tilde{A}$
(Section \ref{SS_alcoves}).
Let $\langle\alpha_{i},\bullet\rangle> \hat{m}_i, i=1,\ldots,m$ be the inequality
defining $\tilde{\Upsilon}_i$. Pick a very small (but independent of $p$)
number $\epsilon>0$. Clearly, if $p$ is
large enough and $\langle\alpha_{i},\lambda_\ell\rangle> \hat{m}_i+\epsilon p,
\ell=1,2$, then there are $\chi_1,\ldots,\chi_k$ such that (*) holds.

Now consider the case when one of $\lambda_\ell$, say $\lambda_1$, satisfies
\begin{equation}\label{eq:walls}
\hat{m}_i<\langle\alpha_{i},\lambda_1\rangle\leqslant
\hat{m}_i+\epsilon p
\end{equation} for some $i$.
This condition says that $\lambda_1$ is relatively close to a wall.

Let $i=1,\ldots,r$ be precisely the indexes $i$, for (\ref{eq:walls}) holds. If $p$
is sufficiently large,  we can find a classical chamber $C$ whose interior points
are positive on all $\alpha_{1},\ldots, \alpha_{r}$. Pick the generators
$\chi'_1,\ldots,\chi'_s$ of $C_{\Z}$ that lie in $\mathcal{P}_2$.  A point of the form $\lambda+(m+1)(\chi'_1+\ldots+\chi'_r)
+m(\chi'_{r+1}+\ldots+\chi'_s)$ lies in the alcove $\tilde{A}$ for all $r=1,\ldots,s,$ and all $m$ bounded
by a function linear in $\epsilon^{-1}$. We can increase $N$ and decrease $\epsilon$ (so that $N\epsilon$
is fixed) in such a way that $\lambda':=\lambda+m(\chi_1'+\ldots+\chi_s')$
now satisfies $\langle\alpha_{\Gamma_i},\lambda'\rangle> \hat{m}_i+\epsilon p$
for all $i$. Thanks to the previous paragraph, this finishes the proof.
\end{proof}


Now we proceed to our next condition on $p$.
Fix $\chi\in \mathcal{P}_2$. Recall  the locus $\paramq^{reg[\chi]}$ of $\lambda\in \paramq$
such that $\lambda,\lambda+\chi$ are in the same quantum chamber. This locus is given by the conditions
of the form $\langle\lambda, \alpha_{\Gamma_i}\rangle\not\in \tilde{\Sigma}_\Gamma(\chi)$, where
$\Gamma$ runs over all classicla walls and $\tilde{\Sigma}_\Gamma(\chi)$ is a suitable finite subset
such that $\tilde{\Sigma}_\Gamma(\chi)\subset
\Sigma_\Gamma$ containing $\tilde{\Sigma}_\Gamma$.

We also assume that $p$ is chosen in such a way that, for $\lambda\in \paramq_{\Z}$, the following
are equivalent:
\begin{itemize}
\item[(i)] $\lambda,\lambda+\chi$ are in the same $p$-alcove,
\item[(ii)] $\langle\lambda,\alpha_\Gamma\rangle\neq (p+1)\sigma+pm$ for
all walls $\Gamma$,  $\sigma\in \tilde{\Sigma}_\Gamma(\chi), m\in \Z$.
\end{itemize}
This holds for $p\gg 0$.

Finally, we need an assumption on  the residues of $\mathsf{c}_{\nu,\lambda}(\mathsf{x})$
mod $p$. Let $\bar{\mathsf{c}}_{\nu,\lambda}(\mathsf{x})$ denote this residue viewed as an element
in $\{0,\ldots,p-1\}$. By our previous
assumptions on $p$, $\bar{\mathsf{c}}_{\nu,\lambda}(\mathsf{x})$ equals $(p+1)\mathsf{c}_{\nu,\lambda}(\mathsf{x})$ mod $p$. We suppose
that the $h$-blocks for $\lambda\in \Lambda$ are {\it ordered} in the following sense: if $\mathsf{x}_1,\mathsf{x}_2$ lie in
the same $h$-block for $\lambda$ (which, recall, means that $\mathsf{c}_{\nu,\lambda}(\mathsf{x}_1)-
\mathsf{c}_{\nu,\lambda}(\mathsf{x}_2)\in \Z$) and $\mathsf{x}_1',\mathsf{x}_2'$ lie in a different
$h$-block, then $\bar{\mathsf{c}}_{\nu,\lambda}(\mathsf{x}_1)>\bar{\mathsf{c}}_{\nu,\lambda}(\mathsf{x}_1')$
implies $\bar{\mathsf{c}}_{\nu,\lambda}(\mathsf{x}_2)>\bar{\mathsf{c}}_{\nu,\lambda}(\mathsf{x}_2')$.
This is clearly the case for $p$ large enough.  

\subsection{p-center}\label{SS_p_center}
We will put one more assumption: that the algebras $\A_{\lambda,\F}$ have large center for $\lambda\in
\paramq_{\F_p}$.
Let the superscript (1) denote the Frobenius twist.
We assume that the inclusion $\F[Y_\F^{(1)}]\hookrightarrow
\F[Y]$ given by $f\mapsto f^p$ lifts to a central embedding $\F[Y_\F^{(1)}]\hookrightarrow \A_{\lambda,\F}$ (so that the former is the associated graded homomorphism of the latter).
The image of $\F[Y_\F^{(1)}]\hookrightarrow \A_{\lambda,\F}$ is easily seen to coincide
with the entire center of the algebra $\A_{\lambda,\F}$.

The existence of the p-center is the classical fact in the case when $Y$ is the nilpotent
cone (and $\A_\lambda$ is a central reduction of the universal enveloping algebra).
For Nakajima quiver varieties, the existence follows from \cite{BFG}.

In particular, for an irreducible $\A_{\lambda,\F}$-module it makes sense to consider its
{\it p-character}, a point in $Y_\F^{(1)}$ that gives the action of the center on the module.

\subsection{Equivalences}
\subsubsection{Abelian equivalences}
For $\lambda\in \paramq_\F, \chi\in \mathcal{P}_2$ we define an $\A_{\lambda+\chi,\F}$-$\A_{\lambda,\F}$-bimodule
$\A_{\lambda,\chi,\F}$ as the specialization of $\F\otimes_\Ring \A_{\paramq,\chi,\Ring}$ to $\lambda$.

\begin{Prop}\label{Prop:char_p_ab_eq}
Let $\lambda,\lambda'=\lambda+\chi$ lie in the same $p$-alcove of $\param_{\Z}$. Then the categories
$\A_{\lambda,\F}\operatorname{-mod},\A_{\lambda',\F}\operatorname{-mod}$ are equivalent
via tensoring with a suitable sequence of bimodules $\A_{\lambda_i,\chi_i,\F}$, where the elements $\lambda_i$ lie in the same $p$-alcove and $\chi_i\in \mathcal{P}_2$.
\end{Prop}
\begin{proof}
Thanks to Lemma \ref{Lem:alcove_point_translation}, it is enough to prove that $\A_{\lambda_i,\chi_i,\F}$
is a Morita equivalence $\A_{\lambda_i+\chi_i,\F}$-$\A_{\lambda_i,\F}$-bimodule
as long as $\lambda_i,\lambda_i+\chi_i$ are in the same $p$-alcove. Lemma \ref{Lem:Morita_Ring}
implies that  $\A_{\paramq^{reg[\chi]},\chi,\Ring}, \A_{\paramq^{reg[-\chi]},\chi,\Ring}$
are mutually inverse Morita equivalence bimodules. Thanks to the equivalence of
(i) and (ii) in the Section \ref{SS_p_assumptions}, we see that $\lambda_i$ (mod $p$) belongs to
$\paramq^{reg}_\Ring$ mod $p$. It follows that $\A_{\lambda_i,\chi_i,\F}$
and $\A_{\lambda_i+\chi_i,-\chi_i,\F}$ are mutually inverse Morita equivalence
bimodules.
\end{proof}

\subsubsection{Perverse equivalences}
Let $A$ be a real chamber and $\Theta$ be its face.
Now pick $\lambda\in \Lambda$ compatible with $(A,\Theta)$.
Let $\chi\in\mathcal{P}_1$ such that $\lambda+\chi$ is compatible with $(A^-,\Theta)$,
where $A^-$ is the real alcove opposite to $A$ with respect to $\Theta$.
Consider the chains of ideals $\I^j_{\lambda,\F}:=\F\otimes_{\Ring}\I^j_{\lambda,\Ring}$
in $\A_{\lambda,\F}$ and the similarly defined chain $\I^j_{\lambda+\chi,\F}$.

\begin{Prop}\label{Prop:pervers_char_p}
The functor $\A_{\lambda,\chi,\F}\otimes^L_{\A_{\lambda,\F}}\bullet:
D^b(\A_{\lambda,\F}\operatorname{-mod})\xrightarrow{\sim} D^b(\A_{\lambda+\chi,\F}\operatorname{-mod})$
is a perverse derived equivalence with respect to the filtrations defined
by chains of  ideals $\I^j_{\lambda,\F},\I^j_{\lambda+\chi,\F}$.
\end{Prop}
\begin{proof}
We have $(\I^j_{?,\F})^2=\I^j_{?,\F}$ ($?=\lambda,\lambda+\chi$) because the similar
equalities hold over $\Ring$. Also note that the vanishing conditions (b1)-(b6) from
Section \ref{SS_Perv} hold over $\F$ because they hold over $\Ring$, see Lemma
\ref{Lem:perv_Ring}. Details are left to the reader.
\end{proof}

\begin{Cor}\label{Cor:alcove_perverse}
Let $(A,\Theta)$ be as above and let $A^-$ be the real alcove that is opposite to $A$
with respect to $\Theta$. Let $\lambda_1,\lambda_2$ be the elements of $\Lambda$
associated to $(A,\Theta),(A^-,\Theta)$, respectively.
Then the categories $\A_{\lambda_1,\F}\operatorname{-mod}$
and $\A_{\lambda_2,\F}\operatorname{-mod}$ are perverse derived equivalent.
\end{Cor}

\section{Modular categories $\mathcal{O}$}
\subsection{Definition and basic properties}
\subsubsection{Definition}
Recall that we have an action of $T_{\F}$ on $X_\F, Y_\F, \A_{\lambda,\F}$ etc.

For $\lambda\in \paramq_\Z$,  we consider the category $\tilde{\OCat}(\A_{\lambda,\F})$
consisting of all (weakly) $T_{\F}$-equivariant finite dimensional $\A_{\lambda,\F}$-modules.

Note that the $p$-character of a simple module in $\tilde{\OCat}(\A_{\lambda,\F})$
(a point in $Y^{(1)}_\F$) is $T_{\F}$-stable. Since $T_\F$ has finitely many
fixed points in $X_{\F}$, we conclude that the only fixed point in $Y^{(1)}_{\F}$
is $0$ and so the $p$-character of any module in $\tilde{\OCat}(\A_{\lambda,\F})$ is zero.

\subsubsection{Baby Verma modules}
Now suppose that $\lambda\in \paramq_\Z$ lies in a $p$-alcove.
Let $\nu$ be a generic one-parameter subgroup.

\begin{Lem}\label{Lem:Ca_char_p}
We have $\Ca_{\nu}(\A_{\lambda,\F})\cong \F[X^T]$.
\end{Lem}
\begin{proof}
Recall, Section \ref{SS_Verma_forms}, that we have
\begin{equation}\label{eq:Cartan_iso_integr}
\Ca_\nu(\A_{\paramq^{reg}, \Ring})\cong \Ring[\paramq^{reg}][X^T].
\end{equation}
Since $\lambda$ is in a $p$-alcove, its reduction modulo $p$ is in
$\paramq^{reg}_\F$.
Specializing (\ref{eq:Cartan_iso_integr}), we get the required isomorphism.
\end{proof}

Let us give an example of an object in $\tilde{\OCat}(\A_{\lambda,\F})$, a baby Verma module.
For $\mathsf{x}\in X^{T}$ and $\kappa\in \mathfrak{X}(T)$, we can form the Verma module
$\Delta_{\nu,\F}(\mathsf{x},\kappa)$ as in characteristic $0$:  $$\Delta_{\nu,\F}(\mathsf{x},\kappa):=
\A_{\lambda,\F}/\A_{\lambda,\F}\A_{\lambda,\F}^{>0}\otimes_{\Ca_\nu(\A_{\lambda,\F})}\F_\mathsf{x},$$
where we put $\F_\mathsf{x}$ in degree $\kappa$. This is not an object in
$\tilde{\OCat}(\A_{\lambda,\F})$, in fact, it is a pro-object.

We define the baby Verma module $\underline{\Delta}_{\nu,\F}(\mathsf{x},\kappa)$ as the fiber of
$\Delta_{\nu,\F}(\mathsf{x},\kappa)$ over $0\in Y^{(1)}_{\F}$.

\begin{Lem}\label{Lem:simple_classif}
A choice of a generic one-parameter subgroup gives rise to an identification of
the set of simples in $\tilde{\OCat}(\A_{\lambda,\F})$ with $X^{T}\times \mathfrak{X}(T)$.
\end{Lem}
\begin{proof}
A standard argument shows that every simple is the quotient of a unique baby Verma module.
This gives the required classification.
\end{proof}

The simple corresponding to $(\mathsf{x},\kappa)$ will be denoted by $\underline{L}_{\nu,\lambda,\F}(\mathsf{x},\kappa)$.


\subsubsection{Equivariant block decomposition}
We will also need a direct sum decomposition of the category $\tilde{\OCat}(\A_{\lambda,\F})$.
Consider the quantum comoment map $\Phi:\mathfrak{t}_{\F}\rightarrow \A_{\lambda,\F}$.
Note that $h$ has degree $0$ hence preserves every graded component of $M\in \tilde{\OCat}(\A_{\lambda,\F})$. Then
$M:=\bigoplus_{\alpha\in \F_p\otimes\mathfrak{X}(T)}M^\alpha$, where $M^\alpha$ denotes the $T_{\F}$-graded
subspace of $M$, where all $h\in \mathfrak{t}_{\F_p}$ acts on the graded component $M_\kappa^\alpha$ with
the single eigenvalue $\alpha+\kappa$. Then $M^\alpha_\kappa$ is a graded $\A_{\lambda,\F}$-submodule
and hence an object in $\tilde{\OCat}(\A_{\lambda,\F})$.

We consider the category $\tilde{\OCat}^\beta(\A_{\lambda,\F}), \beta\in \tf_{\F_p}^*,$
consisting of all $M$ with $M=M^\beta$. These categories will be called {\it equivariant blocks}. Then $\tilde{\OCat}(\A_{\lambda,\F})=\bigoplus_{\beta\in \tf^*_{\F_p}}\tilde{\OCat}^\beta(\A_{\lambda,\F})$.
Note that all categories $\tilde{\OCat}^\beta(\A_{\lambda,\F})$ are obtained from
$\tilde{\OCat}^0(\A_{\lambda,\F})$ by shifting the $T_{\F}$-equivariant structure.

For $\mathsf{x}\in X^T$, let $\Phi_{\mathsf{x}}\in \mathfrak{t}_\F^*$ denote the composition of $\Phi:
\mathfrak{t}_{\F}\rightarrow \A_{\lambda,\F}^T$ and the projection
$\A_{\lambda,\F}^T\twoheadrightarrow \F$ corresponding to $\mathsf{x}$. This is an element
of $\tf^*_{\F_p}$. We note that
$\underline{L}_{\nu,\lambda,\F}(\mathsf{x},\kappa)$ lies in
$\tilde{\OCat}^\beta(\A_{\lambda,\F})$ if and only if $\Phi_\mathsf{x}-\kappa=\beta$ modulo $p$.

\subsubsection{Equivalences}
Recall that a choice of a $T_\F$-equivariant  structure on $O_\Ring(\chi)$ gives a $T_\F$-equivariant structure on $\A_{\paramq,\chi,\Ring}$ and hence on $\A_{\paramq,\chi,\F}$.

\begin{Lem}\label{Lem:transl_O}
For $\lambda\in \paramq_\F$ and $\chi\in \mathcal{P}_2$, the functor $\A_{\lambda,\chi,\F}\otimes^L_{\A_{\lambda,\F}}\bullet$ restricts to
$$D^-_{fin}(\A_{\lambda,\F}\operatorname{-mod}^{T})\rightarrow
D^-_{fin}(\A_{\lambda+\chi,\F}\operatorname{-mod}^{T}).$$
\end{Lem}
Here the subscript ``fin'' stands for the full subcategory of all complexes
with finite dimensional homology.
\begin{proof}
The claim that $\A_{\lambda,\chi,\F}\otimes^L_{\A_{\lambda,\F}}\bullet$
sends finite dimensional modules to complexes with finite dimensional
homology follows from the fact that $\A_{\lambda,\chi,\F}$ is a HC bimodule.
Since $\A_{\lambda,\chi,\F}$ is also $T_\F$-equivariant, it follows that
$\A_{\lambda,\chi,\F}\otimes^L_{\A_{\lambda,\F}}\bullet$ upgrades to a functor
between equivariant categories.
\end{proof}

\begin{Prop}\label{Prop:cat_O_equiv}
Let $\lambda,\lambda+\chi$ lie in the same alcove. Then $\A_{\lambda,\chi,\F}\otimes_{\A_{\lambda,\F}}\bullet$
is an equivalence $\tilde{\mathcal{O}}(\A_{\lambda,\F})$ and $\tilde{\mathcal{O}}(\A_{\lambda+\chi,\F})$
that sends $\underline{L}_{\nu,\lambda,\F}(\mathsf{x},\kappa)$ to $\underline{L}_{\nu,\lambda+\chi,\F}(\mathsf{x},\kappa+\mathsf{wt}_\chi(\mathsf{x}))$.
\end{Prop}
\begin{proof}
By Proposition \ref{Prop:char_p_ab_eq}, $\A_{\lambda,\chi,\F}$ is a Morita equivalence bimodule.
It follows from Lemma \ref{Lem:Verma_transl_global}, that $\A_{\lambda,\chi,\F}$ (with our choice of a $T_\F$-equivariant
structure) sends $\Delta_{\nu,\lambda,\F}(\mathsf{x};\kappa)$ to $\Delta_{\nu,\lambda+\chi,\F}(\mathsf{x};\kappa+\mathsf{wt}_\chi(\mathsf{x}))$. Since
$\underline{L}_{\nu,\lambda,\F}(\mathsf{x};\kappa)$ is the unique graded quotient of $\Delta_{\nu,\lambda,\F}(\mathsf{x};\kappa)$,
our claim follows.
\end{proof}

\subsubsection{Duality}\label{SSS_duality}
Now we want to discuss the contravariant duality between categories $\tilde{\mathcal{O}}(\A_{\lambda,\F})$
and $\tilde{\mathcal{O}}(\A_{-\lambda,\F})$. Recall that we have
an isomorphism $(\A_{\lambda,\F})^{opp}\cong \A_{-\lambda,\F}$. Then $M\mapsto M^*$
defines a contravariant equivalence between $\tilde{\mathcal{O}}(\A_{\lambda,\F})$
and $\tilde{\mathcal{O}}(\A^{opp}_{\lambda,\F})$. But
$\A^{opp}_{\lambda,\F}\cong \A_{-\lambda,\F}$.

\begin{Lem}\label{Lem:dual_simples}
We have $\underline{L}_{\nu,\lambda,\F}(\mathsf{x},\kappa)=\underline{L}_{-\nu,-\lambda,\F}(\mathsf{x},-\kappa)^*$.
\end{Lem}
\begin{proof}
Since $\bullet^*$ is an equivalence, the object $\underline{L}_{-\nu,-\lambda,\F}(\mathsf{x},-\kappa)^*$
is simple. First of all, let us show that it corresponds to the point $\mathsf{x}\in X^T$. Recall, see,
e.g., \cite[Section 4.2]{CWR}, that we have an isomorphism $\Ca_{\nu}(\A_{\lambda,\F})^{opp}\cong \Ca_{-\nu}(\A_{-\lambda,\F})$. This isomorphism intertwines the identifications of these algebras
with $\F[X^T]$. The condition that $\underline{L}_{-\nu,-\lambda,\F}(\mathsf{x},-\kappa)^{\A_{\lambda,\F}^{\nu,<0}}=\F_{\mathsf{x}}$
translates to $\underline{L}_{-\nu,-\lambda,\F}(\mathsf{x},-\kappa)^*/\A_{\lambda,\F}^{\nu,<0}
\underline{L}_{-\nu,-\lambda,\F}(\mathsf{x},-\kappa)^*\cong \F_{\mathsf{x}}$. Since $\underline{L}_{-\nu,-\lambda,\F}(\mathsf{x},-\kappa)^*$ is an irreducible module, it follows that
$\left(\underline{L}_{-\nu,-\lambda,\F}(\mathsf{x},-\kappa)^{*}\right)^{\A_{\lambda,\F}^{>0,\nu}}\cong \F_{\mathsf{x}}$.
Hence $\underline{L}_{-\nu,-\lambda,\F}(\mathsf{x},-\kappa)^*\cong \underline{L}_{\nu,\lambda,\F}(\mathsf{x},\kappa')$
for some $\kappa'$.
Note also that the highest degree (with respect to $\nu$) component of
$\underline{L}_{-\nu,-\lambda,\F}(\mathsf{x},-\kappa)^*$ equals $\langle\kappa,\nu\rangle$, while the
character of the action of $T$ on this component is $\kappa$. So
$\underline{L}_{-\nu,-\lambda,\F}(\mathsf{x},-\kappa)^*\cong \underline{L}_{\nu,\lambda,\F}(\mathsf{x},\kappa)$.
\end{proof}

Using the duality we can define analogs of costandard objects of $\tilde{\mathcal{O}}(\A_{\lambda,\F})$
that will be ind-objects to be denoted by $\nabla^\F_{\nu,\lambda}(\mathsf{x},\kappa)$. Namely, we define
$\nabla^\F_{\nu,\lambda}(\mathsf{x},\kappa)$ as the restricted dual of $\Delta_{-\nu,-\lambda,\F}(\mathsf{x},-\kappa)$.
We note that $\nabla^\F_{\nu,\lambda}(\mathsf{x},\kappa)$ is not the same as $\F\otimes_{\Ring}\nabla_{\nu,\lambda,\Ring}(\mathsf{x},\kappa)$ (the latter is still a pro-object).

\subsection{Highest weight structure}
Let $\mathcal{C}$ be an abelian Artinian and Noetherian category, equivalently, an abelian
category where all objects have finite length. Suppose that we have an
auto-equivalence $\mathcal{S}$ of $\mathcal{C}$ such that the action of
$\Z$ on $\operatorname{Irr}(\mathcal{C})$ induced by $\mathcal{S}$ is free. Further, suppose
that that $\operatorname{Irr}(\mathcal{C})/\Z$ is finite.

Further, let $\leqslant$ be a partial order on $\operatorname{Irr}(\mathcal{C})$ subject to the
following properties:
\begin{itemize}
\item The $\Z$-action preserves the order.
\item We have $L<\mathcal{S}L$ for every $L\in \operatorname{Irr}(\mathcal{C})$.
\item For any two $L,L'$ with $L<L'$,  there is $n\in \Z_{\geqslant 0}$ such that
$L'<\mathcal{S}^{n}L$.
\end{itemize}

For $L\in \operatorname{Irr}(\mathcal{C})$, let $\mathcal{C}_{\leqslant L}$ denote the Serre
span of $L'\in \operatorname{Irr}(\mathcal{C})$ with $L'\leqslant L$.
We say that the pair $(\mathcal{C},\leqslant)$ is {\it periodic highest weight} (shortly, PHW)
if the following holds:
\begin{itemize}
\item[(PHW)] Each quotient $\mathcal{C}_{\leqslant L}/\mathcal{C}_{\leqslant L'}$ with $L'<L$ is
a highest weight category with respect to the order $\leqslant$ in the usual sense
(see Section \ref{SSS_HW}).
\end{itemize}

Recall that by a poset interval we mean a subset $\mathfrak{I}$ such that if $L,L'\in \mathfrak{I}$
satisfy $L\leqslant L'$ and $L''$ some other element of the poset with $L\leqslant L''\leqslant L'$,
then $L''\in \mathfrak{I}$. For an interval $\mathfrak{I}\subset \operatorname{Irr}(\mathcal{C})$
we can define a subquotient category $\mathcal{C}_{\mathfrak{I}}$. Namely, consider the set
$\overline{\mathfrak{I}}$ of all simples $L$ with $L\leqslant L'$ for some $L'\in \mathfrak{I}$.
Then we can consider the Serre span $\mathcal{C}_{\overline{\mathfrak{I}}}$ of all simples in
$\overline{\mathfrak{I}}$. By definition, $\mathcal{C}_{\mathfrak{I}}:=
\mathcal{C}_{\overline{\mathfrak{I}}}/\mathcal{C}_{\overline{\mathfrak{I}}\setminus \mathfrak{I}}$.
We note that (PHW) is equivalent to $\mathcal{C}_{\mathfrak{I}}$ being highest weight
(with respect to the restricted order) for any finite interval $\mathfrak{I}$.

\subsection{Main result}
For simplicity, we will assume  that $\dim T=1$. We will remark on the general case later.

Let us define a partial order $\leqslant_\nu$ on $\operatorname{Irr}(\tilde{\mathcal{O}}(\A_{\lambda,\F}))=
X^T\times \Z$ as follows: $(\mathsf{x},\kappa)\leqslant_\nu (\mathsf{x}',\kappa')$ if the following
hold:
\begin{itemize}
\item $(\mathsf{x},\kappa)$ and $(\mathsf{x}',\kappa')$ belong to the same equivariant block
meaning $\Phi_{\mathsf{x}}-\kappa=\Phi_{\mathsf{x}'}-\kappa'$ in $\F_p$.
\item $(\mathsf{x},\kappa)=(\mathsf{x}',\kappa')$ or $\kappa< \kappa'$.
\end{itemize}

We can define a $\Z$-action on this poset as follows:  $z.(\mathsf{x},\kappa):=
(\mathsf{x}, \kappa+zp)$. Clearly, it satisfies the conditions of the previous section.

Now pick a finite interval $\mathfrak{I}\subset X^T\times \mathfrak{X}(T)$.

\begin{Lem}\label{Lem:finite_mult}
Let $(\mathsf{x},\kappa),(\mathsf{x}',\kappa')\in \mathfrak{I}$. Then there is a graded submodule
$M\subset \Delta_{\nu,\lambda,\F}(\mathsf{x}',\kappa')$ of finite codimension such that
$\underline{L}_{\nu,\lambda,\F}(\mathsf{x},\kappa)$ does not occur in $M$ (i.e., does not
occur in the quotient $M/M'$ for any graded $M'\subset M$ of finite codimension). The dual
statement holds for $\nabla^\F_{\nu,\lambda}(\mathsf{x}',\kappa')$.
\end{Lem}
Note that the intersection of graded  submodules of finite codimension submodules
in $\Delta_{\nu,\lambda,\F}(\mathsf{x},\kappa)$ is zero.
\begin{proof}
Note that the set of $T_\F$-weights in $\underline{L}_{\nu,\lambda,\F}(\mathsf{x},\kappa)$ is finite.
All weight spaces in $\Delta_{\nu,\lambda,\F}(\mathsf{x}',\kappa')$ are finite dimensional. So we can find
a $T_\F$-stable ideal of finite codimension $\mathfrak{m}\subset \F[Y^{(1)}_\F]$ such that
this set of weights doesn't appear in $M:=\mathfrak{m}\Delta_{\nu,\lambda,\F}(\mathsf{x}',\kappa')$.
The module $\Delta_{\nu,\lambda,\F}(\mathsf{x}',\kappa')/M$ is finite dimensional and does not
contain $\underline{L}_{\nu,\lambda,\F}(\mathsf{x},\kappa)$ as a subquotient.

The claim for $\nabla^\F_{\nu,\lambda}(\mathsf{x}',\kappa')$ follows by duality.
\end{proof}

Thanks to the lemma, we can define the object
$\Delta^{\mathfrak{I}}_{\nu,\lambda,\F}(\mathsf{x}',\kappa')\in
\tilde{\OCat}(\A_{\lambda,\F})_{\mathfrak{I}}$
for $(\mathsf{x}',\kappa')\in \mathfrak{I}$ as the image of
$\Delta_{\nu,\lambda,\F}(\mathsf{x}',\kappa')/M$ from the previous lemma.
Analogously, we can define $\nabla^{\F,\mathfrak{I}}_{\nu,\lambda}(\mathsf{x}',\kappa')$.

The following theorem is the main result of this section.

\begin{Thm}\label{Thm:phw_structure}
The category
$\tilde{\OCat}(\A_{\lambda,\F})$ is a PHW category with respect to the order
$\leqslant_\nu$. Moreover, for any finite interval  $\mathfrak{I}\subset X^T\times \mathfrak{X}(T)$,
the standard and costandard objects in the subquotient
$\tilde{\OCat}(\A_{\lambda,\F})_{\mathfrak{I}}$
corresponding to $(\mathsf{x},\kappa)\in \mathfrak{I}$ are $\Delta^{\mathfrak{I}}_{\nu,\lambda,\F}(\mathsf{x},\kappa)$
and $\nabla^{\mathfrak{I},\F}_{\nu,\lambda}(\mathsf{x},\kappa)$.
\end{Thm}

\subsubsection{Examples}
Let us consider two examples of the order $\leqslant_\nu$.
First, let $X=T^*(G/B)$. Recall that $X^T$ is identified with
$W$ and $\mathsf{c}_{\nu,\lambda}(w)=\langle \nu, w\lambda\rangle$. Pick $\nu=\rho^\vee$. So the order on
$W\times \Z$ is as follows:  $(w,\kappa)\leqslant (w',\kappa')$ if
\begin{itemize}
\item $\langle w\lambda,\rho^\vee\rangle-\kappa= \langle w'\lambda,\rho^\vee\rangle-\kappa'$ in $\F_p$,
\item and either $\kappa<\kappa'$ or $(w,\kappa)=(w',\kappa')$.
\end{itemize}

Now let $X=\operatorname{Hilb}_n(\mathbb{A}^2)$. Here $X^T$ is in bijection with the set $\mathcal{P}(n)$
of partitions of $n$. So the order on $\mathcal{P}(n)\times \Z$ is as follows (recall that we set $c=\lambda-1/2$):
$(\psi,\kappa)\leqslant (\psi',\kappa')$ if
\begin{itemize}
\item $c \operatorname{cont}(\psi)-n(\psi)-\kappa= c\operatorname{cont}(\psi')-n(\psi')-\kappa'$ in $\mathbb{F}_p$,
\item and either $\kappa<\kappa'$ or $(\psi,\kappa)=(\psi',\kappa')$.
\end{itemize}

\subsection{Proof of Theorem \ref{Thm:phw_structure}}
The proof is in several steps.

{\it Step 1}. Pick $(\mathsf{x},\kappa)\in \Ifrak$. Set $\Ifrak':=\{(\mathsf{x}',\kappa')\in \Ifrak| (\mathsf{x}',\kappa')
\leqslant (\mathsf{x},\kappa)\}$.  By the construction, the object
$\Delta^{\Ifrak}_{\nu,\lambda,\F}(\mathsf{x},\kappa)$ coincides with
$\Delta^{\Ifrak'}_{\nu,\lambda,\F}(\mathsf{x},\kappa)$. By the construction,
$\Delta^{\Ifrak'}_{\nu,\lambda,\F}(\mathsf{x},\kappa)$ is the projective cover of $\underline{L}_{\nu,\lambda,\F}(\mathsf{x},\kappa)$
in the Serre subcategory $\tilde{\OCat}(\A_{\lambda,\F})_{\mathfrak{I}'}
\subset \tilde{\OCat}(\A_{\lambda,\F})_{\mathfrak{I}}$. Dually, $\nabla^{\Ifrak,\F}_{\nu,\lambda}(\mathsf{x},\kappa)$
is the injective hull of $\underline{L}_{\nu,\lambda,\F}(\mathsf{x},\kappa)$ in that subcategory.

{\it Step 2}.  To finish the proof of the theorem we need to show that the projectives in $\tilde{\OCat}(\A_{\lambda,\F})_{\mathfrak{I}}$
are $\Delta$-filtered (Steps 2-4) and that $\Delta$'s appear in a correct order (Step 5).

Note that, similarly to Lemma \ref{Lem:Ext_Tor}, we have  \begin{equation}\label{eq:Thm_eq1}\Ext^i_{\A_{\lambda,\F},T}(\Delta_{\nu,\lambda,\F}(\mathsf{x},\kappa),
\nabla^{\F}_{\nu,\lambda}(\mathsf{x}',\kappa'))=\operatorname{Tor}_i^{\A_{\lambda,\F},T}
(\Delta_{\nu,\lambda,\F}(\mathsf{x},\kappa), \Delta^r_{-\nu,\lambda,\F}(\mathsf{x}',\kappa'))^*.\end{equation}
Using Lemmas \ref{Lem:Ext_vanish_global} and \ref{Lem:Verma_R_flatness} (as well as the direct analog
of the latter for the right-handed Verma modules) we see that the right hand side
of (\ref{eq:Thm_eq1}) vanishes. We deduce that
\begin{equation}\label{eq:Ext_vanishing_hw2}
\dim\Ext^i_{\A_{\lambda,\F},T}(\Delta_{\nu,\lambda,\F}(\mathsf{x},\kappa),
\nabla^{\F}_{\nu,\lambda}(\mathsf{x}',\kappa'))=0, \text{ for }i>0.
\end{equation}
In  steps 3 and 4 we will deduce that the projectives are $\Delta$-filtered from
(\ref{eq:Ext_vanishing_hw2}).

{\it Step 3}. Take an integer $z$ and consider the category $\A_{\lambda,\F}\operatorname{-mod}^{T}_{\leqslant z}$
of all $T_{\F}$-equivariant $\A_{\lambda,\F}$-modules $M$   such
that \begin{itemize}
\item any $T_\F$-weight $\kappa$ of $M$ satisfies $\kappa\leqslant z$,
\item all weight spaces are finite dimensional.
\end{itemize}
This is a Serre subcategory
of $\A_{\lambda,\F}\operatorname{-mod}^T$. Note that $\Delta_{\nu,\lambda,\F}(\mathsf{x},\kappa), \nabla^\F_{\nu,\lambda}(\mathsf{x},\kappa)\in \A_{\lambda,\F}\operatorname{-mod}^T_{\leqslant z}$
if and only if $\kappa\leqslant z$.

Repeating the proof of Lemma \ref{Lem:finite_mult} we see
that every object in  $\A_{\lambda,\F}\operatorname{-mod}^T_{\leqslant z_2}$ contains
a graded submodule $M$ of finite codimension  such that $\underline{L}_{\nu,\lambda,\F}(\mathsf{x},\kappa)$
with $\kappa\in [z_1,z_2]$ does not appear in $M$. It follows that
\begin{equation}\label{eq:quotient_equiv}\tilde{\OCat}(\A_{\lambda,\F})_{[z_1,z_2]}
\xrightarrow{\sim}\A_{\lambda,\F}\operatorname{-mod}^T_{\leqslant z_2}/\A_{\lambda,\F}\operatorname{-mod}^T_{< z_1} \end{equation}

Thanks to (\ref{eq:Ext_vanishing_hw2}) and the standard results
that the Ext vanishing in an ambient category implies that in any Serre subcategory, we get
\begin{equation}\label{eq:Ext_vanishing_hw3}
\dim\Ext^i_{\A_{\lambda,\F}\operatorname{-mod}^T_{\leqslant z}}(\Delta_{\nu,\lambda,\F}(\mathsf{x},\kappa),
\nabla^{\F}_{\nu,\lambda}(\mathsf{x}',\kappa'))=0, \text{ for }i>0.
\end{equation}

{\it Step 4}. Note that the category $\A_{\lambda,\F}\operatorname{-mod}^T_{\leqslant z}$
has enough projectives. Indeed, for $z'\leqslant z$, the object $P_{z,z'}:=\A_{\lambda,\F}/\A_{\lambda,\F}\A_{\lambda,\F}^{>z-z'}$, where the image
of $1$ has degree $z'$, is projective. This is because $\Hom(P_{z,z'}, M)=M_{z'}$ for
$M\in \A_{\lambda,\F}\operatorname{-mod}^T_{\leqslant z}$.

Every  module in $\A_{\lambda,\F}\operatorname{-mod}^T_{\leqslant z}$
is covered by a direct sum of the modules $P_{z,z'}$ (for finitely generated modules we can take
a finite sum). A standard argument shows that $P_{z,z'}$ is filtered by
$\Delta$'s. For reader's convenience, let us provide this argument. We prove the
existence of filtration by induction on $z-z'$. The case of $z-z'=0$ is obvious:
the object $P_{z,z}$ is the direct sum of Verma modules. Now note that we have
a natural surjection $P_{z,z'}\twoheadrightarrow P_{z-1,z'}$, let $K$ stand
for the kernel.

Note that $\Ext^i_{\A_{\lambda,\F}\operatorname{-mod}^T_{\leqslant z}}(K,
\nabla^\F_{\nu,\lambda}(\mathsf{x}',\kappa'))=0$ for $i>0$. Indeed, for $i>1$ this follows from
$P_{z,z'}$ being projective and $P_{z-1,z'}$ being filtered by $\Delta$'s. We also have
an exact sequence
\begin{align*}&0\rightarrow \Hom(K,\nabla^\F_{\nu,\lambda}(\mathsf{x}',\kappa'))
\rightarrow \Hom(P_{z,z'},\nabla^\F_{\nu,\lambda}(\mathsf{x}',\kappa'))
\rightarrow \Hom(P_{z-1,z'},\nabla^\F_{\nu,\lambda}(\mathsf{x}',\kappa'))\\&
\rightarrow \Ext^1(K,\nabla^\F_{\nu,\lambda}(\mathsf{x}',\kappa'))\rightarrow 0.
\end{align*}
Now consider two cases. If $\kappa'<z$, then the middle homomorphism is
the identity isomorphism $\nabla^\F_{\nu,\lambda}(\mathsf{x}',\kappa')_{z'}
\xrightarrow{\sim} \nabla^\F_{\nu,\lambda}(\mathsf{x}',\kappa')_{z'}$. So
\begin{equation}\label{eq:Hom_Ext1_vanish}
\Hom(K,\nabla^\F_{\nu,\lambda}(\mathsf{x}',\kappa'))=\Ext^1(K,\nabla^\F_{\nu,\lambda}(\mathsf{x}',\kappa'))=0.
\end{equation}
If $\kappa'=z$, then $\Hom(P_{z-1,z'},\nabla^\F_{\nu,\lambda}(\mathsf{x}',\kappa'))=0$
and hence
\begin{equation}\label{eq:Ext1_vanish}
\Hom(K,\nabla^\F_{\nu,\lambda}(\mathsf{x}',\kappa'))=\nabla^\F_{\nu,\lambda}(\mathsf{x}',\kappa')_{z'},\quad
\Ext^1(K,\nabla^\F_{\nu,\lambda}(\mathsf{x}',\kappa'))=0.
\end{equation}
in both cases.

Set $d(\mathsf{x}'):=\dim\nabla^\F_{\nu,\lambda}(\mathsf{x}',\kappa')_{z'}$
From (\ref{eq:Hom_Ext1_vanish}) and the dual version of Lemma \ref{Lem:cost_inclusion}
we conclude that there is an epimorphism
$$\bigoplus_{\mathsf{x}\in X^T} \Delta_{\nu,\lambda,\F}(\mathsf{x},z)^{\oplus d(\mathsf{x})}\twoheadrightarrow K.$$
The kernel of this epimorphism has no nonzero homomorphisms to the objects $\nabla^\F_{\nu,\lambda}(\mathsf{x}',\kappa')$
with $\kappa'\leqslant z$ by (\ref{eq:Hom_Ext1_vanish}) and (\ref{eq:Ext1_vanish}).
On the other hand, if the kernel is nonzero, then it has a finite dimensional graded quotient (compare to
the proof of Lemma \ref{Lem:finite_mult}) that must have a nonzero Hom to some $\nabla^\F_{\nu,\lambda}(\mathsf{x}',\kappa')$.

So we see that $K=0$.
This finishes the proof of the claim that all objects $P_{z,z'}$ admit
$\Delta$-filtrations.

{\it Step 5}. Note that $\Ext^1_{\A_{\lambda,\F}\operatorname{-mod}^T}(\Delta_{\nu,\lambda,\F}(\mathsf{x},\kappa),
\Delta_{\nu,\lambda,\F}(\mathsf{x}',\kappa'))\neq 0$, then $\kappa'>\kappa$. So the Verma modules in a $\Delta$-filtration of
$P_{z,z'}$ appear in the correct order. Now suppose that
$\mathfrak{I}=\{(\mathsf{x},\kappa)| z_1\leqslant \kappa\leqslant z_2\}$
for some integers $z_1<z_2$. The images
of $P_{z_2,z'}$ for $z'\in [z_1,z_2]$ under the quotient functor
$\A_{\lambda,\F}\operatorname{-mod}^T_{\leqslant z_2}\twoheadrightarrow \tilde{\mathcal{O}}(\A_{\lambda,\F})_{\Ifrak}$
are projectives in  $\tilde{\mathcal{O}}(\A_{\lambda,\F})_{\Ifrak}$ and every simple is
covered by one of these projectives.
It follows that the indecomposable projectives in $\tilde{\OCat}(\A_{\lambda,\F})_{\Ifrak}$ are filtered
by $\Delta^{\Ifrak}_{\nu,\lambda,\F}(\mathsf{x},\kappa)$'s in a correct order. So for this choice of the interval $\Ifrak$,
the category $\tilde{\OCat}(\A_{\lambda,\F})_{\Ifrak}$ is highest weight. For a general interval
$\Ifrak'$, the category  $\tilde{\OCat}(\A_{\lambda,\F})_{\Ifrak'}$ is a highest weight subquotient
in  $\tilde{\OCat}(\A_{\lambda,\F})_{\Ifrak}$, where $\Ifrak=[z_1,z_2]$ for suitable $z_1,z_2$,
and hence a highest weight category itself.

\subsubsection{Case of higher dimensional torus}\label{SSS_higher_dim_mod_hw}
Let us explain what modifications are needed to deal with the case when $\dim T>1$.
We can still define a partial order $\leqslant_\nu$ on $X^T\times \mathfrak{X}(T)$
(instead of characters themselves we need now to compare their pairings with $\nu$).
A technical problem here is that most of intervals for this order are infinite.
But we can use a more general definition of a highest weight category
in Remark \ref{Rem:gen_hw} and we still get an analog of Theorem
\ref{Thm:phw_structure}.

\section{Standardly stratified structures}\label{S_SSS}
In this section we prove  the main result of the present paper, Theorem
\ref{Thm:mod_O_st_stratif}. Namely,
let $\lambda$ lie inside the $p$-alcove $^pA$ (corresponding to a real alcove $A$).
For each pair $(\Theta,\bar{\lambda})$, where $\Theta$ is a face  of $A$
and $\bar{\lambda}$ is a parameter compatible with $(A,\Theta)$, see Section
\ref{SS_compat_elements},
we will define a standardly stratified structure (in a suitable sense to be
explained below) on $\tilde{\OCat}(\A_{\lambda,\F})$. Again, for convenience
we assume that $\dim T=1$, the general case can be treated as in Section
\ref{SSS_higher_dim_mod_hw}.

We will get two results about the standardly stratified structure. First, we
will show that the associated graded of $\tilde{\OCat}(\A_{\lambda,\F})$
is essentially a reduction to characteristic $p$ of $\OCat_\nu(\A_{\bar{\lambda},\Q})$.
We will further show that, roughly speaking, standard and proper
standard objects in $\tilde{\OCat}(\A_{\lambda,\F})$ are reductions to characteristic $p$
of projective and simple objects in  $\OCat_\nu(\A_{\bar{\lambda},\Q})$.

\subsection{Main result} In this section, after some preparation we state the main theorem
regarding standardly stratified structures on $\tilde{\mathcal{O}}(\A_{\lambda,\F})$
determined by faces of the real alcove $A$.

\subsubsection{Standardly stratified structures on PHW categories}
We use the ramification of the definition of a standardly stratified structure that appeared in
\cite{LW} and recalled in Section \ref{SSS_SSS}. Let $\mathcal{C}$ be a PHW category with the shift functor $\mathcal{S}$.

\begin{defi}\label{defi:stand_stratif_periodic}
The additional
data of a standardly stratified category is a $\Z$-invariant pre-order $\preceq$ on $\operatorname{Irr}(\mathcal{C})$
that is  compatible with $\leqslant$, meaning that
$$L\prec L'\Rightarrow L<L'\Rightarrow L\preceq L'.$$
We also require that $L\prec \mathcal{S}L$ for all $L$. Since $|\operatorname{Irr}(\mathcal{C})/\Z|<\infty$,
 we see that the equivalence classes for $\prec$ are finite.

The axiom of a standardly stratified structure in this case is that, for each finite interval
$\Ifrak\subset \operatorname{Irr}(\mathcal{C})$ that is the union of some equivalence classes
for $\prec$, the subquotient category $\mathcal{C}_{\Ifrak}$ is standardly stratified with
respect to $\prec$ (in the sense explained in Section \ref{SSS_SSS}).
\end{defi}

\subsubsection{Pre-order determined by $\Theta$}
We will now define a pre-order on $\operatorname{Irr}(\tilde{\OCat}(\A_{\lambda,\F}))$.
Namely, recall that for $\lambda'$ in $\,^p\!A$ we have an equivalence
$\tilde{\OCat}(\A_{\lambda,\F})\xrightarrow{\sim} \tilde{\OCat}(\A_{\lambda',\F})$
of highest weight categories. We will take $\lambda':=\,^p\!\bar{\lambda}$,
where the right hand side is recovered from $\bar{\lambda}$ as explained
in Section \ref{SS_compat_elements}.
We will define a pre-order on $\operatorname{Irr}(\tilde{\OCat}(\A_{\lambda',\F}))^0$
(then we can transfer this pre-order to $\operatorname{Irr}(\tilde{\OCat}(\A_{\lambda',\F}))^\beta$
using the natural bijection between these sets, the simples from different equivariant blocks,
by definition, are not comparable).

Now note that, for a fixed point $\mathsf{x}\in X^T$ (and $\bar{\lambda}$), the map $p\mapsto c_{\nu,\,^p\!\bar{\lambda}}(\mathsf{x})$
is an affine map in $p$. It follows that the possible values of $\kappa$ such that $(\mathsf{x},\kappa)
\in \operatorname{Irr}(\tilde{\OCat}(\A_{\lambda',\F}))^0$ are affine functions in $p$, they are of the form
$\kappa(p)+mp$, where $\kappa$ is one of these functions, and $m$ is an arbitrary integer. Define a pre-order
$\preceq_{\nu,\bar{\lambda}}$ on $\operatorname{Irr}(\tilde{\OCat}(\A_{\lambda',\F}))^0$ by $(\mathsf{x},\kappa)
\preceq_{\nu,\bar{\lambda}} (\mathsf{x},\kappa')$ if the coefficient of $p$ in the affine function
$\kappa'-\kappa$ is nonnegative. We will discuss examples below.

Let us describe the equivalence classes of the resulting pre-order on $\operatorname{Irr}(\tilde{\OCat}(\A_{\lambda',\F})^0)$
(the zero equivariant block). Recall that the labels of simples in
that equivariant block are $(\mathsf{x},\kappa)$ such that $\kappa=\mathsf{c}_{\nu,\lambda'}(\mathsf{x})$
in $\F_p$. We note that $\sim_{\nu,\bar{\lambda}}$ descends to an equivalence
relation between the fixed points that will also be denoted by $\sim_{\nu,\bar{\lambda}}$. So $\mathsf{x}\sim_{\nu,\bar{\lambda}}\mathsf{x}'$ if $\mathsf{c}_{\nu,\bar{\lambda}}(\mathsf{x})-\mathsf{c}_{\nu,\bar{\lambda}}(\mathsf{x}')\in \Z$.
By the definition of $\sim_{\nu,\bar{\lambda}}$, the points $\mathsf{x}$ and $\mathsf{x}'$ are equivalent if and only if the corresponding
simples lie in the same $h$-block of $\OCat_{\nu}(\A_\lambda)$.

Then we have the following easy elementary lemma
that follows from the construction of $\lambda'$ (this was a starting observation for the present paper).

\begin{Lem}\label{Lem:equiv_classes_char_p}
The following two conditions are equivalent:
\begin{itemize}
\item $(\mathsf{x},\kappa)\sim (\mathsf{x}',\kappa')$
\item $\mathsf{x}\sim_{\nu,\bar{\lambda}}\mathsf{x}'$
and $\mathsf{c}_{\nu,\bar{\lambda}}(\mathsf{x})-\kappa=\mathsf{c}_{\nu,\bar{\lambda}}(\mathsf{x}')-\kappa'$.
\end{itemize}
In particular, the order on the equivalence class for $\preceq_{\nu,\bar{\lambda}}$ coincides with $\leqslant_{\nu,\bar{\lambda}}$.
\end{Lem}

So any equivalence class for $\preceq$ gives rise to a rational number $\beta$ (defined up to adding an integer)
as follows: if $(\mathsf{x},\kappa)$ is in the equivalence class, then $\beta-c_{\nu,\bar{\lambda}}(\mathsf{x})\in \Z$.

\begin{Rem}\label{Rem:dependence}
A priori, the pre-order $\preceq_{\nu,\bar{\lambda}}$ depends on the choice of $\bar{\lambda}$
and not only on $\Theta$. However, in the most interesting case when $\Theta$ is a point, it is easy to
see, compare with Lemma \ref{Lem:equiv_preserv} below, that $\preceq_{\nu,\bar{\lambda}}$ is independent
of the choice of $\bar{\lambda}$.
\end{Rem}

\subsubsection{Reductions to characteristic $p$}
Let $M\in \mathcal{O}_{\nu}(\A_{\bar{\lambda},\Q})$ and let $M_\Ring\subset M$ be
a graded $\Ring$-lattice. Pick $z_2\in \Z$ such that $M_z=0$ if $z>z_2$.
Then $M_{\F}:=\F\otimes_\Ring M_\Ring$ lies in $  \A_{\lambda',\F}\operatorname{-mod}^T_{\leqslant z_2}$.

Now choose an integer $z_1<z_2$ and set $\Ifrak=[z_1,z_2]$. We can view $\Ifrak$ as an interval in $\operatorname{Irr}
(\tilde{\OCat}(\A_{\lambda',\F})^0)$.

Let $\pi_{\Ifrak}$ denote the quotient functor $\A_{\lambda',\F}\operatorname{-mod}^T_{\leqslant z_2}
\twoheadrightarrow \tilde{\OCat}(\A_{\lambda,\F})_{\Ifrak}$ (that mods out
$\A_{\lambda',\F}\operatorname{-mod}^T_{< z_1}$, see (\ref{eq:quotient_equiv})).
Set $M_{\F,\Ifrak}:=\pi_{\Ifrak}(M_\F)$.

Recall that we have distinguished lattices $L_{\nu,\bar{\lambda},\Ring}(\mathsf{x},\kappa)
\subset L_{\nu,\bar{\lambda},\Q}(\mathsf{x},\kappa)$ and
$P_{\nu,\bar{\lambda},\Ring}(\mathsf{x},\kappa)
\subset P_{\nu,\bar{\lambda},\Q}(\mathsf{x},\kappa)$, see Section
\ref{SSS_simple_proj_reductions}.  So for $\kappa\in [z_1,z_2]$ we get
objects  $L_{\nu,\bar{\lambda},\Q}(\mathsf{x},\kappa)_\Ifrak:=
\pi_{\Ifrak}(L_{\nu,\bar{\lambda},\F}(\mathsf{x},\kappa))$. Also,
similarly to Lemma \ref{Lem:simple_proj_inj_Ring}, we have a unique subobject
$M_{\Ring}\subset P_{\nu,\bar{\lambda},\Ring}(\mathsf{x},\kappa)$
that is filtered with $\Delta_{\nu,\bar{\lambda},\Ring}(\mathsf{x}',\kappa')$,
where $\kappa'>\kappa$, while the quotient $P_{\nu,\bar{\lambda},\Ring}(\mathsf{x},\kappa)/
M_{\Ring}$ is filtered with $\Delta_{\nu,\bar{\lambda},\Ring}(\mathsf{x}',\kappa')$
with $\kappa'\leqslant \kappa$. We set $P_{\nu,\bar{\lambda},\Q}(\mathsf{x},\kappa)_\Ifrak:=
\pi_{\Ifrak}(P_{\nu,\bar{\lambda},\F}(\mathsf{x},\kappa)/
M_{\F})$.


Now for an $h$-block $\OCat_\nu(\A_{\bar{\lambda},\Q})^\beta$ we define its reduction
$\OCat_\nu(\A_{\bar{\lambda},\F})^\beta$. Namely, let $P$ be the direct sum of all indecomposable
projectives in $\OCat_\nu(\A_{\bar{\lambda},\Q})^\beta$ and let $P_\Ring$ be the direct sum
of their distinguished $\Ring$-forms. By definition,
$\OCat_\nu(\A_{\bar{\lambda},\F})^\beta$ is the category of right modules over
$\F\otimes_\Ring \End_{\A_{\lambda,\Ring}}(P_{\Ring})$.

The following is the main result of this section, and of the entire paper.

\begin{Thm}\label{Thm:mod_O_st_stratif}
Pick an interval $\Ifrak\subset \operatorname{Irr}(\tilde{\OCat}(\A_{\lambda',\F}))$ of the form
$\{(\mathsf{x},\kappa)| z-pm\leqslant \kappa<z\}$, where $z$ is the minimal value of
$\kappa$ in some equivalence class for the pre-order $\preceq$. Then the following claims are true:
\begin{enumerate}
\item The category $\tilde{\OCat}(\A_{\lambda',\F})_\Ifrak$ is standardly stratified
with respect to the pre-order $\preceq$. The objects  $P_{\nu,\bar{\lambda},\F}(\mathsf{x},\kappa)_\Ifrak$ are standard
and the objects $L_{\nu,\bar{\lambda},\F}(\mathsf{x},\kappa)_\Ifrak$ are proper standard
for $(\mathsf{x},\kappa)\in \mathfrak{I}$.
\item The component of the associated graded category $\gr \tilde{\OCat}(\A_{\lambda',\F})^0_\Ifrak$
corresponding to any equivalence class with respect to $\preceq$ is $\OCat_\nu(\A_{\bar{\lambda},\F})^\beta$
(as a highest weight category), where $\beta$ is the rational number corresponding to the equivalence
class, see the discussion after Lemma \ref{Lem:equiv_classes_char_p}. In particular,
$\gr \tilde{\OCat}(\A_{\lambda',\F})_\Ifrak\cong \OCat_\nu(\A_{\bar{\lambda},\F})^{\oplus m}$.
\end{enumerate}
\end{Thm}

\subsubsection{Example of $T^*(G/B)$}
The most interesting case here is when $\Theta$ is a point.

Consider $X=T^*(G/B)$ and take the  fundamental  $p$-alcove:
$\langle \alpha_i^\vee, \lambda\rangle\geqslant 1, \langle\alpha_0^\vee, \lambda\rangle\geqslant 1-p$,
where we write $\alpha_0^\vee$ for the minimal coroot and $i$ runs over $\{1,2,\ldots,r\}$.
Pick $i\in \{0,\ldots,r\}$ and consider the element $\bar{\lambda}_i$ defined by
$\langle \alpha_j^\vee, \bar{\lambda}_i\rangle=1$ for $j\neq i$ (in particular,  $\bar{\lambda}_0=\rho$).  The corresponding point $\lambda'_i=\,^p\!\bar{\lambda}_i$ is given
\begin{itemize}
\item
by $\lambda'_0=\rho$ if $i=0$,
\item
and by $\langle\alpha_j^\vee,\lambda'_i\rangle=1$ for $j\neq 0,i,
\langle\alpha_0^\vee,\lambda'\rangle=1-p$.
\end{itemize}
We have $\mathsf{c}_{\nu,\,^p\!\bar{\lambda}}(w)=
\langle w\,^p\!\bar{\lambda}, \rho^\vee\rangle$ is an affine function in $p$.

Two elements  $w_1,w_2\in X^T\cong W$ lie in the same $h$-block if and only if
$\langle w_1\bar{\lambda}-w_2\bar{\lambda}, \rho^\vee\rangle\in \Z$. So each of these equivalence
classes splits into the union of right cosets for the integral Weyl group $W_{[\lambda]}$
of $\lambda$ (generated by the reflections $s_\alpha$, where $\alpha$ is such that
$\langle \alpha^\vee,\lambda\rangle\in \Z$). The categories $\OCat_\nu(\A_{\bar{\lambda}})$
are the direct sums of the regular blocks of the BGG categories $\mathcal{O}$ for $W_{[\lambda]}$.

\subsubsection{Example of $\operatorname{Hilb}_n(\mathbb{A}^2)$}
Now consider the case of $X=\operatorname{Hilb}_n(\mathbb{A}^2)$ so that the algebra
$\A_{\lambda}$ is a spherical rational Cherednik algebra for $S_n$. As usual we set $c=\lambda-1/2$.

Recall that the $p$-alcoves (for the parameter $c$) are the intervals of the form $\displaystyle [\frac{(p+1)a}{b}+s,
\frac{(p+1)a'}{b'}-s]$, where $a/b<a'/b'$ are two rational numbers with denominators
between $2$ and $n$ such that there are no rational numbers with such denominators between
$a/b,a'/b'$.

Pick $\bar{\lambda}=\frac{a}{b}+s$. It is again clear that $\mathsf{c}_{\nu,\,^p\!\bar{\lambda}}(\mathsf{x})$
is an affine function in $p$, compare to Lemma \ref{Lem:RCA_highest_wt}.
By that lemma, we have $\mu\sim_{\nu,\bar{\lambda}}\mu'$
if and only if $\operatorname{cont}(\mu)-\operatorname{cont}(\mu')$ is divisible by $b$.
The order within the equivalence class is as follows: $\mu< \mu'$ if $\operatorname{cont}(\mu)>\operatorname{cont}(\mu')$.

We note that, thanks to results of \cite{Rouquier,VV_proof}, the category $\OCat_{\nu}(\A_{\lambda})$
is equivalent to the category of modules over the $q$-Schur algebra $\mathcal{S}_{q}(n)$ for the quantum
$\mathfrak{gl}_n$ (a labelling preserving equivalence of highest weight categories).
It follows that $\OCat_{\nu}(\A_{\lambda,\F})\cong \mathcal{S}_{q,\F}(n)\operatorname{-mod}$.

\subsection{Proof of Theorem \ref{Thm:mod_O_st_stratif}}
\subsubsection{Three technical lemmas}
In the proof we will use the following three technical lemmas.
The first one is a straightforward consequence of Corollary
\ref{Lem:Ext_freeness}.

\begin{Lem}\label{Lem:homom_reduction}
Let $M_{\Ring},N_{\Ring}$
be modules of the form $L_{\nu,\lambda,\Ring}(\mathsf{x},\kappa), P_{\nu,\lambda,\Ring}(\mathsf{x}',\kappa')$,
where $\mathsf{x},\mathsf{x}'$ are in the same $h$-block and $|\kappa-\kappa'|\leqslant m$,
where $m$ has the same meaning as in Corollary \ref{Lem:Ext_freeness}.
Then for all $i$, the $\Ring$-modules $\Ext^i_{\A_{\lambda,\Ring},T}(M_{\Ring},N_{\Ring})$
are free and finitely generated over $\Ring$. Then we automatically have
\begin{align*}
&\C\otimes_{\Ring}\Ext^i_{\A_{\lambda,\Ring},T}(M_{\Ring},N_{\Ring})
\xrightarrow{\sim} \Ext^i_{\A_{\lambda,\Ring},T}(M_{\C},N_{\C}), \\
&\F\otimes_{\Ring}\Ext^i_{\A_{\bar{\lambda},\Ring},T}(M_\Ring,N_\Ring)\xrightarrow{\sim}
\Ext^i_{\A_{\lambda',\F},T}(M_\F,N_\F).
\end{align*}
\end{Lem}

\begin{Lem}\label{Lem:hom_vanishing}
Let $\beta$ be an equivalence class for $\sim_\nu$ and let $M_{\Ring}$ be an object in
$\A_{\lambda,\Ring}\operatorname{-mod}^T$ filtered by
$L_{\nu,\lambda,\Ring}(\mathsf{x},\kappa)$, where $(\mathsf{x},\kappa)\in \beta$.
Then for any finite interval $\Ifrak$ containing $\beta$
and any $N_\F\in \tilde{\OCat}(\A_{\lambda',\F})_{\Ifrak,\prec \beta}$ we have
$$\Ext^i_{\tilde{\OCat}(\A_{\lambda',\F})_{\Ifrak}}(\M_{\F,\Ifrak},N)=0, \quad \forall i\geqslant 0.$$
\end{Lem}
\begin{proof}
Of course, it is enough to assume that $M_{\Ring}=L_{\nu,\lambda,\Ring}(\mathsf{x},\kappa)$.
By Lemma \ref{Lem:simple_proj_inj_Ring}, $M_{\Ring}$ admits a resolution by objects filtered
by $\Delta_{\nu,\lambda,\Ring}(\mathsf{x}',\kappa')$ for $(\mathsf{x}',\kappa')\in \beta$.
So it is enough to prove that
$$\Ext^i_{\tilde{\OCat}(\A_{\lambda',\F})_{\Ifrak}}(\Delta_{\nu,\lambda,\F}(\mathsf{x}',\kappa')_\Ifrak,N)=0,
\quad \forall i\geqslant 0.$$
These equalities hold because $\Delta_{\nu,\lambda,\F}(\mathsf{x}',\kappa')_\Ifrak$ is the standard
corresponding to $(\mathsf{x}',\kappa')$ and $N$ is filtered by simples whose
labels are less than the labels in  $\beta$.
\end{proof}

\begin{Cor}\label{Cor:hom_vanishing}
Let $M_\Ring$ be as in Lemma \ref{Lem:hom_vanishing} and $M_\F:=\F\otimes_{\Ring}M_{\Ring}$.
Then we have $M_\F=L\pi_\beta^!\circ \pi_\beta(M_\F)$,
where we write $\pi_\beta$ for the quotient functor $\A_{\lambda,\F}\operatorname{-mod}^{T}_{\preceq \beta}
\twoheadrightarrow \tilde{\OCat}(\A_{\lambda,\F})_\beta$ and $L\pi_\beta^!$ for the
left adjoint functor.
\end{Cor}

\begin{Lem}\label{Lem:proj_simpl_reduction}
Pick a label $(\mathsf{x},\kappa)\in \operatorname{Irr}(\tilde{\OCat}(\A_{\lambda,\F})^0)$
and let $\beta$ be the $\sim_\nu$ equivalence class of  $(\mathsf{x},\kappa)$.
Then $P_{\nu,\bar{\lambda},\F}(\mathsf{x},\kappa)_\beta$ is the projective in $\tilde{\OCat}(\A_{\lambda',\F})_\beta$
labelled by $(\mathsf{x},\kappa)$, while $L_{\nu,\bar{\lambda},\F}(\mathsf{x},\kappa)_\beta$ is the simple
labelled by $(\mathsf{x},\kappa)$.
\end{Lem}
\begin{proof}
Pick $(\mathsf{x}',\kappa')\in \beta$. By Corollary \ref{Lem:Ext_freeness}, we have
$$\Hom_{\A_{\lambda,\Ring},T}(\Delta_{\nu,\bar{\lambda},\Ring}(\mathsf{x}',\kappa'), L_{\nu,\bar{\lambda},\Ring}(\mathsf{x},\kappa))=
\Ring^{\oplus \delta_{\mathsf{x},\mathsf{x}'}\delta_{\kappa,\kappa'}}.$$
Applying Lemma \ref{Lem:homom_reduction}, we see that
$$\dim\Hom_{\A_{\lambda',\F},T}(\Delta_{\nu,\bar{\lambda},\F}(\mathsf{x}',\kappa'), L_{\nu,\bar{\lambda},\F}(\mathsf{x},\kappa)))=
\delta_{\mathsf{x},\mathsf{x}'}\delta_{\kappa,\kappa'}.$$

By Corollary \ref{Cor:hom_vanishing}, $\Delta_{\nu,\bar{\lambda},\F}(\mathsf{x}',\kappa')=
\pi_\beta^!(\Delta_{\nu,\bar{\lambda},\F}(\mathsf{x}',\kappa')_\beta)$.
From here we deduce that
$$\dim\Hom_{\tilde{\OCat}(\A_{\lambda',\F})_\beta}
(\Delta_{\nu,\bar{\lambda},\F}(\mathsf{x}',\kappa')_\beta, L_{\nu,\bar{\lambda},\F}(\mathsf{x},\kappa)_\beta))=
\delta_{\mathsf{x},\mathsf{x}'}\delta_{\kappa,\kappa'}.$$

Recall that the objects $\Delta_{\nu,\bar{\lambda},\F}(\mathsf{x}',\kappa')_\beta$ are standard. By Lemma
\ref{Lem:cost_inclusion}, $L_{\nu,\bar{\lambda},\F}(\mathsf{x},\kappa)_\beta\hookrightarrow \nabla^\F_{\nu,\lambda'}(\mathsf{x},\kappa)_\beta$.
It remains to show that the head of $L_{\nu,\bar{\lambda},\F}(\mathsf{x},\kappa)_\beta$ coincides
with $\pi_\beta(\underline{L}_{\nu,\lambda',\F}(\mathsf{x},\kappa))$.

By Corollary \ref{Cor:hom_vanishing}, $\pi_\beta^!(L_{\nu,\bar{\lambda},\F}(\mathsf{x},\kappa)_\mathfrak{J})=
L_{\nu,\bar{\lambda},\F}(\mathsf{x},\kappa)_\beta$.
So using Lemma \ref{Lem:homom_reduction} again and arguing as in the first paragraph of the proof
of the present lemma, we conclude
$$\dim\Hom_{\tilde{\OCat}(\A_{\lambda',\F})_\beta}
(L_{\nu,\bar{\lambda},\F}(\mathsf{x},\kappa)_\beta, L_{\nu,\bar{\lambda},\F}(\mathsf{x}',\kappa')_\beta))=
\delta_{\mathsf{x},\mathsf{x}'}\delta_{\kappa,\kappa'}.$$
Since  $L_{\nu,\bar{\lambda},\F}(\mathsf{x}',\kappa')_\beta\hookrightarrow
\nabla^\F_{\nu,\lambda'}(\mathsf{x}',\kappa')_\beta$ we conclude that
the head of $L_{\nu,\bar{\lambda},\F}(\mathsf{x},\kappa)_\beta$ is
$\pi_\beta(\underline{L}_{\nu,\lambda',\F}(\mathsf{x},\kappa))$. This finishes the proof that
$L_{\nu,\bar{\lambda},\F}(\mathsf{x},\kappa)_\beta=\pi_\beta(\underline{L}_{\nu,\lambda',\F}(\mathsf{x},\kappa))$.

Let us show that $P_{\nu,\bar{\lambda},\F}(\mathsf{x},\kappa)_\beta$ is the projective cover of
$\pi_\beta(\underline{L}_{\nu,\lambda',\F}(\mathsf{x},\kappa))$. From Lemma \ref{Lem:homom_reduction}
it follows that
$$\dim\Ext^i_{\A_{\lambda',\F},T}(P_{\nu,\bar{\lambda},\F}(\mathsf{x},\kappa),
L_{\nu,\bar{\lambda},\F}(\mathsf{x}',\kappa'))=
\delta_{i,0}\delta_{\mathsf{x},\mathsf{x}'}\delta_{\kappa,\kappa'}.$$
By Corollary \ref{Cor:hom_vanishing}, $L\pi_{\beta}^!P_{\nu,\bar{\lambda},\F}(\mathsf{x},\kappa)_\beta=P_{\nu,\bar{\lambda},\F}(\mathsf{x},\kappa)$.
We conclude that, for $i=0,1$, we have
$$\dim \Ext^i_{\tilde{\OCat}(\A_{\lambda',\F})_\beta}(P_{\nu,\bar{\lambda},\F}(\mathsf{x},\kappa)_\beta,
L_{\nu,\bar{\lambda},\F}(\mathsf{x}',\kappa')_\beta)=
\delta_{i,0}\delta_{\mathsf{x},\mathsf{x}'}\delta_{\kappa,\kappa'}.$$
We already know that  $L_{\nu,\bar{\lambda},\F}(\mathsf{x}',\kappa')_\beta$ is the simple object corresponding
to $(\mathsf{x}',\kappa')$. Hence $P_{\nu,\bar{\lambda},\F}(\mathsf{x},\kappa)_\beta$ is the projective object
corresponding to $(\mathsf{x},\kappa)$.
\end{proof}

\subsubsection{Completion of the proof}
\begin{proof}[Proof of Theorem \ref{Thm:mod_O_st_stratif}]
Let $\beta$ be as in Lemma \ref{Lem:proj_simpl_reduction} and let
$\Ifrak$ be an interval containing $\beta$. By Lemma \ref{Lem:proj_simpl_reduction},
for $(\mathsf{x},\kappa)\in \beta$, the object $L_{\nu,\bar{\lambda}, \F}(\mathsf{x},\kappa)_{\beta}$ (that coincides
with $\pi_\beta(L_{\nu,\bar{\lambda}, \F}(\mathsf{x},\kappa)_\Ifrak)$ by the definitions)
is simple. By Corollary \ref{Lem:hom_vanishing}, $L\pi_{\beta}^!(L_{\nu,\bar{\lambda}, \F}(\mathsf{x},\kappa)_{\beta})=L_{\nu,\bar{\lambda}, \F}(\mathsf{x},\kappa)_\Ifrak$.
It follows, in particular, that $\pi_{\beta}^!$ is exact.

Thanks to Lemma \ref{Lem:comp_SS_structure}, to finish the proof of the claim that $\preceq$ defines a standardly
stratified structure, it remains to show that $\pi_{\beta}^*:\tilde{\OCat}(\A_{\lambda,\F})_\beta
\rightarrow \tilde{\OCat}(\A_{\lambda,\F})_{\Ifrak,\preceq\beta}$ is exact. For this we will use the duality
functor $\bullet^*$ from Section \ref{SSS_duality}. Recall that
$\underline{L}_{\nu,\lambda,\F}(\mathsf{x},\kappa)=\underline{L}_{-\nu,-\lambda,\F}(\mathsf{x},-\kappa)^*$.
It follows that $\bullet^*$ induces an equivalence $\bullet^*:\tilde{\OCat}(\A_{\lambda,\F})_{\Ifrak}\xrightarrow{\sim}\tilde{\OCat}(\A_{-\lambda,\F})^{opp}_{-\Ifrak}$
for any interval $\Ifrak$. Clearly, the functors $\bullet^{\dagger}$ intertwine
$(\pi_{\beta}^{opp})^!$ with $\pi_\beta^*$. Since we know that $(\pi_\beta^{opp})^!$ is exact,
we see that $\pi_\beta^*$ is exact.

It follows that the pre-order $\preceq$ indeed defines a standardly
stratified structure. Corollary \ref{Cor:hom_vanishing} and
Lemma \ref{Lem:proj_simpl_reduction} imply that $L_{\nu,\lambda,\F}(\mathsf{x},\kappa)_{\Ifrak},
P_{\nu,\lambda,\F}(\mathsf{x},\kappa)_{\Ifrak}$ are the proper standard and standard objects.
This finishes the proof of (1).

Part (2) easily follows from Lemma \ref{Lem:homom_reduction} (applied to $M_\Ring=N_\Ring:=P_\Ring$)
and the projective part of Lemma \ref{Lem:proj_simpl_reduction}.
\end{proof}

\subsubsection{The case of $\dim T>1$}
The case when $\dim T>1$ can be handled similarly to the situation of highest weight structures.
In particular, we can use a more general notion of a standardly stratified category,
see Remark \ref{Rem:gen_SS}.

%

\subsection{Wall-crossing for categories $\tilde{\mathcal{O}}$}
Now let $A_1,A_2$ be two real alcoves that are opposite with respect a common face $\Theta$.
Let $\bar{\lambda}_1,\bar{\lambda}_2=\bar{\lambda}_1+\chi$ be elements in $\Lambda$ compatible with $(A_1,\Theta),
(A_2,\Theta)$ with $\chi\in\mathcal{P}_1$. Set $\lambda_i':=\,^p\bar{\lambda}_i, i=1,2$.
Then we have the wall-crossing functor
$\A_{\lambda'_1,\chi,\F}\otimes^L_{\A_{\lambda'_1,\F}}\bullet: D^b(\A_{\lambda_1',\F}\operatorname{-mod}^T)
\xrightarrow{\sim} D^b(\A_{\lambda_2',\F}\operatorname{-mod}^T)$.

Our goal in this section is to show, that, in a suitable sense, the wall-crossing functor
is a partial Ringel duality functor between the categories $\tilde{\OCat}(\A_{\lambda'_1,\F})$
and $\tilde{\OCat}(\A_{\lambda'_2,\F})$.

\subsubsection{Preparation}
Let us start with the following elementary lemma.

\begin{Lem}\label{Lem:h_block_preserv}
Let $(\mathsf{x},\kappa)$ and $(\mathsf{x}',\kappa')$ lie in the same equivariant block
for $\lambda\in \paramq_{\F}$. Then, for any $\chi\in \operatorname{Pic}(X)$, $(\mathsf{x},\kappa+\mathsf{wt}_\chi(\mathsf{x})),
(\mathsf{x}',\kappa'+\mathsf{wt}_\chi(\mathsf{x}'))$ lie in the same equivariant block.
\end{Lem}
\begin{proof}
Recall, \cite[Section 6.1]{CWR}, $\mathsf{c}_{\lambda}(\mathsf{x})-\mathsf{c}_{\lambda}(\mathsf{x}')$ is
an affine function in $\lambda$ and $(\mathsf{c}_{\lambda+\chi}(\mathsf{x})-\mathsf{c}_{\lambda+\chi}(\mathsf{x}'))-
(\mathsf{c}_{\nu,\lambda}(\mathsf{x})-\mathsf{c}_{\nu,\lambda}(\mathsf{x}'))
=\mathsf{wt}_\chi(\mathsf{x})-\mathsf{wt}_\chi(\mathsf{x}')$. Our claim easily follows from here.
\end{proof}

So by choosing a suitable equivariant structure on $O(\chi)$, we can assume that
$\underline{L}_\nu(\mathsf{x},\kappa+\mathsf{wt}_\chi(\mathsf{x}))\in \tilde{\OCat}(\A_{\lambda+\chi,\F})^0$
if and only if $\underline{L}_\nu(\mathsf{x},\kappa)\in \tilde{\OCat}(\A_{\lambda,\F})^0$.

Arguing as in the proof of  Lemma \ref{Lem:h_block_preserv} and using the condition that $\chi$
is independent of $p$, we get the following.

\begin{Lem}\label{Lem:equiv_preserv}
We have $(\mathsf{x},\kappa)\sim_{\nu,\lambda'_1}(\mathsf{x}',\kappa')$ if and only if
$(\mathsf{x},\kappa+\mathsf{wt}_\chi(\mathsf{x}))\sim_{\nu,\lambda_2'}(\mathsf{x}',\kappa'+\mathsf{wt}_\chi(\mathsf{x}'))$.
\end{Lem}

So we have a bijection between equivalence classes of simples in $\tilde{\OCat}(\A_{\lambda'_1,\F})^0$
and in $\tilde{\OCat}(\A_{\lambda'_2,\F})^0$. This bijection is compatible with $\preceq$.
So to an interval (with respect to $\preceq$) $\Ifrak\subset \operatorname{Irr}(\tilde{\OCat}(\A_{\lambda'_1,\F})^0)$
we can assign an interval   $\Ifrak^\chi\subset \operatorname{Irr}(\tilde{\OCat}(\A_{\lambda'_2,\F})^0)$.

\subsubsection{Main result}
The following is the main result of the section.

\begin{Thm}\label{Thm:WC_part_Ringel}
The functor
$$\WC_{\lambda'_2\leftarrow \lambda'_1}:\A_{\bar{\lambda},\chi,\F}\otimes^L_{\A_{\bar{\lambda},\F}}\bullet:
D^b(\A_{\lambda'_1,\F}\operatorname{-mod}^T)\xrightarrow{\sim} D^b(\A_{\lambda'_2,\F}\operatorname{-mod}^T)$$
induces a partial Ringel duality functor  $D^b(\tilde{\OCat}(\A_{\lambda'_1,\F})_{\Ifrak})\xrightarrow{\sim} D^b(\tilde{\OCat}(\A_{\lambda'_2,\F})_{\Ifrak^\chi})$
(with respect to the standardly stratified structures defined
by $\bar{\lambda}$).
\end{Thm}
\begin{proof}
Recall that, in Step 3 of the proof of Theorem \ref{Thm:phw_structure},
for $z\in \Z$ we have defined the category $\A_{\lambda,\F}\operatorname{-mod}^T_{\leqslant z}$.
Inside we can consider the subcategory $\A_{\lambda,\F}\operatorname{-mod}^{T,0}_{\leqslant z}
\subset \A_{\lambda,\F}\operatorname{-mod}^{T}_{\leqslant z}$
defined similarly to $\tilde{\OCat}(\A_{\lambda,\F})^0\subset \tilde{\OCat}(\A_{\lambda,\F})$.
This subcategory is a direct summand. We have sufficiently many projectives in
$\A_{\lambda,\F}\operatorname{-mod}^T_{\leqslant z}$ and they have no higher self-extensions in
$\A_{\lambda,\F}\operatorname{-mod}^T$, this was established in the proof of
Theorem \ref{Thm:phw_structure}. So   the natural functor
$D^b(\A_{\lambda,\F}\operatorname{-mod}^T_{\leqslant z})\hookrightarrow
D^b(\A_{\lambda,\F}\operatorname{-mod}^T)$ is a full embedding.
So, for an interval $[z_1,z_2]$ corresponding to $\Ifrak$,
we have $D^b(\tilde{\OCat}(\A_{\lambda,\F})_\Ifrak)=
D^b(\A_{\lambda,\F}\operatorname{-mod}^T_{\leqslant z_2})/
D^b(\A_{\lambda,\F}\operatorname{-mod}^T_{< z_1})$.
It follows that $D^b(\tilde{\OCat}(\A_{\lambda,\F})^0_\Ifrak)=
D^b(\A_{\lambda,\F}\operatorname{-mod}^{T,0}_{\leqslant z_2})/
D^b(\A_{\lambda,\F}\operatorname{-mod}^{T,0}_{< z_1})$.
From here and the previous section we deduce that
$\WC_{\lambda'_2\leftarrow \lambda'_1}$ indeed
induces a functor $D^b(\tilde{\OCat}^0(\A_{\lambda'_1,\F})_{\Ifrak})\xrightarrow{\sim} D^b(\tilde{\OCat}^0(\A_{\lambda'_2,\F})_{\Ifrak^\chi})$.

It follows from Lemma \ref{Lem:WC_R} and the $\Ring$-flatness of
$\A_{\bar{\lambda},\chi,\Ring},\Delta_{\nu,\bar{\lambda},\Ring}(\mathsf{x},\kappa)$ (Corollary \ref{Cor:translation_flatness} and Lemma \ref{Lem:Verma_R_flatness})
that $$\WC_{\lambda'_2\leftarrow \lambda'_1}(\Delta_{\nu,\bar{\lambda},\F}(\mathsf{x},\kappa))=
\nabla_{\nu,\bar{\lambda}+\chi,\F}(\mathsf{x}, \kappa+\mathsf{wt}_\chi(\mathsf{x})).$$ From here it follows
that $\WC_{\lambda'_2\leftarrow \lambda'_1}$ is compatible with the filtrations on the categories
$D^b(\tilde{\OCat}(\A_{\lambda'_1,\F})_{\Ifrak}),$ $D^b(\tilde{\OCat}(\A_{\lambda'_2,\F})_{\Ifrak^\chi})$
coming from the standardly stratified structures.
In particular, the functor induces an equivalence of the associated graded categories.
Assume  $\mathfrak{J}$ is an equivalence class. By Theorem \ref{Thm:mod_O_st_stratif}, for $(\mathsf{x},\kappa)\in \mathfrak{J}$,
the costandard object in $\tilde{\OCat}(\A_{\lambda'_2,\F})_{\mathfrak{J}^\chi}$
is $\pi_{\mathfrak{J}}(\nabla_{\nu,\bar{\lambda}+\chi,\F}(\mathsf{x}, \kappa+\mathsf{wt}_\chi(\mathsf{x})))$.
It follows that the equivalence of the associated graded categories induced by $\WC_{\lambda'_2\leftarrow \lambda'_1}$
is a Ringel duality functor. So the functor
$D^b(\tilde{\OCat}(\A_{\lambda'_1,\F})_{\Ifrak})\xrightarrow{\sim} D^b(\tilde{\OCat}(\A_{\lambda'_2,\F})_{\Ifrak^\chi})$
induced by $\WC_{\lambda'_2\leftarrow \lambda'_1}$ is a partial Ringel duality.
\end{proof}

\section{Applications}
\subsection{Wall-crossing bijections}
Let $\Theta,\lambda',\chi$ be as above. We consider the category $\A_{\lambda',\F}\operatorname{-mod}_0$
of all finite dimensional $\A_{\lambda',\F}$-modules with generalized zero $p$-character.
We also consider the category $D^b_0(\A_{\lambda',\F}\operatorname{-mod})$ of all objects
in $D^b_0(\A_{\lambda',\F}\operatorname{-mod})$ with homology in $\A_{\lambda',\F}\operatorname{-mod}_0$.
Note that $\mathfrak{WC}_{\lambda'+\chi\leftarrow \lambda'}=
\A_{\lambda',\chi,\F}\otimes^L_{\A_{\lambda',\F}}\bullet$ restricts to a perverse
equivalence
$$D^b_0(\A_{\lambda',\F}\operatorname{-mod})\xrightarrow{\sim}
D^b_0(\A_{\lambda'+\chi,\F}\operatorname{-mod})$$

Let us classify the irreducible objects in $\A_{\lambda',\F}\operatorname{-mod}_0$.

\begin{Lem}\label{Lem:irred_modular}
Fix a generic one-parameter subgroup $\nu$. Then we get a bijection between $\operatorname{Irr}(\A_{\lambda',\F}\operatorname{-mod}_0)$ and $X^T$.
\end{Lem}
\begin{proof}
Every irreducible from $\A_{\lambda',\F}\operatorname{-mod}_0$ has a $T_\F$-equivariant
structure, unique up to a twist with a character. It is still irreducible as a $T_{\F}$-equivariant
module and hence is  $\underline{L}(\mathsf{x},\kappa)$ for some character $\kappa$. This implies the claim of the lemma.
\end{proof}

As any perverse equivalence, $\mathfrak{WC}_{\lambda'+\chi\leftarrow \lambda'}$ gives rise to
a bijection $\operatorname{Irr}(\A_{\lambda',\F}\operatorname{-mod}_0)\xrightarrow{\sim}
\operatorname{Irr}(\A_{\lambda'+\chi,\F}\operatorname{-mod}_0)$, i.e., a self-bijection of
$X^T$. We want to compare this bijection with a similarly defined bijection for categories  $\mathcal{O}$ in characteristic $0$.

\begin{Prop}\label{Prop:WC_bijections}
For $p\gg 0$, the self-bijection of $X^T$ induced  by $\WC_{\lambda'+\chi\leftarrow \lambda'}$
coincides with the bijection $\operatorname{Irr}(\mathcal{O}_\nu(\A_{\bar{\lambda}}))
\xrightarrow{\sim} \operatorname{Irr}(\mathcal{O}_\nu(\A_{\bar{\lambda}+\chi}))$ coming
from the wall-crossing functor $\WC_{\bar{\lambda}+\chi\leftarrow \bar{\lambda}}$.
\end{Prop}
\begin{proof}
Let $\varphi$ denote the  bijection coming from $\mathfrak{WC}_{\bar{\lambda}+\chi\leftarrow \bar{\lambda}}$.
What we need to show is that the wall-crossing bijection $\tilde{\varphi}$
for the categories $\tilde{\OCat}(\A_{?,\F})$ has the form $(\mathsf{x},\kappa)\mapsto (\varphi(\mathsf{x}), \kappa')$.
Note that  $\WC_{\lambda'+\chi\leftarrow \lambda'}$ respects the filtrations on the categories
$\tilde{\OCat}(\A_{\lambda',\F}), \tilde{\OCat}(\A_{\lambda'+\chi,\F})$ coming from the face
$\Theta$. So the induced functor between the associated graded categories is perverse and the induced bijection is the same as $\tilde{\varphi}$. On the other hand, the associated graded categories are
$\OCat^T_{\nu}(\A_{\bar{\lambda},\F})$ and $\OCat^T_{\nu}(\A_{\bar{\lambda}+\chi,\F})$
and the functor induced by  $\WC_{\lambda'+\chi\leftarrow \lambda'}$ is the Ringel duality.
So forgetting the $T$-character component, we see that the bijection
$X^T=\operatorname{Irr}(\A_{\lambda',\F}\operatorname{-mod}_0)\xrightarrow{\sim}
\operatorname{Irr}(\A_{\lambda'+\chi,\F}\operatorname{-mod}_0)=X^T$ coincides
with the bijection between the sets of simples in $\OCat_{\nu}(\A_{\bar{\lambda},\F})$ and $\OCat_{\nu}(\A_{\bar{\lambda}+\chi,\F})$
induced by the Ringel duality. But $\WC_{\bar{\lambda}+\chi\leftarrow \bar{\lambda}}$ is
the Ringel duality between $\OCat_{\nu}(\A_{\bar{\lambda}})$ and $\OCat_{\nu}(\A_{\bar{\lambda}+\chi})$
and since $p$ is very large, the bijections induced by the Ringel duality over $\F$ and over $\C$
coincide. This completes the proof.
\end{proof}

Wall-crossing bijections were studied (and sometimes computed combinatorially) in a number of cases.
Paper \cite{cacti} studied the case of $X=T^*(G/B)$. There we have seen that the wall-crossing
bijections through the faces containing $0$ define an action of the so called cactus group
$\mathsf{Cact}_W$ on $W$. Using Proposition \ref{Prop:WC_bijections} one can show that this action
extends to an action of the affine cactus group. We note that combinatorial recipes to compute
the action are not known in general.

The case of rational Cherednik algebras of type $A$ (and of more general cyclotomic rational
Cherednik algebras) was considered in \cite{cher_supp}. It was shown that the wall-crossing
bijections for rational Cherednik algebras of type A are extended Mullineux involutions,
see \cite[Corollary 5.7]{cher_supp} for a precise statement. Therefore the wall-crossing
bijection induced by $\WC_{\lambda'+\chi\leftarrow \lambda'}$ is given by the same rule.

\subsection{Gradings}
In this section, following an idea of Bezrukavnikov, we produce  graded lifts of the category
$\tilde{\OCat}(\A_{\lambda',\F})$ that come from the contacting torus action on $X$.
Then we show that our grading lift induces a grading lift of $\OCat_\nu(\A_{\bar{\lambda}})$.
Finally, we compare Koszulity properties of the grading lifts on
$\tilde{\OCat}(\A_{\lambda',\F})$ and on $\OCat_\nu(\A_{\bar{\lambda}})$.

\subsubsection{Grading on $\tilde{\OCat}(\A_{\lambda',\F})$}
Let $Y^{(1),\wedge_0}_\F$ denote the spectrum of the completion $\F[Y^{(1)}]^{\wedge_0}$, an
$\F$-scheme, and let $X^{(1),\wedge_0}_\F$ be its preimage in $X^{(1)}_\F$.

We assume that the microlocal quantization $\A^\theta_{\lambda',\F}$ of $X_\F$
is obtained (by completing with respect to the filtration) from a Frobenius constant
quantization (in the sense of \cite{BK_quant_p}). By abusing the notation,
we denote the corresponding Azumaya algebra on $X^{(1)}_\F$ also by
$\A^\theta_{\lambda',\F}$. This holds in the examples we consider
(and in more general examples of Slodowy varieties, \cite{BMR}, and of Nakajima quiver
varieties, \cite{BFG}).

Consider the restriction $\A^{\theta,\wedge_0}_{\lambda',\F}$ of $\A^\theta_{\lambda',\F}$ to $X^{(1)\wedge_0}_\F$.
We assume that it splits, let $\mathcal{V}_\F^{\wedge_0}$ denote the splitting bundle.
Recall that $\mathcal{V}_\F^{\wedge_0}$ is defined  up to a twist with a line bundle. The splitting bundle again
exists in all examples we consider, see \cite{BMR} for the case of Slodowy varieties
and \cite{BL} for the case of Nakajima quiver varieties.

The splitting bundle $\mathcal{V}_\F^{\wedge_0}$ has no higher self-extensions.
It follows that it admits a $T_{\F}\times \F^\times$-equivariant structure,
where we write $\F^\times$ for the contracting torus. We can choose a $T_{\F}$-equivariant
structure on $\mathcal{V}_\F^{\wedge_0}$ so that the isomorphism
$\mathcal{E}nd(\mathcal{V}_\F^{\wedge_0})\cong \A^{\theta,\wedge_0}_{\lambda',\F}$ is
$T_{\F}$-equivariant.

Since the action of  $\F^\times$ is contracting we can uniquely extend $\mathcal{V}_{\F}^{\wedge_0}$
to a  $T_{\F}\times \F^\times$-equivariant vector bundle $\mathcal{V}_\F$. Set $\tilde{A}_\F:=\operatorname{End}(\mathcal{V}_\F)$.
We can consider the categories $\tilde{A}_\F\operatorname{-mod}^T_{0}$ of all $T_\F$-equivariant
finite dimensional $\tilde{A}_\F$-modules (automatically supported at $0\in Y^{(1)}_{\F}$)
and $\tilde{A}_\F\operatorname{-mod}^{T\times \F^\times}_0$ of all $T_\F\times \F^\times$-equivariant
finite dimensional $\tilde{A}_\F$-modules. The construction of $\tilde{A}_\F$ yields a
category equivalence $\tilde{\OCat}(\A_{\lambda',\F})\cong \tilde{A}_\F\operatorname{-mod}_0^{T}$.
So  $\tilde{A}_\F\operatorname{-mod}^{T\times \F^\times}_0$ is a graded lift of
$\tilde{\OCat}(\A_{\lambda',\F})$. Note that this graded lift is independent of the choice
of $\lambda'$ in its $p$-alcove.

\subsubsection{From graded lift  $\tilde{\OCat}(\A_{\lambda',\F})$
to graded lift of $\OCat_\nu(\A_{\bar{\lambda}})$}
Now we are going to produce a graded lift of $\OCat_\nu(\A_{\bar{\lambda}})$
from a graded lift of $\tilde{\OCat}(\A_{\lambda',\F})$. Note that a graded
lift of $\tilde{\OCat}(\A_{\lambda',\F})$ induces that of any quotient
$\tilde{\OCat}(\A_{\lambda',\F})_{\leqslant z_2}/\tilde{\OCat}(\A_{\lambda',\F})_{\leqslant z_1}$.
In particular, $\OCat_{\nu}(\A_{\bar{\lambda},\F})$ is the direct sum of such subquotients
so we get graded lifts of  $\OCat_{\nu}(\A_{\bar{\lambda},\F})$.

Now we are in the following situation. Let $\Ring$ be a finite localization of $\Z$
and let $B_\Ring$ be an $\Ring$-algebra that is a free finite rank $\Ring$-module
(in our case $B_{\Ring}=\operatorname{End}(P_{\Ring})^{opp}$, where $P_{\Ring}$
is an $\Ring$-form of a pro-generator of $\mathcal{O}_\nu(\A_\lambda)$).
We need to check that for $p\gg 0$ there is a natural bijection between graded lifts
of $B_{\overline{\F}_p}\operatorname{-mod}$ and of $B_{\C}\operatorname{-mod}$.

For this we need to note that, for a finite dimensional algebra $B$ over an algebraically closed
field $\mathbb{K}$, graded lifts of $B\operatorname{-mod}$ are parameterized by conjugacy classes of
one-parameter subgroups of $H:=\operatorname{Aut}(B)/B^\times$.

In our case $H_{\K}$ is the base change of an algebraic group scheme $H_\Ring$ from $\Ring$
to $\K$ assuming that $\operatorname{char}\K$ is 0 or is large enough. After taking a finite
extension of $\Ring$, we can find a subgroup subscheme $T_{\Ring}\subset H_{\Ring}$ that
becomes a maximal torus in $H_\K$ after a base change to $\K$. The conjugacy classes
of one-parameter subgroups in $H_{\K}$ are the orbits of the action of $N_{H_{\K}}(T_{\K})$
on the lattice of one-parameter subgroups in $T_{\K}$. This is clearly independent of $\K$.

This establishes a required bijection between the graded lifts.

\subsubsection{Koszulity}
Let us now show that if $\tilde{A}_\F\operatorname{-mod}^{T\times \F^\times}_0$
is Koszul for infinitely many $p$, then the category $\mathcal{O}_{\nu}(\A_\lambda)$
is Koszul as well.

Recall that being Koszul for the category $\tilde{A}_\F\operatorname{-mod}^{T\times \F^\times}_0$ means
that for every simple $L\in \tilde{A}_\F\operatorname{-mod}^{T}_0$ one can find a lift
$\tilde{L}\in \tilde{A}_\F\operatorname{-mod}^{T\times \F^\times}_0$ such that
$\operatorname{Ext}^i(\tilde{L},\tilde{L}')$ is concentrated in degree $i$ for all
simples $L,L'$ (here the Ext's are taken in $\tilde{A}_\F\operatorname{-mod}^{T}_0$).
Let $\tilde{\OCat}(\A_{\lambda,\F})^{\F^\times}$ denote the corresponding graded lift
of $\tilde{\OCat}(\A_{\lambda,\F})$.

\begin{Lem}\label{Lem:Koszul_subquotient}
For every interval $\Ifrak\subset \Z$ (with respect to $\leqslant_\nu$) the subquotient category
the induced graded lift $\tilde{\OCat}(\A_{\lambda,\F})_\Ifrak$ is Koszul.
\end{Lem}
\begin{proof}
Recall that the natural functor $D^b(\tilde{\OCat}(\A_{\lambda,\F})_{\leqslant z})
\hookrightarrow D^b(\tilde{\OCat}(\A_{\lambda,\F}))$ is a full embedding.
So the induced graded lift of $\tilde{\OCat}(\A_{\lambda,\F})_{\leqslant z}$
is Koszul. Note that $\tilde{\OCat}(\A_{\lambda,\F})_{\leqslant z}$
has enough projectives. Let $P(\tilde{L})$ denote the projective cover of
$\tilde{L}$ in $\tilde{\OCat}(\A_{\lambda,\F})_{\leqslant z}$. Since
$\tilde{\OCat}(\A_{\lambda,\F})_{\leqslant z}$ is Koszul, we can find
a projective resolution of any $\tilde{L}$ of the form $\ldots\rightarrow P_1\rightarrow
P_0\rightarrow 0$ such that $P_i$ is the direct sum of $P(\tilde{L}')\langle i\rangle$'s with some
multiplicities, where $\langle\bullet\rangle$ means a grading shift.
Then $\pi_{\Ifrak}(P_\bullet)$ is the projective resolution of $\pi_{\Ifrak}(\tilde{L})$.
Hence  $\tilde{\OCat}(\A_{\lambda,\F})_\Ifrak$ is Koszul.
\end{proof}

In particular, we see that the categories $\OCat_\nu(\A_{\lambda,\F})$ acquire  Koszul graded lifts
for infinitely many $p$. Let us deduce from here that  $\OCat_\nu(\A_{\lambda})$ does. This follows
from the next lemma.

\begin{Lem}\label{Lem:Koszul_preserv}
Let $\Ring$ be a ring that is a finite algebraic extension of $\Z$. Let
$B_{\Ring}$ be an $\Ring$-algebra that is a free finite rank $\Ring$-module.
Suppose that $B_{\Ring}$ has finite homological
dimension. Finally, suppose that
the fibers of $B_{\Ring}$ at infinitely many maximal ideals of $\Ring$ carry Koszul gradings.
Then $B_{\operatorname{Frac}(\Ring)}$ carries a Koszul grading.
\end{Lem}
\begin{proof}
By replacing $\Ring$ with its finite localization, we can achieve that there are objects
$L^i_{\Ring}, i=1,\ldots,k,$ such that the fibers of these objects at any point of $\Ring$
are the simples over the corresponding algebra. Moreover, localizing $\Ring$
further, we can assume that
$\operatorname{Ext}^\ell_{B_{\Ring}}(L^i_{\Ring}, L^j_{\Ring})$ are projective $\Ring$-modules
for all $i,j,\ell$ (here we use that $B_{\Ring}$ has finite homological dimension).

As we have seen above, after replacing $\Ring$ with its finite algebraic extension, we can assume that
there is a bijection between gradings on all the fibers of $B_{\Ring}$, moreover, we can assume that
every grading of every fiber comes from a grading on $B_\Ring$. So pick a grading that gives rise to
a Koszul grading in some fiber. We  claim that it gives a Koszul grading on some fiber. This follows
from the observation that all $\Ring$-modules    $\operatorname{Ext}^\ell_{B_{\Ring}}(L^i_{\Ring}, L^j_{\Ring})$
are graded and all graded components are projective $\Ring$-modules.
\end{proof}

\end{document}